\documentclass[11pt, twoside]{article}

\usepackage{graphicx}
\usepackage[caption=false]{subfig}

\usepackage[T1]{fontenc}
\usepackage[utf8]{inputenc}
\usepackage{amsmath}
\usepackage{amssymb,amsfonts}
\usepackage{amsthm}
\usepackage{mathtools}
\usepackage{bm}
\usepackage[mathscr]{eucal}
\usepackage{stmaryrd}
\usepackage{xcolor}
\usepackage{microtype}
\allowdisplaybreaks

\usepackage[margin=1in]{geometry}

\usepackage{multirow}
\usepackage{makecell}
\usepackage{booktabs}

\usepackage{algorithm}
\usepackage{algpseudocode}

\usepackage{cite}
\usepackage{hyperref}
\hypersetup{hidelinks}
\usepackage[nameinlink]{cleveref}

\numberwithin{equation}{section}

\theoremstyle{plain}
\newtheorem{theorem}{Theorem}[section]
\newtheorem{lemma}[theorem]{Lemma}

\newtheorem{corollary}[theorem]{Corollary}

\theoremstyle{definition}

\theoremstyle{remark}

\usepackage{fancyhdr}
\begin{document}

\title{Pressure-Robust $H(\mathrm{div})$-Conforming HDG Methods for the Steady Stokes Equations with an Application to Tangential Boundary Control}

\author{Gang Chen%
\thanks{School of Mathematics, Sichuan University, Chengdu, China (\mbox{cglwdm@scu.edu.cn}).}
\and
Wenyi Liu%
\thanks{School of Mathematics, Sichuan University, Chengdu, China (\mbox{liuwenyi@stu.scu.edu.cn}).}
\and
Yangwen Zhang%
\thanks{Department of Mathematics, University of Louisiana at Lafayette, Lafayette, LA 70503, USA (\mbox{yangwen.zhang@louisiana.edu}).}
}

\date{\today}

\maketitle

\begin{abstract}
We develop a family of $H(\mathrm{div})$-conforming hybridizable discontinuous Galerkin methods for the steady Stokes equations based on BDM and RT velocity spaces with either discontinuous or continuous hybrid traces. In contrast to our earlier pressure-robust HDG method for tangential boundary control, the present analysis does not require the pressure to belong to $H^1$; instead, the consistency argument only assumes low pressure regularity. The discrete velocities are exactly divergence-free, which yields pressure robustness. For the BDM variants we derive optimal energy-norm estimates and optimal $L^2$-velocity convergence, while for the RT variants we obtain optimal velocity convergence and weaker pressure estimates. We also analyze the hybridized linear system and prove a uniform spectral equivalence for the pressure Schur complement relevant to iterative solvers. As an application, we revisit the Stokes tangential boundary control problem and derive error estimates for the control, state, and adjoint variables using the BDM discontinuous-trace scheme. Two- and three-dimensional numerical experiments confirm the predicted convergence rates, the exact divergence-free property, and the robustness of the method with respect to the viscosity parameter.
\end{abstract}



\begingroup

\section{Introduction}

Exactly divergence-free and pressure-robust discretizations are particularly attractive for incompressible flow problems. For the steady Stokes equations, one expects the velocity error to depend primarily on the divergence-free part of the forcing, rather than on the size of the irrotational component of the load. Hybridizable discontinuous Galerkin (HDG) methods are well suited to this setting: they combine local conservation, static condensation, and high-order approximation in a flexible finite element framework. The aim of this paper is to develop and analyze a family of $H(\mathrm{div})$-conforming HDG discretizations for the steady Stokes equations and to apply the BDM discontinuous-trace variant to a tangential boundary control problem.

As an application, we consider the tangential Dirichlet boundary control problem
\begin{align}\label{Ori_problem1}
	\min_{\bm u\in \bm U} J(\bm u)=\frac{1}{2}\|\bm y_{\bm u}-\bm y_{d}\|^2_{\bm L^{2}(\Omega)}+\frac{\gamma}{2}\| \bm u\|^2_{\bm U},
\end{align}
where $\gamma>0$ is a regularization parameter, $\bm U = \{\bm u = u\bm\tau; \ u\in L^2(\Gamma)\}$, and $\|\bm u\|_{\bm U} = \|u\|_{L^2(\Gamma)}$, with $\bm\tau$ denoting the unit tangential vector on $\Gamma=\partial\Omega$. The desired state is denoted by $\bm y_d$, while $\bm y_{\bm u}$ is the unique solution in the transposition sense; see \cite[Definition 2.3]{GongStokes_Tangential1}. The state equations read
\begin{align}\label{Ori_problem2}
	-\nu \Delta\bm y+\nabla p =\bm f \quad  \text{in } \Omega, \qquad
	\nabla\cdot\bm y=0 \quad  \text{in } \Omega, \qquad
	\bm y=\bm u \quad  \text{on }  \Gamma, \qquad
	\int_{\Omega} p=0.
\end{align}
Here $\nu=\mathrm{Re}^{-1}$ is the viscosity coefficient.

The analysis and numerical approximation of optimal control problems constrained by Stokes or Navier--Stokes equations have received sustained attention; representative contributions include \cite{MR1613897,MR1871460,MR1885709,MR2051068,MR2326293,MR2974748,MR3233093,MR3357635,MR3663000}. On the discretization side, divergence-conforming HDG methods for the Stokes and Navier--Stokes equations were initiated in \cite{MR3194122} and subsequently extended in several directions, including exactly divergence-free weak Galerkin and HDG schemes for Stokes, quasi-Newtonian Stokes, and stationary Navier--Stokes equations \cite{MR3549196,MR3671717,2023Chen,MR3833698}. A related line of work uses $H(\mathrm{div})$-conforming velocity spaces without additional Lagrange multipliers; see, for example, \cite{MR339677,MR3511719,MR3826676,MR4015217,MR3941890}.

The present paper is also motivated by our earlier work on boundary control. In \cite{GongStokes_Tangential1} we analyzed an HDG method for tangential boundary control of the Stokes system, but the method was not pressure-robust. In \cite{MR4527345} we proposed a pressure-robust HDG discretization for the same control problem; however, its error analysis required the pressure to belong to $H^1$. The main purpose of the present paper is to remove that regularity assumption while simplifying the discrete formulation.

Our main contributions are as follows.
\begin{itemize}
	\item We introduce a family of $H(\mathrm{div})$-conforming HDG discretizations for the steady Stokes equations based on BDM and RT velocity spaces with either discontinuous or continuous hybrid traces.
	\item The discrete velocities are exactly divergence-free, which yields pressure robustness.
	\item The consistency analysis only requires low pressure regularity, thereby removing the $H^1$-pressure assumption used in \cite{MR4527345}.
	\item For the BDM variants we prove optimal energy-norm estimates and optimal $L^2$-velocity convergence; for the RT variants we prove optimal velocity convergence together with weaker pressure estimates.
	\item We reformulate the Schur-complement discussion in terms of spectral equivalence and apply the BDM discontinuous-trace method to tangential boundary control.
\end{itemize}

The remainder of the paper is organized as follows. In \Cref{sec2} we introduce notation and collect the approximation tools used throughout the analysis. \Cref{sec3,sec4} develop the BDM- and RT-based discontinuous-trace HDG methods and establish well-posedness and a priori error estimates. \Cref{sec5} discusses the corresponding continuous-trace variants. In \Cref{sec6} we prove a spectral equivalence for the Schur complement. \Cref{sec7} applies the BDM discontinuous-trace scheme to the tangential boundary control problem. Finally, \Cref{sec8} presents numerical experiments in two and three space dimensions.
\endgroup

\section{Notation and Preliminary Results}\label{sec2}

\begingroup
In this section we introduce the notation used throughout the paper and collect several standard approximation and trace estimates that will be used repeatedly in the analysis.
\endgroup

\subsection{Notation}
Consider an open polygonal domain $\Omega \subset \mathbb{R}^d$ (where $d$ can be 2 or 3) with a Lipschitz boundary denoted as $\Gamma = \partial \Omega$. We define a bounded domain $\mathcal{D} \subset \mathbb{R}^s$, where $s$ matches the dimension ($s = d$ for $d$ dimensions and $s = d-1$ for $d-1$ dimensions). Within this context, $H^m(\mathcal{D})$ and $H^m_0(\mathcal{D})$ denote the Sobolev spaces over $\mathcal{D}$, with $m$ representing the order of the spaces. The norms and semi-norms corresponding to these spaces are denoted by $\|\cdot\|_{m,\mathcal{D}}$ and $|\cdot|_{m,\mathcal{D}}$, respectively.

We represent the inner product within the Sobolev space $H^m(\mathcal{D})$ as $(\cdot, \cdot)_{m, \mathcal{D}}$, where $(\cdot, \cdot)_{\mathcal{D}}$ is shorthand for $(\cdot, \cdot)_{0, \mathcal{D}}$. In the case where $\mathcal{D}$ corresponds to $\Omega$, we use the notation
$$
\|\cdot\|_m:=\|\cdot\|_{m, \Omega}, \quad|\cdot|_m:=|\cdot|_{m, \Omega}, \quad(\cdot, \cdot):=(\cdot, \cdot)_{\Omega} .
$$

Specifically, when ${\mathcal{D}} \subset \mathbb{R}^{d-1}$, we employ $\langle\cdot, \cdot\rangle_{\mathcal{D}}$ in place of $(\cdot, \cdot)_{\mathcal{D} }$. Vector-valued variables and spaces are denoted using boldface font, such as $\bm u$ representing a velocity vector field. For tensor variables, we adopt capital fonts, for instance, $\mathbb L = \nabla\bm u$ indicating the gradient of the velocity vector field. Scalar variables and spaces are represented in italics, with $p$ signifying a pressure scalar value.

Let $\mathcal{P}_k(\mathcal{D})$ denote polynomials in $d$ variables with a total degree at most $k$, where $k$ is a non-negative integer. Additionally, we require the following spaces:
\begin{align*}    L_0^2(\mathcal{D})&:=\left\{q \in L^2(\mathcal{D}):(q, 1)=0\right\}, \\  \boldsymbol{H}({\rm div}; \mathcal{D})&:=\left\{\bm{v} \in [{L}^2(\mathcal{D})]^d: \nabla \cdot \boldsymbol{v} \in L^2(\mathcal{D})\right\}.
\end{align*}

\begingroup
Consider $\left\lbrace \mathcal{T}_h\right\rbrace_{h>0} $ as a shape-regular sequence of simplicial elements within the polyhedral domain $\Omega$. Given a subdivision $\mathcal{T}_h=\cup\left\lbrace T \right\rbrace $, we represent the collection of its edges or faces as $\mathcal E_h$. Furthermore, we designate $h_T$ as the diameter of each element $T\in\mathcal{T}_h$. We conventionally denote $h$ as the maximum of these $h_T$ values across all $T \in \mathcal{T}_h$. For simplicity, we define:
$$
\mathfrak{h}|_T=h_T ,\quad \eta|_{\partial T}=h_T^{-1}.
$$
\endgroup

The assortment $\mathcal{T}_h$ earns the label `shape-regular' when $h_T\leq C\rho_T$, where $\rho_T $ denotes the diameter of the largest ball that can be inscribed within $T$. Here, the presence of a constant $C=C(\mathcal{T}_h)$ is intrinsic to this definition. To ensure its applicability with $C$ being detached from $h$, we necessitate the existence of a collection of triangulations $\mathcal{T}_h$, each specifically labeled with the overarching parameter $h$.

Employing augmented inequalities, we find immense utility in circumventing the introduction of constants that might vary in different instances and hinge on $h$, $h_T$, and the coefficient $\nu$. Specifically, the notation $a \lesssim b$ signifies $a \leq C b$, where $C>0$ remains untethered from $h$, $h_T$, and $\nu$. In situations demanding emphasis, we will distinctly highlight the presence of the constant $C$.

We further introduce the subsequent inner product and norm, both contingent on the mesh:
\begin{align*}
	(u, v)_{ \mathcal{T}_h} :=\sum_{T \in \mathcal{T}_h}( u, v)_{ T}, \quad	\|u\|_{0, \mathcal{T}_h}^2 :=\sum_{T \in \mathcal{T}_h}\|u\|_{0, T}^2 .
\end{align*}

To streamline the notation and avoid excessive subscripts, we employ succinct notations:
\begin{align*}
	\|u\|_{\mathcal{T}_h}^2	 = \|u\|_{0, \mathcal{T}_h}^2, \quad \|u\|_{m}^2 = \|u\|_{m, \Omega}^2.	
\end{align*}

\begingroup
\subsection{Preliminary Results}
We first record standard inverse and trace estimates.
\begin{lemma}[Inverse and Trace Inequalities{\cite[Lemma 1.46]{di2011mathematical}, \cite[Lemma 1.138]{ern2004theory}}] There exist constants $C$ and $C_{\max}$ that are independent of $h_T$, such that for every $T \in \mathcal{T}_h$ and every polynomial function $v_h$ in $T$, the following inequalities hold:
	\begin{subequations}
		\begin{align}
			|v_h|_{s,T} &\leq C h_T^{m-s}|v_h|_{m,T}, \quad 0 \leq m \leq s, \label{trace_1}\\
			h_T^{\frac{1}{2}}\|v_h\|_{0,\partial T} &\leq C_{\max }\|v_h\|_{0,T}.\label{trace_2}
		\end{align}
	\end{subequations}
	
	Moreover, there exists a constant $C > 0$ that is independent of $h_T$, such that for any $v \in H^1(T)$, the following local trace inequality holds:
	\begin{align}
		\|v\|_{0,\partial T} \leq C(h_T^{-\frac{1}{2}}\|v\|_{0,T}+h_T^{\frac{1}{2}}|v|_{1,T}).
	\end{align}
\end{lemma}

For all $E\in\mathcal{E}_h$ and $T\in\mathcal{T}_h$, we denote by ${\Pi}_{k}^{\partial}:[L^2(E)] \rightarrow[\mathcal{P}_{k}(E)]$ and ${\Pi}_{k}^{o}:[L^2(T)]\rightarrow[\mathcal{P}_{k}(T)]$ the scalar $L^2$ projections. For vector- and tensor-valued fields these projections are understood componentwise. More precisely, for every $v\in L^2(E)$ and $p\in L^2(T)$, ${\Pi}_{k}^{\partial} v$ and ${\Pi}_{k}^{o} p$ are the unique elements of $\mathcal{P}_{k}(E)$ and $\mathcal{P}_{k}(T)$, respectively, that satisfy:
\begin{align}\label{L2}
	\langle \Pi_k^{\mathrm{\partial}} {v}, w_k\rangle_{E}&=\langle {v}, w_k \rangle_{E},\quad \forall w_k\in\mathcal{P}_{k}(E), E\in \mathcal{E}_h {,}\\
	( \Pi_k^{\mathrm{o}} {p}, q_k)_{ T} &=(  {p}, q_{k})_{T},\quad \forall q_k\in\mathcal{P}_{k}(T), T\in\mathcal{T}_h.
\end{align}
\endgroup

\begin{lemma}[Approximation Results\cite{ern2004theory,girault2012finite}]\label{projection properties}
	For any $T\in \mathcal{T}_h$, $E\in\mathcal E_h$, and $k\geq1$ with $1\leq m\leq k+1$, the following interpolation estimates and bounds hold.
	\begin{subequations}
		\begin{align}
			\|{\Pi}^o_k v\|_{0, T} &\leq\|v\|_{0, T}, &&\forall v \in L^2(T),\label{2.7} \\
			\|{\Pi}^{\partial}_k v\|_{0, E} &\leq\|v\|_{0, E}, &&\forall v \in L^2(E),\\
			\|v-{\Pi}^{\partial}_k v\|_{0, \partial T} &\lesssim h_T^{m-1 / 2}|v|_{m, T}, &&\forall v \in H^m(T), \\
			\left\|v-{\Pi}^o_k v\right\|_{0, T}+h_T\left|v-{\Pi}^o_k v\right|_{1, T} &\lesssim h_T^m|v|_{m, T}, &&\forall v \in H^m(T).\label{2.4'}
		\end{align}
	\end{subequations}
\end{lemma}

\begin{lemma}[The BDM Projection{\cite[page 113]{boffi2013mixed}}{\cite[page 39]{du2019invitation}}]\label{BDM}
	For any $T\in\mathcal{T}_h$ and $\bm v\in[H^1(T)]^d$, there exists a unique $\bm\Pi_k^{\rm BDM}\bm v\in[\mathcal P_k(T)]^d,k\geq1,$ such that
	\begin{align}
		& \left\langle \bm\Pi_k^{\mathrm{{B}DM}} \boldsymbol{v} \cdot \boldsymbol{n}, w_k\right\rangle_{\partial T}=\left\langle\boldsymbol{v} \cdot \boldsymbol{n}, w_k\right\rangle_{\partial T}, && \forall w_k \in \mathcal{P}_k(E), E \subset \partial T, \\
		& \left(\bm\Pi_k^{\mathrm{{B}DM}} \boldsymbol{v}, \nabla p_{k-1}\right)_T=\left(\boldsymbol{v}, \nabla p_{k-1}\right)_T, && \forall p_{k-1} \in\mathcal P_{k-1}(T),\label{2.6} \\
		& \left(\bm\Pi_k^{\mathrm{{B}DM}} \boldsymbol{v}, {\bf curl}\left(b_T p_{k-2}\right)\right)_T=\left(\boldsymbol{v}, {\bf curl}\left(b_T p_{k-2}\right)\right)_T, && \forall p_{k-2} \in \mathcal P_{k-2}(T), \label{2.12}
	\end{align}
	when $d=2$ and $k \geq 2$, and
	$$
	\left(\boldsymbol{\Pi}_k^{\rm BDM} \boldsymbol{v}, \boldsymbol{w}\right)_T=(\boldsymbol{v}, \boldsymbol{w})_T, \quad \forall \boldsymbol{w} \in \boldsymbol{P}_k^*(T), 
	$$
	when $d=3$, where $b_T=\lambda_1\lambda_2\lambda_3$ in \eqref{2.12} is the bubble function on $T$, and
	$$
	\boldsymbol{P}_k^*(T):=\left\{\boldsymbol{v} \in\left[\mathcal P_k(T)\right]^d: \nabla \cdot \boldsymbol{v}=0, \boldsymbol{v} \cdot \boldsymbol{n}_E=0, \forall E \subset \partial T\right\}.
	$$
	Moreover, it holds the following interpolation approximation:
	\begin{align}\label{BDM error}
		\|\boldsymbol{v}-\boldsymbol{\Pi}_k^{\rm BDM} \boldsymbol{v}\|_{0, T} \lesssim h_T^m|\boldsymbol{v}|_{m, T}, \quad \forall \boldsymbol{v} \in\left[H^m(T)\right]^d,\quad 1 \leq m \leq k+1.
	\end{align}
	Furthermore, $\boldsymbol{\Pi}_k^{\rm BDM}$ holds the commutativity property:
	\begin{align}\label{commutativity}
		&(\nabla\cdot\boldsymbol{\Pi}_k^{\rm BDM}\bm v,p)_T=(\nabla\cdot\bm v,p)_T,\quad\forall p\in\mathcal P_{k-1}(T).
	\end{align}
\end{lemma}

\begingroup
\section{H(div)-conforming HDG method with discontinuous traces}\label{sec3}
In this section we introduce the first $H(\mathrm{div})$-conforming HDG scheme, based on BDM velocity spaces with discontinuous traces, and prove its well-posedness and a priori error estimates.
\subsection{The finite element space and the HDG method}\label{first Scheme}
First, for $k\ge 1$, we define
\begin{align*}
	\bm{\Sigma}_h&=\{\mathbb L_h\in \bm L^2(\Omega):\mathbb L_h|_T\in [\mathcal P_{k-1}(T)]^{d\times d},\forall T\in\mathcal T_h\},\\
	\bm {V^g}_h&=\{\bm v_h\in  {\bm H({\rm div};\Omega)}: \bm v_h|_T\in [\mathcal P_k(T)]^d,\forall T\in\mathcal T_h,\bm v_h\cdot \bm n |_{\partial \Omega}=\Pi_{k}^{\partial} (\bm g\cdot\bm n)\},\\
	\widehat{\bm V}_h^{\bm g}&=\{\widehat{\bm v}_h\in\bm L^2(\mathcal E_h):\widehat{\bm v}_h|_E\in [\mathcal P_{k-1}(E)]^d,\forall E\in\mathcal E_h,	\widehat{\bm v}_h|_{\Gamma}=\bm{\Pi}_{k-1}^{\partial}\bm g
	\},\\
	Q_h&=\{q_h\in L_0^2(\Omega):q_h|_T\in \mathcal{P}_{k-1}(T),\forall T\in\mathcal T_h\}.
\end{align*}
\endgroup

Next, we introduce $\mathbb L = \nu \nabla\bm u$ into equation \eqref{Ori_problem2} to derive a formulation involving velocity gradient, velocity, and pressure.
\begin{subequations}\label{2.2}
	\begin{align}
		\nu^{-1}\mathbb L -\nabla \bm u&=0&\text{in }\Omega,\label{2.2a}\\
		-\nabla\cdot\mathbb L+\nabla p&=\bm f&\text{in }\Omega,\\
		\nabla\cdot\bm u&=0&\text{in }\Omega,\\
		\bm u&=\bm g&\text{on }\partial\Omega.
	\end{align}
\end{subequations}

\begingroup
Subsequently, we introduce the globally divergence-free HDG method for \eqref{2.2}, aiming to find a solution $(\mathbb L_h, \bm u_h, \widehat{\bm u}_h, q_h) \in \bm{\Sigma}_h \times \bm{V^g}_h \times \widehat{\bm V}_h^{\bm g} \times Q_h$ that satisfies:
\begin{subequations}\label{2.3}
	\begin{align}
		\nu^{-1}(\mathbb L_h,\mathbb G_h)_{\mathcal T_h}+(\bm u_h,\nabla\cdot\mathbb G_h)_{\mathcal T_h}
		-\langle \widehat{\bm u}_h ,\mathbb G_h\bm n \rangle_{\partial\mathcal T_h}&= 0,\label{2.3a}\\
		-(\nabla\cdot\mathbb L_h,\bm v_h)_{\mathcal T_h}
		+\langle\mathbb L_h\bm n,\widehat{\bm v}_h \rangle_{\partial\mathcal T_h} -(p_h,\nabla \cdot\bm v_h)_{\mathcal T_h} \nonumber
		&\\
		+\nu\langle \eta(\bm{\Pi}_{k-1}^{\partial}\bm u_h-\widehat{\bm u}_h ),
		\bm{\Pi}_{k-1}^{\partial}\bm v_h-\widehat{\bm v}_h
		\rangle_{\partial\mathcal T_h}
		&=(\bm f,\bm v_h)_{\mathcal T_h},\\
		(\nabla\cdot\bm u_h,q_h)_{\mathcal T_h}&=0
	\end{align}
	holds for all $(\mathbb G_h,\bm v_h,\widehat{\bm v}_h,q_h)\in \bm{\Sigma}_h\times\bm V_h^{\bm 0}\times \widehat{\bm V}_h^{\bm 0}\times Q_h$.
\end{subequations}
To simplify the presentation, we define 
\begin{align*}
	&\mathscr B(\mathbb L_h,\bm u_h,\widehat{\bm u}_h,p_h;
	\mathbb G_h,\bm v_h,\widehat{\bm v}_h, q_h)\\
	&\qquad:= 	\nu^{-1}(\mathbb L_h,\mathbb G_h)_{\mathcal T_h}+(\bm u_h,\nabla\cdot\mathbb G_h)_{\mathcal T_h}
	-\langle \widehat{\bm u}_h,\mathbb G_h\bm n \rangle_{\partial\mathcal T_h}\\
	&\qquad\quad	-(\nabla\cdot\mathbb L_h,\bm v_h)_{\mathcal T_h}	+\langle\mathbb L_h\bm n,\widehat{\bm v}_h \rangle_{\partial\mathcal T_h} -(p_h,\nabla \cdot\bm v_h)_{\mathcal T_h}\\
	&\qquad\quad	+\nu\langle \eta(\bm{\Pi}_{k-1}^{\partial}\bm u_h-\widehat{\bm u}_h),
	\bm{\Pi}_{k-1}^{\partial}\bm v_h-\widehat{\bm v}_h
	\rangle_{\partial\mathcal T_h}
	+ (\nabla\cdot\bm u_h,q_h)_{\mathcal T_h}.
\end{align*}

The HDG formulation \eqref{2.3} can be succinctly described as follows: find  $(\mathbb L_h,\bm u_h,\widehat{\bm u}_h,q_h)\in \bm{\Sigma}_h\times\bm {V^g}_h\times \widehat{\bm V}_h^{\bm g}\times Q_h$ such that 
\begin{align}\label{2.4}
	\mathscr B(\mathbb L_h,\bm u_h,\widehat{\bm u}_h,p_h;
	\mathbb G_h,\bm v_h,\widehat{\bm v}_h,q_h)=	(\bm f,\bm v_h)_{\mathcal T_h}
\end{align}
holds for all $(\mathbb G_h,\bm v_h,\widehat{\bm v}_h,q_h)\in \bm{\Sigma}_h\times\bm V_h^{\bm 0}\times \widehat{\bm V}_h^{\bm 0}\times Q_h$.

\subsection{Well-Posedness of the Scheme}
We next prove well-posedness of the HDG scheme. To this end we introduce, for $(\mathbb G_h,\bm v_h,\widehat{\bm v}_h)\in\bm{\Sigma}_h\times\bm V_h^{\bm 0}\times \widehat{\bm V}_h^{\bm 0}$, the norm
\begin{align}\label{norm1}
	\interleave(\mathbb G_h,\bm v_h,\widehat{\bm v}_h) \interleave^2
	:=\nu^{-1}\|\mathbb G_h\|_{\mathcal T_h}^2
	+\nu\|\nabla\bm v_h\|_{\mathcal T_h}^2
	+\nu\|\eta^{1/2}(\bm\Pi_{k-1}^{\partial}\bm v_h-\widehat{\bm v}_h)\|_{\partial\mathcal T_h}^2.
\end{align}

In this context, it's worth recalling that $\eta|_{\partial T}=h_T^{-1}$.

\begin{lemma}\label{check norm1}
	The semi-norm $\interleave\cdot\interleave$ defined by \eqref{norm1} is a norm  {on $\bm{\Sigma}_h\times\bm V_h^{\bm 0}\times \widehat{\bm V}_h^{\bm 0}$.}
\end{lemma}
\begin{proof}
	Since $\interleave\cdot\interleave$ is a semi-norm, our goal is to show that $\interleave(\mathbb G_h, \bm v_h, \widehat{\bm v}_h) \interleave^2 = 0$ implies that $\mathbb G_h$, $\bm v_h$, and $\widehat{\bm v}_h$ are all equal to $\bm 0$. We begin by assuming $\interleave(\mathbb G_h, \bm v_h, \widehat{\bm v}_h) \interleave^2 = 0$. This assumption leads to the following conclusions:
	\begin{align*}
		\mathbb G_h = \nabla \bm v_h = \bm 0 \textup{ on each } T\in \mathcal T_h  \textup{ and  }\bm\Pi_{k-1}^{\partial}\bm v_h - \widehat{\bm v}_h=\bm 0 \textup{ on }\partial T. 
	\end{align*}
	
	First, when $\nabla\bm v_h = \bm 0$, it implies that $\bm v_h$ is constant on each element $T$. This, combined with $\bm\Pi_{k-1}^{\partial}\bm v_h = \widehat{\bm v}_h$, further implies that $\widehat{\bm v}_h$ is also constant on $\partial T$. Second, since $\widehat{\bm v}_h = 0$ on $\partial \Omega$, this indicates that $\bm v_h$ is zero on elements that share at least one face or one edge with the boundary.
	
	Since $\widehat{\bm v}_h$ is single-valued on the interior face or edge, we can conclude that $\bm v_h=0$ on all interior elements. It is clear that if $\|\mathbb G_h\|_{\mathcal T_h}=0$, it implies that $\mathbb G_h=\bm 0$. 
\end{proof}

The next theorem gives the inf-sup condition for the bilinear form $\mathscr B$.
\begin{theorem}\label{Theorem 2.3}
	{ For all $(\mathbb L_h,\bm u_h,\widehat{\bm u}_h,p_h)\in \bm{\Sigma}_h \times \bm V_h^{\bm 0} \times \widehat{\bm V}_h^{\bm 0} \times Q_h$, there} exists a constant $\beta>0$ such that
	\begin{subequations}
		\begin{align}
			\sup_{\bm 0\neq (\mathbb G_h,\bm v_h,\widehat{\bm v}_h,q_h)\in \bm{\Sigma}_h\times\bm V_h^{\bm 0}\times \widehat{\bm V}_h^{\bm 0}\times Q_h}\frac{\mathscr B(\mathbb L_h,\bm u_h,\widehat{\bm u}_h,p_h;
				\mathbb G_h,\bm v_h,\widehat{\bm v}_h,q_h)}{	\interleave(\mathbb G_h,\bm v_h,\widehat{\bm v}_h)\interleave+\|q_h\|_{\mathcal T_h}}
			\ge \beta( 	\interleave(\mathbb L_h,\bm u_h,\widehat{\bm u}_h)\interleave+\|p_h\|_{\mathcal T_h}  ),\label{BDMinfsup1}\\
			\sup_{\bm 0\neq (\mathbb G_h,\bm v_h,\widehat{\bm v}_h,q_h)\in \bm{\Sigma}_h\times\bm V_h^{\bm 0}\times \widehat{\bm V}_h^{\bm 0}\times Q_h}\frac{	\mathscr B(
				\mathbb G_h,\bm v_h,\widehat{\bm v}_h,q_h;\mathbb L_h,\bm u_h,\widehat{\bm u}_h,p_h)}{	\interleave(\mathbb G_h,\bm v_h,\widehat{\bm v}_h)\interleave+\|q_h\|_{\mathcal T_h}}
			\ge \beta( 	\interleave(\mathbb L_h,\bm u_h,\widehat{\bm u}_h)\interleave+\|p_h\|_{\mathcal T_h}  ).\label{BDMinfsup2}
		\end{align}	
	\end{subequations}
\end{theorem}

\begin{proof}
	We will provide a proof for \eqref{BDMinfsup1} since the proof of \eqref{BDMinfsup2} follows a similar approach. To establish \eqref{BDMinfsup1}, we will select appropriate test functions and proceed through the following four steps.
	
	\textbf{Step 1.} Considering the tuple $\left( \mathbb G_1, \boldsymbol{v}_1, \widehat{\boldsymbol{v}}_{1}, q_h \right) = \left( \mathbb{L}_{h}, \boldsymbol{u}_{h}, \widehat{\boldsymbol{u}}_{h}, p_{h} \right)$ in the space $\bm{\Sigma}_h \times \bm V_h^{\bm 0} \times \widehat{\bm V}_h^{\bm 0} \times Q_h$, and with reference to the definition of $\mathscr B$, it becomes evident that
	\begin{align}	
		&\mathscr{B}\left(\mathbb{L}_{h}, \boldsymbol{u}_{h}, \widehat{\boldsymbol{u}}_{h}, p_{h} ;\mathbb{L}_{h}, \boldsymbol{u}_{h}, \widehat{\boldsymbol{u}}_{h}, p_{h}\right)=\nu^{-1}\left\|\mathbb L_h\right\|_{\mathcal T_h}^2+\nu\|\eta^{1/2}(\bm\Pi_{k-1}^{\partial}\bm u_h-\widehat{\bm u}_h)\|_{\partial\mathcal T_h}^2.\label{step1}
	\end{align} 
	
	\textbf{Step 2.} Considering the tuple  $\left( \mathbb G_2,\boldsymbol{v}_2,\widehat{\boldsymbol{v}}_{2},q_h\right)=\left(-\nu\nabla\bm u_h , \bm 0, \bm 0,0 \right)\in  \bm{\Sigma}_h\times\bm V_h^{\bm 0}\times \widehat{\bm V}_h^{\bm 0}\times Q_h $. By applying Green's formula, we can derive the following:
	\begin{align}
		\mathscr{B}&\left(\mathbb{L}_{h}, \boldsymbol{u}_{h}, \widehat{\boldsymbol{u}}_{h}, p_{h} ;-\nu\nabla\bm u_h , \bm 0, \bm 0,0 \right)\nonumber\\
		&=	(\mathbb L_h,-\nabla\bm u_h)_{\mathcal T_h}-\nu(\bm u_h,\Delta\bm u_h)_{\mathcal T_h}
		+\nu\langle \widehat{\bm u}_h,\nabla\bm u_h\bm n \rangle_{\partial\mathcal T_h}\nonumber\\
		&=	(\mathbb L_h,-\nabla\bm u_h)_{\mathcal T_h}+\nu(\nabla\bm u_h,\nabla\bm u_h)_{\mathcal T_h}-\nu\langle {\bm u}_h- \widehat{\bm u}_h,\nabla\bm u_h \bm n \rangle_{\partial\mathcal T_h}.\label{step2_1}
	\end{align}
	Regarding the initial two terms in \eqref{step2_1}, we can establish their relationship using the Cauchy-Schwarz inequality, leading to the following result:
	\begin{align}\label{2.20}
		(\mathbb L_h,-\nabla\bm u_h)_{\mathcal T_h}+\nu(\nabla\bm u_h,\nabla\bm u_h)_{\mathcal T_h}\geq \nu\left\| \nabla\bm u_h\right\|^2_{\mathcal T_h}-\left\| \mathbb L_h\right\|_{\mathcal T_h} \left\| \nabla\bm u_h\right\|_{\mathcal T_h}.
	\end{align}
	Subsequently, utilizing the orthogonality property of $\bm\Pi^{\partial}_{k-1}$ in \eqref{L2}, we obtain:
	\begin{align}\label{2.21}
		-\nu\langle {\bm u}_h- \widehat{\bm u}_h,\nabla\bm u_h \bm n \rangle_{\partial\mathcal T_h}&=-\nu\langle \bm\Pi^{\partial}_{k-1}{\bm u}_h- \widehat{\bm u}_h,\nabla\bm u_h \bm n \rangle_{\partial\mathcal T_h}\nonumber\\
		&\geq-C_1\nu\|\eta^{1/2}( \bm\Pi^{\partial}_{k-1}{\bm u}_h- \widehat{\bm u}_h)\| _{\partial\mathcal T_h}\left\| \nabla\bm u_h\right\|_{\mathcal T_h},
	\end{align}
	where we used the trace inequality \eqref{trace_2} in \eqref{2.21}. In conjunction with \eqref{2.20} and \eqref{2.21}, and through the application of Young's inequality, we obtain:
	\begin{align}
		\mathscr{B}&\left(\mathbb{L}_{h},\boldsymbol{u}_{h}, \widehat{\boldsymbol{u}}_{h}, p_{h} ;-\nu\nabla\bm u_h , \bm 0, \bm 0,0 \right)\nonumber\\
		& 	\geq\nu\left\| \nabla\bm u_h\right\|^2_{\mathcal T_h}-\left\| \mathbb L_h\right\|_{\mathcal T_h} \left\| \nabla\bm u_h\right\|_{\mathcal T_h}-\nu C_1\|\eta^{1/2}( \bm\Pi^{\partial}_{k-1}{\bm u}_h- \widehat{\bm u}_h)\| _{\partial\mathcal T_h}\left\| \nabla\bm u_h\right\|_{\mathcal T_h}\nonumber\\
		&\geq\nu\left\| \nabla\bm u_h\right\|^2_{\mathcal T_h}-(\dfrac{1}{4}\nu\left\| \nabla\bm u_h\right\|^2_{\mathcal T_h}+\nu^{-1}\left\| \mathbb L_h\right\|^2_{\mathcal T_h})\nonumber\\
		&\quad-(\dfrac{1}{4}\nu\left\| \nabla\bm u_h\right\|^2_{\mathcal T_h}+(C_1)^2\nu\|\eta^{1/2}( \bm\Pi^{\partial}_{k-1}{\bm u}_h- \widehat{\bm u}_h)\|^2 _{\partial\mathcal T_h})\nonumber\\
		&\geq\dfrac{1}{2}\nu\left\| \nabla\bm u_h\right\|^2_{\mathcal T_h}-(\nu^{-1}\left\| \mathbb L_h\right\|^2_{\mathcal T_h}+({C_1})^2\nu\|\eta^{1/2}( \bm\Pi^{\partial}_{k-1}{\bm u}_h- \widehat{\bm u}_h)\|^2 _{\partial\mathcal T_h}).\label{step2}
	\end{align}
	
	\textbf{Step 3.} Given that $p_h\in L_{0}^{2}(\Omega)$ and invoking the continuous inf-sup condition (see \cite[Page 6, Lemma 2.2]{brezzi2008mixed}), it follows that there exists a $\boldsymbol{v} \in [H_{0}^{1}(\Omega)]^d$ such that:
	\begin{align}\label{continuous inf-sup}
		\nabla \cdot \boldsymbol{v}=p_h,\quad|\boldsymbol{v}|_{H^1(\Omega)} \lesssim\left\|p_h\right\|_{L^2(\Omega)}.
	\end{align}
	Subsequently, when we consider the tuple  
	$$\left( \mathbb G_3, {\boldsymbol{v}_3},\widehat{\boldsymbol{v}}_{3},q_h\right)=\left(0 , \boldsymbol{\Pi}_k^{\rm BDM}\bm v, \bm{\Pi}_{k-1}^{\partial}\bm v,0 \right)\in  \bm{\Sigma}_h\times\bm V_h^{\bm 0}\times \widehat{\bm V}_h^{\bm 0}\times Q_h,  $$ we get:
	\begin{align}
		&\mathscr{B}(\mathbb{L}_{h}, \boldsymbol{u}_{h}, \widehat{\boldsymbol{u}}_{h}, p_{h} ;0,\boldsymbol{\Pi}_k^{\rm BDM}\bm v ,\bm{\Pi}_{k-1}^{\partial}\bm v, 0)\nonumber\\
		&\qquad= -(\nabla\cdot\mathbb L_h,\boldsymbol{\Pi}_k^{\rm BDM}\bm v)_{\mathcal T_h}	+\langle\mathbb L_h\bm n,\bm{\Pi}_{k-1}^{\partial}\bm v \rangle_{\partial\mathcal T_h} -(p_h,\nabla \cdot\boldsymbol{\Pi}_k^{\rm BDM}\bm v)_{\mathcal T_h}\nonumber\\
		&\qquad\quad+\nu\langle \eta(\bm{\Pi}_{k-1}^{\partial}\bm u_h-\widehat{\bm u}_h ),
		\bm{\Pi}_{k-1}^{\partial}\boldsymbol{\Pi}_k^{\rm BDM} \boldsymbol{v} -\bm{\Pi}_{k-1}^{\partial}\bm v\rangle_{\partial\mathcal T_h}\nonumber\\
		&\qquad= (\mathbb L_h,\nabla\boldsymbol{\Pi}_k^{\rm BDM}\bm v)_{\mathcal T_h}+\langle\mathbb L_h\bm n,\bm{\Pi}_{k-1}^{\partial}\bm v-\boldsymbol{\Pi}_k^{\rm BDM}\bm v \rangle_{\partial\mathcal T_h}-(p_h,\nabla \cdot\boldsymbol{\Pi}_k^{\rm BDM}\bm v)_{\mathcal T_h}\label{step3_1}\\
		&\qquad\quad+\nu\langle \eta(\bm{\Pi}_{k-1}^{\partial}\bm u_h-\widehat{\bm u}_h ),
		\bm{\Pi}_{k-1}^{\partial}\boldsymbol{\Pi}_k^{\rm BDM} \boldsymbol{v} -\bm{\Pi}_{k-1}^{\partial}\bm v\rangle_{\partial\mathcal T_h}.\nonumber
	\end{align}
	Let us analyze the first term in \eqref{step3_1}, it shows 
	\begin{align}
		(\mathbb L_h,&\nabla\boldsymbol{\Pi}_k^{\rm BDM}\bm v)_{\mathcal T_h}\nonumber\\&\leq\left\|\mathbb L_h\right\|_{\mathcal T_h}\|\nabla\boldsymbol{\Pi}_k^{\rm BDM}\bm v\|_{\mathcal T_h}&&\text{(Cauchy-Schwarz inequality)}\nonumber\\
		&\leq \left\|\mathbb L_h\right\|_{\mathcal T_h}( \| \nabla \bm v \|_{\mathcal T_h}+\|\nabla\boldsymbol{\Pi}_k^{\rm BDM}\bm v-\nabla \bm v\|_{\mathcal T_h}) &&\text{(triangle inequality)} \nonumber \\
		&=\left\|\mathbb L_h\right\|_{\mathcal T_h}( |\boldsymbol{v}|_{ 1}+\|\nabla(\boldsymbol{\Pi}_k^{\rm BDM}\bm v-\bm v)\|_{\mathcal T_h}). &&\text{(definition of semi-norm)} \label{first term}
	\end{align}
	Now we proceed to the second term enclosed in parentheses in \eqref{first term}. By applying the triangle inequality, we derive:
	\begin{align*}
		&\|\nabla(\boldsymbol{\Pi}_k^{\rm BDM}\bm v-\bm v)\|_{\mathcal T_h}\nonumber\\
		&\quad \leq\left\| \nabla(\bm v-\bm{\Pi}_{k} \bm v)\right\| _{\mathcal T_h}+\| \nabla(\boldsymbol{\Pi}_k^{\rm BDM}\bm v-\bm{\Pi}_{k}\bm v)\| _{\mathcal T_h}\nonumber \\
		&\quad\leq c_1\left|\bm v\right|_{ 1}+c_2\|\mathfrak{h}^{-1} \bm{\Pi}_{k}(\boldsymbol{\Pi}_k^{\rm BDM}\bm v-\bm v)\| _{\mathcal T_h}\nonumber &&\text{(by \eqref{2.4'} and \eqref{trace_1})}\\
		&\quad\leq c_1|\bm v|_{ 1}+ {c_2}\|\mathfrak{h}^{-1}(\boldsymbol{\Pi}_k^{\rm BDM}\bm v-\bm v)\|_{\mathcal T_h}\nonumber &&\text{(by \eqref{2.7})}\\
		&\quad\leq c_1|\bm v|_{ 1}+ {c_3}|\bm v|_{ 1}\nonumber &&\text{(by \eqref{BDM error})}\\
		&\quad=C|\bm v|_{ 1},
	\end{align*}
	where $\mathfrak{h}|_T=h_T$, and $ {c_1,c_2,c_3}$ are positive constant independent of $h$, $h_T$ and $\nu$.  Combined with \eqref{first term}, this yields:
	\begin{align}
		(\mathbb L_h,\nabla\boldsymbol{\Pi}_k^{\rm BDM}\bm v)_{\mathcal T_h}
		\leq(1+C)\left\|\mathbb L_h\right\|_{\mathcal T_h}|\bm v|_{ 1}
		\leq C_2\left\|\mathbb L_h\right\|_{\mathcal T_h}\left\|p_h\right\|_{\mathcal T_h}.\label{2.26'}
	\end{align}
	Furthermore, we utilize the commutativity property of $\boldsymbol{\Pi}_k^{\rm BDM}$ as described in \eqref{commutativity}, in conjunction with \eqref{continuous inf-sup}, resulting in:
	\begin{align}
		(p_h,\nabla \cdot\boldsymbol{\Pi}_k^{\rm BDM}\bm v)_{\mathcal T_h}=(p_h,\nabla \cdot\bm v)_{\mathcal T_h}=\left\|p_h\right\|^2_{\mathcal T_h}.\label{2.27}
	\end{align}
	Utilizing the trace inequality and the approximations presented in \eqref{2.4'} and \eqref{BDM error}, we obtain:
	\begin{align}
		\langle\mathbb L_h\bm n, &\bm{\Pi}_{k-1}^{\partial}\bm v-\bm{\Pi}_k^{\rm BDM}\bm v \rangle_{\partial\mathcal T_h}\\
		&=\langle\mathbb L_h\bm n,\bm{\Pi}_{k-1}^{\partial}\bm v-\bm v \rangle_{\partial\mathcal T_h}+	\langle\mathbb L_h\bm n,\bm v-\bm{\Pi}_k^{\rm BDM}\bm v \rangle_{\partial\mathcal T_h}\nonumber\\
		&\leq
		\|\mathfrak{h}^{1/2}\mathbb L_h\|_{\partial\mathcal T_h}( \|\mathfrak{h}^{-1/2}(\bm{\Pi}_{k-1}^{\partial}\bm v-\bm v )\|_{\partial\mathcal T_h}+\| \mathfrak{h}^{-1/2}(\bm v-\boldsymbol{\Pi}_k^{\rm BDM}\bm v)\|_{\partial\mathcal T_h}) \nonumber\\
		&\leq  c_1\|\mathbb L_h\|_{\mathcal T_h}(                                                                                                                                                                                                                       c_2|\bm v |_{ 1}+c_3|\bm v |_{ 1})\nonumber\\
		&\leq C_3\left\|\mathbb L_h\right\|_{\mathcal T_h}\left\|p_h\right\|_{\mathcal T_h}.\label{2.28}
	\end{align}
	Likewise, when considering the final term in \eqref{step3_1}, we also arrive at:
	\begin{align}
		\nu&\langle \eta(\bm{\Pi}_{k-1}^{\partial}\bm u_h-\widehat{\bm u}_h ),
		\bm{\Pi}_{k-1}^{\partial}\boldsymbol{\Pi}_k^{\rm BDM} \boldsymbol{v} -\bm{\Pi}_{k-1}^{\partial}\bm v\rangle_{\partial\mathcal T_h}\nonumber\\
		&\quad\leq \nu\|\eta^{1/2}(\bm{\Pi}_{k-1}^{\partial}\bm u_h-\widehat{\bm u}_h )\|_{\mathcal T_h}\|	\eta^{1/2}\bm{\Pi}_{k-1}^{\partial}(\boldsymbol{\Pi}_k^{\rm BDM} \boldsymbol{v} -\bm v)\|_{\partial\mathcal T_h}\nonumber\\
		&\quad\leq C\nu\|\eta^{1/2}(\bm{\Pi}_{k-1}^{\partial}\bm u_h-\widehat{\bm u}_h )\|_{\partial\mathcal T_h}|\bm v|_{ 1}\nonumber\\
		&\quad\leq C_4\nu\|\eta^{1/2}(\bm{\Pi}_{k-1}^{\partial}\bm u_h-\widehat{\bm u}_h )\|_{\partial\mathcal T_h}\|p_h\|_{\mathcal T_h}.\label{2.29}
	\end{align}
	In combination with \eqref{2.26'}, \eqref{2.27}, \eqref{2.28}, and \eqref{2.29}, we reach the following conclusion:
	\begin{align}
		\mathscr{B}&\left(\mathbb{L}_{h}, \boldsymbol{u}_{h}, \widehat{\boldsymbol{u}}_{h}, p_{h} ;0,-\boldsymbol{\Pi}_k^{\rm BDM}\bm v ,-\bm{\Pi}_{k-1}^{\partial}\bm v, 0\right)\nonumber\\
		&\qquad\geq\left\|p_h\right\|^2_{\mathcal T_h}-(C_2+C_3)\left\|\mathbb L_h\right\|_{\mathcal T_h}\left\|p_h\right\|_{\mathcal T_h}\nonumber\\
		&\quad \qquad -C_4\nu\|\eta^{1/2}(\bm{\Pi}_{k-1}^{\partial}\bm u_h-\widehat{\bm u}_h )\|_{\partial\mathcal T_h}\|p_h\|_{\mathcal T_h}\nonumber\\
		&\qquad\geq\dfrac{1}{2}\left\|p_h\right\|^2_{\mathcal T_h}-(C_2+C_3)^2\left\|\mathbb L_h\right\|_{\mathcal T_h}^{2}-(C_4)^2\nu\|\eta^{1/2}(\bm{\Pi}_{k-1}^{\partial}\bm u_h-\widehat{\bm u}_h )\|^2_{\partial\mathcal T_h}.\nonumber\\
		&\qquad=\dfrac{1}{2}\left\|p_h\right\|^2_{\mathcal T_h}-(C_5)^2\left\|\mathbb L_h\right\|_{\mathcal T_h}^{2}-(C_4)^2\nu^2\|\eta^{1/2}(\bm{\Pi}_{k-1}^{\partial}\bm u_h-\widehat{\bm u}_h )\|^2_{\partial\mathcal T_h}. \label{step3}
	\end{align}
	Here, $C_1, C_2, C_3, C_4$, and $C_5$ (where $C_5$ is defined as $C_2 + C_3$) are positive constants that do not depend on $h$, $h_T$, or the fluid viscosity coefficient $\nu$. 
	
	\textbf{Step 4.} Based on the implications of \eqref{step1}, \eqref{step2}, and \eqref{step3}, when we test the tuple $(\mathbb G_h, \bm v_h, \widehat{\bm v}_h, q_h)$ with the following components:
	\begin{align*}
		\mathbb G_h&=( 1+(C_1)^2+\nu (C_4)^2+\nu (C_5)^2) \mathbb G_1+\mathbb G_2+\mathbb G_3\in  \bm{\Sigma}_h,\\
		\boldsymbol{v}_h&=\left( 1+(C_1)^2+\nu (C_4)^2+\nu (C_5)^2\right)\boldsymbol{v}_1+\boldsymbol{v}_2+\boldsymbol{v}_3\in\bm {V^0}_{ h},\\
		\widehat{\boldsymbol{v}}_{h}&=\left( 1+(C_1)^2+\nu (C_4)^2+\nu (C_5)^2\right)\widehat{\boldsymbol{v}}_{1}+\widehat{\boldsymbol{v}}_{2}+\widehat{\boldsymbol{v}}_{3}\in\widehat{\bm V}_h^{\bm 0},\\
		q_h&=p_h\in Q_h,			
	\end{align*}
	and then sum up the results, we obtain:
	\begin{align*}
		\mathscr B(\mathbb L_h,\bm u_h,\widehat{\bm u}_h,p_h;
		\mathbb G_h,\bm v_h,\widehat{\bm v}_h,q_h)\geq\dfrac{1}{2}( \interleave(\mathbb L_h,\bm u_h,\widehat{\bm u}_h)\interleave^2+\|p_h\|_{\mathcal T_h}^2  ).
	\end{align*}
	To conclude the proof, we also need to control the norm of the constructed test tuple. By definition,
	\[
	\interleave(\mathbb G_2,\bm v_2,\widehat{\bm v}_2)\interleave
	=\nu^{1/2}\|\nabla\bm u_h\|_{\mathcal T_h}
	\le \interleave(\mathbb L_h,\bm u_h,\widehat{\bm u}_h)\interleave.
	\]
	Moreover, the stability of the BDM projection, the trace inequality, and the continuous inf-sup estimate \eqref{continuous inf-sup} imply
	\[
	\nu^{1/2}\|\nabla\bm v_3\|_{\mathcal T_h}
	+\nu^{1/2}\|\eta^{1/2}(\bm\Pi_{k-1}^{\partial}\bm v_3-\widehat{\bm v}_3)\|_{\partial\mathcal T_h}
	\lesssim \|p_h\|_{\mathcal T_h},
	\]
	and hence $\interleave(\mathbb G_3,\bm v_3,\widehat{\bm v}_3)\interleave\lesssim \|p_h\|_{\mathcal T_h}$. Therefore
	\[
	\interleave(\mathbb G_h,\bm v_h,\widehat{\bm v}_h)\interleave+\|q_h\|_{\mathcal T_h}
	\lesssim \interleave(\mathbb L_h,\bm u_h,\widehat{\bm u}_h)\interleave+\|p_h\|_{\mathcal T_h}.
	\]
	Dividing the lower bound above by this estimate yields \eqref{BDMinfsup1}.
\end{proof}
\endgroup

\subsection{Error Estimations}\label{section3.3}

{
	We refer back to the analogous integration by parts formulation as presented in \cite[Page 1766, (2.7)]{MR2837483} in the following manner.
	
	\begin{lemma}\label{itegration}
		For any $0<\epsilon<1/2$,   $\nabla\cdot(\mathbb L -p \bm I)\in L^{2}(\Omega)$,  $\bm I\in \mathbb R^{d\times d}$ is the identity matrix,  $\mathbb L\in [H^{\epsilon}(\Omega)]^{d\times d}$, and $p \in H^{\epsilon}(\Omega)$, with $\bm v_h\in \bm V_h^{\bm 0}$, the following relations hold:
		\begin{align}\label{I1}
			(\nabla\cdot(\mathbb L - p\bm I),\bm v_h)_{\mathcal T_h}
			+
			(\mathbb L - p\bm I,\nabla\bm v_h )_{\mathcal T_h}
			=
			\langle
			(\mathbb L - p\bm I)\bm n,\bm v_h
			\rangle_{\partial\mathcal T_h},
		\end{align}
		and consequently,
		\begin{align}\label{I2}
			(\nabla\cdot(\mathbb L - p\bm I),\bm v_h)_{\mathcal T_h}
			+
			(\mathbb L - p\bm I,\nabla\bm v_h )_{\mathcal T_h}
			=
			\langle
			\mathbb L \bm n,\bm v_h
			\rangle_{\partial\mathcal T_h}.
		\end{align}
		\begin{proof}
			To establish \eqref{I1}, we aim to demonstrate the validity of the following equation:
			\begin{align}\label{I3}
				(\nabla\cdot(\mathbb L - p\bm I),\bm v_h)_{T}
				+
				(\mathbb L - p\bm I,\nabla\bm v_h )_{T}
				=
				\langle
				(\mathbb L - p\bm I)\bm n,\bm v_h
				\rangle_{\partial T}.
			\end{align}
			In this context, the inner product can be regarded as the dual pair of $H^{\epsilon-1/2}(E)$ and $H^{1/2-\epsilon}(E)=H_0^{1/2-\epsilon}(E)$   on the right side. The validity of \eqref{I3} ensues from the standard density argument and the fact that \eqref{I3} holds for $C^{\infty}(\bar{T})$ functions.
			
			Given $p|_{\partial T}\in H^{{\epsilon-1/2}}(\partial T)$ and $\bm v_h\cdot\bm n\in H^{1/2 -\epsilon}(\partial T)$, along with $\bm v_h\in \bm V_h^{\bm 0}\subset \bm H({\rm div};\Omega)$ with $\bm v_h\cdot\bm n|{\partial\Omega}=0$, it follows that:
			\begin{align*}
				\langle p\bm I\bm n,\bm v_h \rangle_{\partial\mathcal T_h}=	\langle p,\bm v_h\cdot\bm n \rangle_{\partial\mathcal T_h}=0.
			\end{align*}
			Consequently, \eqref{I2} can be derived from \eqref{I1}.
		\end{proof}
	\end{lemma}

}

In this section, we will begin by deriving the error equation as given in \eqref{BDM error equation}. Subsequently, we will conduct an analysis of the $L^2$ error estimation for velocity.

Our approach commences with the establishment of discrete errors, defined in relation to the following quantities:
\begin{align*}
	\bm\xi_h^{\mathbb L}&=\bm{\Pi}_{k-1}^o\mathbb L-\mathbb L_h,\quad
	\bm\xi_h^{\bm u}=\bm{\Pi}_{k}^{\rm BDM}\bm u-\bm u_h,\\
	\bm\xi_h^{\widehat{\bm u}} &=\bm{\Pi}_{k-1}^{\partial}\bm u-\widehat{\bm u}_h,\quad
	\xi_h^{p}={\Pi}_{k-1}^{o}p-p_h,
\end{align*}
where $\bm{\Pi}_{k-1}^o$, $\bm{\Pi}_{k-1}^{\partial}$, ${\Pi}_{k-1}^o$ are $L^2$ projections  {defined in \eqref{L2}.}

\begingroup
\begin{lemma}\label{BDM error equation lemma}
	Assuming that the solution $(\mathbb{L}, \bm{u}, p) \in [ {H^{\epsilon}}(\Omega)]^{d\times d} \times [H^1(\Omega)]^d \times( {H^{\epsilon}}(\Omega)\cap L_0^2(\Omega))$  {with $0<\epsilon<1/2$} satisfies \Cref{2.2}, we can establish the following equation for all $(\mathbb{G}_h, \bm{v}_h, \widehat{\bm{v}}_h, q_h) \in \bm{\Sigma}_h \times \bm{V}^0_h \times \widehat{\bm{V}}_h^0 \times Q_h$:
	\begin{align}\label{BDM error equation}
		\mathscr B(\bm{\Pi}_{k-1}^o\mathbb L,\bm{\Pi}_k^{\rm BDM}\bm u,\bm\Pi_{k-1}^{\partial}{\bm u},\Pi_{k-1}^op;
		\mathbb G_h,\bm v_h,\widehat{\bm v}_h,q_h)=	(\bm f,\bm v_h)_{ \mathcal T_h}	+E(\mathbb L,\bm u;\bm v_h,\widehat{\bm v}_h), 
	\end{align}
	where the error term is defined as:
	\begin{align*}
		&E(\mathbb L,\bm u;\mathbb G_h,\bm v_h,\widehat{\bm v}_h)=\langle(\bm{\Pi}_{k-1}^o\mathbb L-\mathbb L)\bm n,\widehat{\bm v}_h -\bm v_h\rangle_{\partial\mathcal T_h} 
		\\	&\qquad
		+\nu\langle \eta(\bm{\Pi}_k^{\rm BDM}\bm u-\bm u),
		\bm{\Pi}_{k-1}^{\partial}\bm v_h-\widehat{\bm v}_h
		\rangle_{\partial\mathcal T_h}
		+ (\bm{\Pi}_k^{\rm BDM}\bm u - \bm u,\nabla\cdot\mathbb G_h)_{\mathcal T_h}.
	\end{align*}
\end{lemma}
\begin{proof}
	In accordance with the definition of $\mathscr{B}$, we observe that:
	\begin{align*}
		&\mathscr B(\bm{\Pi}_{k-1}^o\mathbb L,\bm{\Pi}_k^{\rm BDM}\bm u,\bm\Pi_{k-1}^{\partial}{\bm u},\Pi_{k-1}^op;
		\mathbb G_h,\bm v_h,\widehat{\bm v}_h,q_h)\\
		&\qquad= 	\nu^{-1}( \bm{\Pi}_{k-1}^o\mathbb L,\mathbb G_h)_{\mathcal T_h}+(\bm{\Pi}_k^{\rm BDM}\bm u,\nabla\cdot\mathbb G_h)_{\mathcal T_h}
		-\langle \bm\Pi_{k-1}^{\partial}{\bm u},\mathbb G_h\bm n \rangle_{\partial\mathcal T_h}\\
		&\qquad\quad	-(\nabla\cdot\bm{\Pi}_{k-1}^o\mathbb L,\bm v_h)_{\mathcal T_h}	+\langle\bm{\Pi}_{k-1}^o\mathbb L\bm n,\widehat{\bm v}_h \rangle_{\partial\mathcal T_h} -(\Pi_{k-1}^op,\nabla \cdot\bm v_h)_{\mathcal T_h}\\
		&\qquad\quad	+\nu\langle\eta\bm{\Pi}_{k-1}^{\partial}(\bm{\Pi}_k^{\rm BDM}\bm u-{\bm u}),
		\bm{\Pi}_{k-1}^{\partial}\bm v_h-\widehat{\bm v}_h
		\rangle_{\partial\mathcal T_h}
		+ (\nabla\cdot\bm{\Pi}_k^{\rm BDM}\bm u,q_h)_{\mathcal T_h}.
	\end{align*}
	For convenience, we define
	\begin{align*}
		&R_1:=\nu^{-1}( \bm{\Pi}_{k-1}^o\mathbb L,\mathbb G_h)_{\mathcal T_h}+(\bm{\Pi}_k^{\rm BDM}\bm u,\nabla\cdot\mathbb G_h)_{\mathcal T_h}-\langle \bm\Pi_{k-1}^{\partial}{\bm u},\mathbb G_h\bm n \rangle_{\partial\mathcal T_h},\\
		&R_2:=-(\nabla\cdot\bm{\Pi}_{k-1}^o\mathbb L,\bm v_h)_{\mathcal T_h}	+\langle\bm{\Pi}_{k-1}^o\mathbb L\bm n,\widehat{\bm v}_h \rangle_{\partial\mathcal T_h} -(\Pi_{k-1}^op,\nabla \cdot\bm v_h)_{\mathcal T_h},\\
		&R_3:=\nu\langle\eta\bm{\Pi}_{k-1}^{\partial}(\bm{\Pi}_k^{\rm BDM}\bm u-{\bm u}),
		\bm{\Pi}_{k-1}^{\partial}\bm v_h-\widehat{\bm v}_h
		\rangle_{\partial\mathcal T_h}
		+ (\nabla\cdot\bm{\Pi}_k^{\rm BDM}\bm u,q_h)_{\mathcal T_h}.
	\end{align*}
	Utilizing the properties of $ \bm{\Pi}_{k-1}^{\partial}$ in \eqref{L2} and the properties of $ \bm{\Pi}_{k}^{\rm BDM}$ in \eqref{2.6}, we can derive the following results:
	\begin{align}\label{BDM R1}
		R_1&=\nu^{-1}( \bm{\Pi}_{k-1}^o\mathbb L,\mathbb G_h)_{\mathcal T_h}+(\bm{\Pi}_k^{\rm BDM}\bm u,\nabla\cdot\mathbb G_h)_{\mathcal T_h}-\langle \bm\Pi_{k-1}^{\partial}{\bm u},\mathbb G_h\bm n \rangle_{\partial\mathcal T_h}\nonumber\\
		&=\nu^{-1}( \mathbb L,\mathbb G_h)_{\mathcal T_h}+(\bm u,\nabla\cdot\mathbb G_h)_{\mathcal T_h}-\langle {\bm u},\mathbb G_h\bm n \rangle_{\partial\mathcal T_h}
		+(\bm{\Pi}_k^{\rm BDM}\bm u - \bm u,\nabla\cdot\mathbb G_h)_{\mathcal T_h}
		\nonumber\\
		&=\nu^{-1}( \mathbb L,\mathbb G_h)_{\mathcal T_h}-(\nabla\bm u,\mathbb{G}_h)_{\mathcal T_h}	+(\bm{\Pi}_k^{\rm BDM}\bm u - \bm u,\nabla\cdot\mathbb G_h)_{\mathcal T_h}.
	\end{align}
	
	Using \Cref{itegration} and the orthogonality of the $L^2$ projection, we arrive at the following expression:
	\begin{align}\label{BDM R2}
		R_2&=(\bm{\Pi}_{k-1}^o\mathbb L,\nabla\bm v_h)_{\mathcal T_h}-\left\langle \bm{\Pi}_{k-1}^o\mathbb L\bm n,\bm v_h\right\rangle_{\partial\mathcal T_h} +\langle\bm{\Pi}_{k-1}^o\mathbb L\bm n,\widehat{\bm v}_h \rangle_{\partial\mathcal T_h}\nonumber\\
		&\quad -(\Pi_{k-1}^op-p,\nabla \cdot\bm v_h)_{\mathcal T_h}-(p,\nabla \cdot\bm v_h)_{\mathcal T_h}\nonumber\\
		&=(\mathbb L,\nabla\bm v_h)_{\mathcal T_h}+\left\langle \bm{\Pi}_{k-1}^o\mathbb L\bm n,\widehat{\bm v}_h-\bm v_h\right\rangle_{\partial\mathcal T_h}-(p,\nabla \cdot\bm v_h)_{\mathcal T_h}\nonumber\\
		& 
		=(\mathbb L-p\bm I,\nabla\bm v_h)_{\mathcal T_h}+\left\langle \bm{\Pi}_{k-1}^o\mathbb L\bm n,\widehat{\bm v}_h-\bm v_h\right\rangle_{\partial\mathcal T_h}\nonumber\\
		& =-(\nabla\cdot(\mathbb{L}-p\bm I),\bm v_h)_{\mathcal T_h}
		+\langle(\mathbb{L}-p\bm I)\bm n,\bm v_h \rangle_{\partial\mathcal T_h}
		+\left\langle \bm{\Pi}_{k-1}^o\mathbb L\bm n,\widehat{\bm v}_h-\bm v_h\right\rangle_{\partial\mathcal T_h}\nonumber\\
		& =-(\nabla\cdot(\mathbb{L}-p\bm I),\bm v_h)_{\mathcal T_h}
		+\langle(\bm{\Pi}_{k-1}^o\mathbb L-\mathbb L)\bm n,\widehat{\bm v}_h-\bm v_h\rangle_{\partial\mathcal T_h}.
	\end{align}
	By the orthogonality of $\bm{\Pi}_{k-1}^o$ and the commutativity property of $\boldsymbol{\Pi}_k^{\rm BDM}$ as given in \eqref{commutativity}, we can state:
	\begin{align}\label{BDM R3}
		R_3=\nu\langle\eta(\bm{\Pi}_k^{\rm BDM}\bm u-{\bm u}),
		\bm{\Pi}_{k-1}^{\partial}\bm v_h-\widehat{\bm v}_h
		\rangle_{\partial\mathcal T_h}.
	\end{align}
	
	Therefore, by consolidating equations \eqref{BDM R1}, \eqref{BDM R2}, and \eqref{BDM R3}, along with the conditions $\nu^{-1}\mathbb L - \nabla \bm u = 0$ and $ -\nabla\cdot(\mathbb{L}-p\bm I) = \bm f$  { in $L^2(\Omega)$ product}, we have completed the proof.
\end{proof}

\begin{lemma}\label{BDM_E}
	For $(\mathbb L,\bm u)\in[{H}^{k}(\Omega)]^{d\times d}\times[ H^{k+1}(\Omega)]^{d}$, it holds the following estimate
	\begin{align}
		|E(\mathbb L,\bm u;\mathbb G_h,\bm v_h,\widehat{\bm v}_h)|\lesssim \nu^{1/2} h^{k}|\bm u|_{k+1}\interleave(\mathbb G_h,\bm v_h,\widehat{\bm v}_h)\interleave,\quad \forall (\bm v_h,\widehat{\bm v}_h)\in \bm {V^0}_{ h}\times \widehat{\bm V}_h^{\bm 0}.
	\end{align}
\end{lemma}
\begin{proof}
	We introduce the following notations:
	\begin{align*}
		E_1&:=\langle(\bm{\Pi}_{k-1}^o\mathbb L-\mathbb L)\bm n,\widehat{\bm v}_h -\bm v_h\rangle_{\partial\mathcal T_h},\\
		 E_2&:=\nu\langle \eta(\bm{\Pi}_k^{\rm BDM}\bm u-\bm u),
		\bm{\Pi}_{k-1}^{\partial}\bm v_h-\widehat{\bm v}_h
		\rangle_{\partial\mathcal T_h},\\
	E_3&:=(\bm{\Pi}_k^{\rm BDM}\bm u - \bm u,\nabla\cdot\mathbb G_h)_{\mathcal T_h}.
	\end{align*}
	By employing the triangle inequality and the properties of \Cref{projection properties}, we can establish the following inequalities:
	\begin{align}
		|E_1|&\leq\|\bm{\Pi}_{k-1}^o\mathbb L-\mathbb L\|_{\partial\mathcal T_h}(\|\bm{\Pi}_{k-1}^{\partial}\bm v_h-\bm v_h\|_{\partial\mathcal T_h} +\|\bm{\Pi}_{k-1}^{\partial}\bm v_h-\widehat{\bm v}_h\|_{\partial\mathcal T_h})\nonumber\\
		&\lesssim\nu h^{k}|\bm u|_{k+1}(\|\nabla\bm v_h\|_{\mathcal T_h}+\|\eta^{1/2}(\bm{\Pi}_{k-1}^{\partial}\bm v_h-\widehat{\bm v}_h)\|_{\partial\mathcal T_h})\nonumber\\
		&\leq\nu^{1/2}h^{k}|\bm u|_{k+1}\interleave(\mathbb G_h,\bm v_h,\widehat{\bm v}_h)\interleave.\label{BDM_E1}
	\end{align}
	By applying the trace inequality and \eqref{BDM error}, we obtain:

	\begin{align}
		|E_2|&\leq\nu\|\eta^{1/2}(\bm{\Pi}_{k}^{\rm BDM}\bm u-\bm u)\|_{\partial\mathcal T_h}\|\eta^{1/2}(\bm{\Pi}_{k-1}^{\partial}\bm v_h-\widehat{\bm v}_h)\|_{\partial\mathcal T_h}\nonumber\\
		&\lesssim \nu^{1/2} h^{k}|\bm u|_{k+1} \interleave(\mathbb G_h,\bm v_h,\widehat{\bm v}_h)\interleave,\label{BDM_E2}
	\end{align}
		{
			and
				\begin{align}
				|E_3|&\leq\sum_{T\in\mathcal T_h}\|\bm{\Pi}_{k}^{\rm BDM}\bm u-\bm u\|_{T}\|\nabla\cdot\mathbb G_h\|_T
				\leq\sum_{T\in\mathcal T_h}h_T^{-1}\|\bm{\Pi}_{k}^{\rm BDM}\bm u-\bm u\|_{T}\|\mathbb G_h\|_T
				\nonumber\\
				&\lesssim \nu^{1/2} h^{k}|\bm u|_{k+1} \interleave(\mathbb G_h,\bm v_h,\widehat{\bm v}_h)\interleave.\label{BDM_E3}
			\end{align}
			
}
	Hence, we have established the conclusion by combining \eqref{BDM_E1} and \eqref{BDM_E3}.
	
\end{proof}

\begin{theorem}\label{th-35} 
	Let $ (\mathbb L, \bm u,p) \in [H^k(\Omega)]^{d\times d} \times [H^{k+1}(\Omega)]^d\times (H^{\epsilon}(\Omega)\cap L_0^2(\Omega))$  {with $0<\epsilon<1/2$}, and $(\mathbb G_h, \bm v_h, \widehat{\bm v}_h, q_h) \in \bm{\Sigma}_h \times \bm V_h^{\bm 0} \times \widehat{\bm V}_h^{\bm 0} \times Q_h$ be the solutions to \Cref{2.2,2.4}, respectively. Then, we can establish the following error estimate:
	\begin{align}\label{BDM Priori Error}
		\begin{split}
		&\interleave(\bm{\Pi}_{k-1}^o\mathbb L-\mathbb L_h,\bm{\Pi}_{k}^{\rm BDM}\bm u-\bm u_h,\bm{\Pi}_{k-1}^{\partial}\bm u-\widehat{\bm u}_h)\interleave \\
		&\qquad+\|\Pi_{k}^op-p_h\|_{\mathcal{T}_h}
		\lesssim \nu^{1/2} h^{k}|\bm u|_{k+1}. 
		\end{split}
	\end{align}
	
\end{theorem}
\begin{proof}
	Firstly, by applying the triangle inequality, we obtain:
	\begin{equation}\label{2.24}
		\begin{aligned}
			\interleave(\mathbb L_h-\mathbb L,\bm u_h-\bm u,\widehat{\bm u}_h-\bm u)\interleave\leq&\interleave(\bm{\Pi}_{k-1}^o\mathbb L-\mathbb L_h, \bm{\Pi}_{k}^{\rm BDM}\bm u-\bm u_h, \bm{\Pi}_{k-1}^{\partial}\bm u-\widehat{\bm u}_h)\interleave\\
			&+\interleave(\bm{\Pi}_{k-1}^o\mathbb L-\mathbb L, \bm{\Pi}_{k}^{\rm BDM}\bm u-\bm u, \bm{\Pi}_{k-1}^{\partial}\bm u-{\bm u})\interleave.\\
		\end{aligned}
	\end{equation}
	Using \Cref{Theorem 2.3} and \Cref{BDM_E,BDM error equation lemma}, we can deduce:
	\begin{equation}\label{2.1.5}
		\begin{aligned}
			\interleave(\bm{\Pi}_{k-1}^o\mathbb L-\mathbb L_h,&\bm{\Pi}_{k}^{\rm BDM}\bm u-\bm u_h,\bm{\Pi}_{k-1}^{\partial}\bm u-\widehat{\bm u}_h)\interleave+\|\Pi_{k}^op-p_h\|_{\mathcal{T}_h}\\
			&\lesssim \sup_{\bm 0\neq (\mathbb G_h,\bm v_h,\widehat{\bm v}_h,q_h)\in \bm{\Sigma}_h\times\bm V_h^{\bm 0}\times \widehat{\bm V}_h^{\bm 0}\times Q_h}\frac{	\mathscr R(\bm\xi_h^{\mathbb L},\bm\xi_h^{\bm u},\bm\xi_h^{\widehat{\bm u}},\xi_h^{p};
				\mathbb G_h,\bm v_h,\widehat{\bm v}_h,q_h)}{	\interleave(\mathbb G_h,\bm v_h,\widehat{\bm v}_h)\interleave+\|q_h\|_{\mathcal{T}_h}}\\
			&=\sup_{\bm 0\neq (\mathbb G_h,\bm v_h,\widehat{\bm v}_h,q_h)\in \bm{\Sigma}_h\times\bm V_h^{\bm 0}\times \widehat{\bm V}_h^{\bm 0}\times Q_h}\frac{E(\mathbb L,\bm u;\mathbb G_h,\bm v_h,\widehat{\bm v}_h)}{	\interleave(\mathbb G_h,\bm v_h,\widehat{\bm v}_h)\interleave+\|q_h\|_{\mathcal{T}_h}}\\
			&\lesssim\nu^{1/2} h^{k}|\bm u|_{k+1},
		\end{aligned}
	\end{equation}
which finishes our proof.
\end{proof}
\endgroup

By the above results and the triangle inequality we can get the following estimations.
\begin{corollary}\label{estimete for p}
	Let $(\mathbb L, \bm u, p) \in [H^k(\Omega)]^{d\times d} \times [H^{k+1}(\Omega)]^d \times  (H^{k}(\Omega)\cap L_0^2(\Omega))$, and $(\mathbb G_h, \bm v_h, \widehat{\bm v}_h, q_h) \in \bm{\Sigma}_h \times \bm V_h^{\bm 0} \times \widehat{\bm V}_h^{\bm 0} \times Q_h$ be the solutions to \Cref{2.2,2.4}, respectively. Then the following error estimate holds:
	\begin{align*}
		\nu^{-1/2}\|\mathbb L-\mathbb L_h\|_{\mathcal T_h}
		+\nu^{1/2}\|\nabla(\bm u-\bm u_h)\|_{\mathcal T_h}
		&\leq Ch^k\nu^{1/2}|\bm u|_{k+1},\\
		\|p_h-p\|_{\mathcal T_h}&\leq Ch^k(\nu^{1/2}|\bm u|_{k+1}+ |p|_{k}).
	\end{align*}
\end{corollary}

The $L^2$ error of velocity can be analyzed using standard duality arguments. We initiate the process by considering the dual problem associated with \Cref{2.2}, where we seek a solution $(\bm\xi, \theta)$ that satisfies:
\begin{subequations}
	\begin{align}\label{3.24}
		\nu^{-1}\bm\phi -\nabla \bm \xi&=0&\text{ in}\quad\Omega,\\
		-\nabla\cdot\bm\phi+\nabla \theta&=\boldsymbol{\Pi}_k^{\rm BDM} \boldsymbol{u}-\boldsymbol{u}_{h}&\text{ in}\quad\Omega,\\
		\nabla \cdot \bm\xi&=0 &\text{ in}\quad\Omega,\\
		\bm\xi&=\mathbf{0}&\text{ on}\quad\partial\Omega.
	\end{align}
\end{subequations}

We introduce a regularity hypothesis, assuming the existence of a constant $C>0$ independent of $\nu$, such that:
\begin{align}\label{regularity}
	\nu|\boldsymbol{\xi}|_{2}+|\theta|_{1} \leq C\left\|\boldsymbol{\Pi}_k^{\rm BDM} \boldsymbol{u}-\boldsymbol{u}_{h}\right\|_{0}.
\end{align}

It's worth noting that \eqref{regularity} is satisfied when $\Omega$ is convex in $\mathbb{R}^2$ \cite{MR0404849} or when $\Omega$ is a convex polyhedron in $\mathbb{R}^3$ \cite{MR977489}.

\begin{theorem}\label{theorem 2.8}
	Let $(\mathbb L, \bm u, p) \in [H^{k}(\Omega)]^{d\times d} \times [H^{k+1}(\Omega)]^d \times  (H^{\epsilon}(\Omega)\cap L_0^2(\Omega))$  {with $0<\epsilon<1/2$} represent the solution to \Cref{2.2}, and let $(\mathbb G_h, \bm v_h, \widehat{\bm v}_h, q_h) \in \bm{\Sigma}_h \times \bm V_h^{\bm 0} \times \widehat{\bm V}_h^{\bm 0} \times Q_h$ be the solution to \Cref{2.4}. In this context, the error estimate is given by:
	\begin{align}
		\|\bm u-\bm u_h\|_0\le C h^{k+1}|\bm u|_{k+1}.
	\end{align}
\end{theorem}

\begingroup
\begin{proof}
	Similarly to \Cref{BDM error equation lemma}, for all $(\mathbb G_h, \bm v_h, \widehat{\bm v}_h, q_h) \in \bm{\Sigma}_h \times \bm V_h^{\bm 0} \times \widehat{\bm V}_h^{\bm 0} \times Q_h$, we have the following relation:
	\begin{align}\label{2.18}
		\mathscr B(\bm{\Pi}_{k-1}^o\bm\phi,&\bm{\Pi}_k^{\rm BDM}\bm \xi,\bm\Pi_{k-1}^{\partial}{\bm \xi},\Pi_{k-1}^o\theta;
		\mathbb G_h,\bm v_h,\widehat{\bm v}_h,q_h)	 \nonumber\\&=(\boldsymbol{\Pi}_k^{\rm BDM} \boldsymbol{u}-\boldsymbol{u}_{h},\bm v_h)	+E(\bm\phi,\bm \xi;\bm v_h,\widehat{\bm v}_h).
	\end{align}
	Now, taking $\bm v_h = \bm{\Pi}_k^{\rm BDM}\bm u - \bm u_h$ and noticing the following identity:
	\begin{align*}
		\mathscr B(\mathbb L,-\bm u,p;\mathbb G,\bm v,q)=\mathscr B(\mathbb G,-\bm v,q;\mathbb L,\bm u,p).
	\end{align*}
	Hence, 
	\begin{align*}
		\|\boldsymbol{\Pi}_k^{\rm BDM}\boldsymbol{u}-\boldsymbol{u}_{h}\|_0^2
		&=\mathscr B(\bm{\Pi}_{k-1}^o\bm\phi,\bm{\Pi}_k^{\rm BDM}\bm \xi,\bm\Pi_{k-1}^{\partial}{\bm \xi},\Pi_{k-1}^o\theta;
		\bm\xi_h^{\mathbb L},\bm\xi_h^{\bm u},\bm\xi_h^{\widehat{\bm u}},-\xi_h^{p})\\
		& \quad -E(\bm\phi,\bm \xi;\bm\xi_h^{\mathbb L},\bm\xi_h^{\bm u},\bm\xi_h^{\widehat{\bm u}})\\
		&=\mathscr B(\bm\xi_h^{\mathbb L},-\bm\xi_h^{\bm u},\bm\xi_h^{\widehat{\bm u}},-\xi_h^{p};\bm{\Pi}_{k-1}^o\bm\phi,-\bm{\Pi}_k^{\rm BDM}\bm \xi,\bm\Pi_{k-1}^{\partial}{\bm \xi},\Pi_{k-1}^o\theta)\\
		&\quad -E(\bm\phi,\bm \xi;\bm\xi_h^{\mathbb L},\bm\xi_h^{\bm u},\bm\xi_h^{\widehat{\bm u}}).
	\end{align*}
	Since $\bm\xi\in[H_0^1(\Omega)]^d$ and $\bm u\in[{H}^2(\Omega)]^d$, we can simplify the term involving the projections as follows:
	\begin{align*}
		\langle(\bm{\Pi}_{k-1}^o\mathbb L-\mathbb L)\bm n,\bm\Pi_{k-1}^{\partial}{\bm \xi}\rangle_{\partial\mathcal T_h}&=\nu\langle(\bm{\Pi}_{k-1}^o\nabla\bm u-\nabla\bm u)\bm n , \bm\Pi_{k-1}^{\partial}{\bm \xi}\rangle_{\partial\mathcal T_h}\\
		& =\nu\langle\bm{\Pi}_{k-1}^o\nabla\bm u\cdot\bm n, \bm\Pi_{k-1}^{\partial}{\bm \xi}\rangle_{\partial\mathcal T_h}\\&=\nu\langle\bm{\Pi}_{k-1}^o\nabla\bm u\cdot\bm n, {\bm \xi}\rangle_{\partial\mathcal T_h}\\
		&=\nu\langle(\bm{\Pi}_{k-1}^o\nabla\bm u-\nabla\bm u)\bm n , {\bm \xi}\rangle_{\partial\mathcal T_h}.
	\end{align*}
	By the approximation properties of projections, we have 
	\begin{align}
		E(\mathbb L,\bm u;
		\bm{\Pi}^o_{k-1}\bm\phi,
		\bm{\Pi}_k^{\rm BDM}\bm \xi,\bm\Pi_{k-1}^{\partial}{\bm \xi})
		&=	\langle(\bm{\Pi}_{k-1}^o\mathbb L-\mathbb L)\bm n,\bm\Pi_{k-1}^{\partial}{\bm \xi} -\bm{\Pi}_k^{\rm BDM}\bm \xi\rangle_{\partial\mathcal T_h}\nonumber \\
		&\quad+\nu\langle \eta(\bm{\Pi}_k^{\rm BDM}\bm u-\bm u),
		\bm{\Pi}_{k-1}^{\partial}(\bm{\Pi}_k^{\rm BDM}\bm \xi-{\bm \xi})
		\rangle_{\partial\mathcal T_h}\nonumber\\
		&\quad +(\bm{\Pi}_k^{\rm BDM}\bm u - \bm u,\nabla\cdot	\bm{\Pi}^o_{k-1}\bm\phi)_{\mathcal T_h}  \nonumber\\
		&\leq\nu\left\|\bm{\Pi}_{k-1}^o\nabla\bm u-\nabla\bm u\right\|_{\partial\mathcal T_h}\left\|\bm{\Pi}_k^{\rm BDM}\bm \xi-\bm \xi \right\|_{\partial\mathcal T_h}\nonumber\\
		&\quad+\nu \left\|\eta(\bm{\Pi}_k^{\rm BDM}\bm u-\bm u) \right\|_{\partial\mathcal T_h}\left\|\bm{\Pi}_k^{\rm BDM}\bm \xi-\bm \xi \right\|_{\partial\mathcal T_h}\nonumber\\
		&\quad  +\|\bm{\Pi}_k^{\rm BDM}\bm u - \bm u\|_{\mathcal T_h}\|\nabla\cdot	\bm{\Pi}^o_{k-1}\bm\phi\|_{\mathcal T_h}  \nonumber\\
		&\lesssim \nu h^{k+1}|\bm u|_{k+1}(|\bm\phi|_1+|\bm \xi|_2).\label{2.19}
	\end{align}
	Similarly, and from  \eqref{2.1.5}  we have
	\begin{align*}
		&E(\bm\phi,\bm \xi;
		\bm{\Pi}_{k-1}^o\mathbb L-\mathbb L_h,
		\bm{\Pi}_k^{\rm BDM}\bm u-\bm u_h,\bm{\Pi}_{k-1}^{\partial}\bm u-\widehat{\bm u}_h)\\
		&\qquad=\langle(\bm{\Pi}_{k-1}^o\bm\phi-\bm\phi)\bm n,\bm{\Pi}_{k-1}^{\partial}\bm u-\widehat{\bm u}_h-(\bm{\Pi}_k^{\rm BDM}\bm u-\bm u_h)\rangle_{\partial\mathcal T_h} \\
		&\qquad\quad+\nu\langle \eta(\bm{\Pi}_k^{\rm BDM}\bm \xi-\bm \xi),
		\bm{\Pi}_{k-1}^{\partial}(\bm{\Pi}_k^{\rm BDM}\bm u-{\bm u_h})-(\bm{\Pi}_{k-1}^{\partial}\bm u-\widehat{\bm u}_h)
		\rangle_{\partial\mathcal T_h}\\
		&\qquad\quad  +(\bm{\Pi}_k^{\rm BDM}\bm \xi - \bm \xi,\nabla\cdot(	\bm{\Pi}_{k-1}^o\mathbb L-\mathbb L_h))_{\mathcal T_h} \\
		&\qquad\leq\nu\|\bm{\Pi}_{k-1}^o\nabla\bm \xi-\nabla\bm\xi\|_{\partial\mathcal T_h}(\|\bm{\Pi}_{k-1}^{\partial}\bm u-\widehat{\bm u}_h\|_{\partial\mathcal T_h}+\|\bm{\Pi}_k^{\rm BDM}\bm u-\bm u_h\|_{\partial\mathcal T_h} )\\
		&\qquad\quad+ \nu\|\eta(\bm{\Pi}_k^{\rm BDM}\bm \xi-\bm \xi) \|_{\partial\mathcal T_h} ( \|\bm{\Pi}_{k-1}^{\partial}\bm u-\widehat{\bm u}_h \|_{\partial\mathcal T_h}+\|\bm{\Pi}_k^{\rm BDM}\bm u-\bm u_h\|_{\partial\mathcal T_h} )\\
		&\qquad\quad +\sum_{T\in\mathcal T_h}\|\bm{\Pi}_k^{\rm BDM}\bm \xi - \bm \xi\|_{T}\|\nabla\cdot(	\bm{\Pi}_{k-1}^o\mathbb L-\mathbb L_h)\|_{T} \\
		&\qquad\lesssim\nu  |\bm\xi|_2 h^{k+1}|\bm u|_{k+1}.
	\end{align*}
	Then, together with (\ref{2.19}) and regularity hypothesis (\ref{regularity}), it holds
	\begin{align*}
		\|\bm{\Pi}_k^{\rm BDM}\bm u-\bm u_h \|_0 \lesssim h^{k+1}|\bm u|_{k+1}.
	\end{align*}
	By triangle inequality and the approximation property of the $\bm\Pi_k^{\rm BDM}$, we can get the $ L^2$ error estimation for velocity.
\end{proof}

\section{RT-based \texorpdfstring{$H(\mathrm{div})$}{H(div)}-conforming HDG method with discontinuous traces}\label{sec4}
Let's refine the earlier finite element space mentioned in \Cref{first Scheme}. Likewise, it's imperative to demonstrate its well-posedness and analyze its error. Prior to delving into these aspects, let's introduce some well-established findings to aid in the analysis.
\endgroup

We first define the local Raviart--Thomas space of order $k\geq0$ as introduced in \cite{NEDELECJ.C1986Anfo, raviart1975mixed}.
\begin{align*}
	\bm{RT}_k(T)=[\mathcal{\bm P}_k(T)]^d+\bm x\mathcal{\bm P}_k(T).
\end{align*} 

The subsequent lemmas provide fundamental properties of the spaces $\bm{RT}_k(T)$, as referenced in \cite{brezzi2008mixed}.

\begin{lemma}\label{RT_lemma}
	For any $\bm v\in\bm{RT}_k(T),  \nabla\cdot\bm v|_T=0$ indicates $\bm v\in[\mathcal{\bm P}_k(T)]^d$.
\end{lemma}
\begin{proof}
	See \cite[Page.9, Lemma 3.1]{brezzi2008mixed}.
\end{proof}

\begin{lemma}
	For any $T\in\mathcal{T}_h$ and given $\bm v\in[H^1(T)]^d$, there exists a unique $\bm\Pi^{RT}_{k}\bm v\in\bm{RT}_k(T)$ such that
	\begin{align}
		&\langle \bm\Pi^{RT}_{k}\bm v\cdot\bm n_E,\omega_k\rangle_E=\left\langle \bm v\cdot\bm n_E,\omega_k\right\rangle_E,&&\forall \omega_k\in\mathcal{P}_k(E),E\subset\partial T,\label{4.2}\\
		&(\bm\Pi^{RT}_{k}\bm v,\bm\omega_{k-1})_T=(\bm v,\bm\omega_{k-1})_T,&&\forall \bm\omega_{k-1}\in[\mathcal{P}_{k-1}(T)]^d.\label{4.3}
	\end{align}\par
	If $k=0,\bm\Pi^{RT}_{k}\bm v$ is determined only by \eqref{4.2}. Moreover, it holds the estimates for the ${RT}$ projection:
	\begin{align}
		\|\bm v- \bm\Pi^{RT}_{k}\bm v \|_{0,T} \lesssim h^m_T|\bm v|_{m,T},\quad \forall \boldsymbol{v} \in[H^m(T)]^d, 1 \leq m \leq k+1.\label{4.4}
	\end{align}
	Furthermore, $\boldsymbol{\Pi}_k^{{RT}}$ holds the commutativity property:
	\begin{align}
		(\nabla\cdot\bm\Pi^{RT}_{k}\bm v,p)_T=(\nabla\cdot\bm v,p)_T,\quad\forall p\in\mathcal P_{k}(T).\label{4.5}
	\end{align}
\end{lemma}
\begin{proof}
	References for equations \eqref{4.2} and \eqref{4.3} can be found in \cite[Page 10, Lemma 3.2]{brezzi2008mixed}, for \eqref{4.4}, refer to \cite[Page 13, Proposition 2.3]{du2019invitation}, and for \eqref{4.5}, see \cite[Page 11, Proposition 2.2]{du2019invitation}.
\end{proof}

\subsection{The Discrete Scheme}
We maintain consistent notation while addressing the subsequent finite element spaces. When $ k\ge 1$, we define
\begin{align*}
	\bm{\Sigma}_h&=\{\mathbb L_h\in \bm L^2(\Omega):\mathbb L_h|_T\in [\mathcal P_{k}(T)]^{d\times d},\forall T\in\mathcal T_h\},\\
	\bm {V^g}_h&=\{\bm v_h\in  {\bm H({\rm div};\Omega)}: \bm v_h|_T\in\bm{RT}_k(T),\forall T\in\mathcal T_h,\bm v_h\cdot \bm n |_{\partial \Omega}=\Pi_k^{\partial} (\bm g\cdot\bm n)\},\\
	\widehat{\bm V}_h^{\bm g}&=\{\widehat{\bm v}_h\in L^2(\mathcal E_h):\widehat{\bm v}_h|_E\in [\mathcal P_{k}(E)]^d,\forall E\in\mathcal E_h,	\widehat{\bm v}_h|_{\Gamma}=\bm{\Pi}_{k}^{\partial}\bm g
	\},\\
	Q_h&=\{q_h\in L_0^2(\Omega):q_h|_T\in P_{k}(T),\forall T\in\mathcal T_h\}.
\end{align*}

\begingroup
Find $(\mathbb L_h,\bm u_h,\widehat{\bm u}_h,q_h)\in \bm{\Sigma}_h\times\bm {V^g}_h\times \widehat{\bm V}_h^{\bm g}\times Q_h$ such that 
\begin{subequations}
	\begin{align}
		\nu^{-1}(\mathbb L_h,\mathbb G_h)_{\mathcal T_h}+(\bm u_h,\nabla\cdot\mathbb G_h)_{\mathcal T_h}
		-\langle \widehat{\bm u}_h ,\mathbb G_h\bm n \rangle_{\partial\mathcal T_h}&= 0,\\
		-(\nabla\cdot\mathbb L_h,\bm v_h)_{\mathcal T_h}
		+\langle\mathbb L_h\bm n,\widehat{\bm v}_h \rangle_{\partial\mathcal T_h} -(p_h,\nabla \cdot\bm v_h)_{\mathcal T_h} \nonumber
		&\\
		+\nu\langle \eta(\bm{\Pi}_{k}^{\partial}\bm u_h-\widehat{\bm u}_h ),
		\bm{\Pi}_{k}^{\partial}\bm v_h-\widehat{\bm v}_h
		\rangle_{\partial\mathcal T_h}
		&=(\bm f,\bm v_h)_{\mathcal T_h},\\
		(\nabla\cdot\bm u_h,q_h)_{\mathcal T_h}&=0
	\end{align}
\end{subequations}
{
	holds for all $(\mathbb G_h,\bm v_h,\widehat{\bm v}_h,q_h)\in \bm{\Sigma}_h\times\bm V_h^{\bm 0}\times \widehat{\bm V}_h^{\bm 0}\times Q_h$,
	To simplify the notations we  define
}
\begin{align*}
	&\mathscr R(\mathbb L_h,\bm u_h,\widehat{\bm u}_h,p_h;
	\mathbb G_h,\bm v_h,\widehat{\bm v}_h, q_h)\\
	&\qquad= 	\nu^{-1}(\mathbb L_h,\mathbb G_h)_{\mathcal T_h}+(\bm u_h,\nabla\cdot\mathbb G_h)_{\mathcal T_h}
	-\langle \widehat{\bm u}_h,\mathbb G_h\bm n \rangle_{\partial\mathcal T_h}\\
	&\qquad\quad	-(\nabla\cdot\mathbb L_h,\bm v_h)_{\mathcal T_h}	+\langle\mathbb L_h\bm n,\widehat{\bm v}_h \rangle_{\partial\mathcal T_h} -(p_h,\nabla \cdot\bm v_h)_{\mathcal T_h}\\
	&\qquad\quad	+\nu\langle \eta(\bm{\Pi}_{k}^{\partial}\bm u_h-\widehat{\bm u}_h),
	\bm{\Pi}_{k}^{\partial}\bm v_h-\widehat{\bm v}_h
	\rangle_{\partial\mathcal T_h}
	+ (\nabla\cdot\bm u_h,q_h)_{\mathcal T_h}.
\end{align*}
The corresponding HDG method reads:  {Find $(\mathbb L_h,\bm u_h,\widehat{\bm u}_h,q_h)\in \bm{\Sigma}_h\times\bm {V^g}_h\times \widehat{\bm V}_h^{\bm g}\times Q_h$ such that }
\begin{align}\label{second scheme}
	\mathscr R(\mathbb L_h,\bm u_h,\widehat{\bm u}_h,p_h;
	\mathbb G_h,\bm v_h,\widehat{\bm v}_h,q_h)=	(\bm f,\bm v_h)_{ \mathcal T_h}
\end{align}
{holds for all $(\mathbb G_h,\bm v_h,\widehat{\bm v}_h,q_h)\in \bm{\Sigma}_h\times\bm V_h^{\bm 0}\times \widehat{\bm V}_h^{\bm 0}\times Q_h$.}

\subsection{Well-Posedness of the Scheme}
We next establish well-posedness of the RT-based scheme \eqref{second scheme}. As in \Cref{sec3}, we work with the norm
\begin{align}\label{norm2}
	\interleave(\mathbb G_h,\bm v_h,\widehat{\bm v}_h) \interleave ^2	
	&:=\nu^{-1}\|\mathbb G_h\|_{\mathcal T_h}^2
	+\nu\|\nabla\bm v_h\|_{\mathcal T_h}^2
	+\nu\|\eta^{1/2}(\bm\Pi_{k}^{\partial}\bm v_h-\widehat{\bm v}_h)\|_{\partial\mathcal T_h}^2.
\end{align}
As in \Cref{check norm1}, one verifies that \eqref{norm2} defines a norm on $\bm{\Sigma}_h\times\bm V_h^{\bm 0}\times\widehat{\bm V}_h^{\bm 0}$.
\begin{theorem}\label{LBB2}
	We have 
	\begin{align}
		\sup_{\bm 0\neq (\mathbb G_h,\bm v_h,\widehat{\bm v}_h,q_h)\in \bm{\Sigma}_h\times\bm V_h^{\bm 0}\times \widehat{\bm V}_h^{\bm 0}\times Q_h}\frac{\mathscr R(\mathbb L_h,\bm u_h,\widehat{\bm u}_h,p_h;
			\mathbb G_h,\bm v_h,\widehat{\bm v}_h,q_h)}{	\interleave(\mathbb G_h,\bm v_h,\widehat{\bm v}_h)\interleave+\|q_h\|_{\mathcal T_h}}
		\ge \beta( 	\interleave(\mathbb L_h,\bm u_h,\widehat{\bm u}_h)\interleave+\|p_h\|_{\mathcal T_h}  ),\label{RTlbb1}\\
		\sup_{\bm 0\neq (\mathbb G_h,\bm v_h,\widehat{\bm v}_h,q_h)\in \bm{\Sigma}_h\times\bm V_h^{\bm 0}\times \widehat{\bm V}_h^{\bm 0}\times Q_h}\frac{	\mathscr R(
			\mathbb G_h,\bm v_h,\widehat{\bm v}_h,q_h;\mathbb L_h,\bm u_h,\widehat{\bm u}_h,p_h)}{	\interleave(\mathbb G_h,\bm v_h,\widehat{\bm v}_h)\interleave+\|q_h\|_{\mathcal T_h}}
		\ge \beta( 	\interleave(\mathbb L_h,\bm u_h,\widehat{\bm u}_h)\interleave+\|p_h\|_{\mathcal T_h}  ).\label{4.10}
	\end{align}
\end{theorem}
\begin{proof}
	The proof is carried out in four steps.
	
	\noindent\textbf{Step 1.} Considering $\left( \mathbb G_1,\boldsymbol{v}_1,\widehat{\boldsymbol{v}}_{1},q_h\right)=\left( \mathbb{L}_{h}, \boldsymbol{u}_{h}, \widehat{\boldsymbol{u}}_{h},  p_{h} \right)\in  \bm{\Sigma}_h\times\bm V_h^{\bm 0}\times \widehat{\bm V}_h^{\bm 0}\times Q_h$, and utilizing the definition of $\mathscr R$, it readily follows that
	\begin{align}	
		&\mathscr{R}\left(\mathbb{L}_{h}, \boldsymbol{u}_{h}, \widehat{\boldsymbol{u}}_{h}, p_{h} ;\mathbb{L}_{h}, \boldsymbol{u}_{h}, \widehat{\boldsymbol{u}}_{h}, p_{h}\right)=\nu^{-1}\left\|\mathbb L_h\right\|_{\mathcal T_h}^2+\nu\|\eta^{1/2}(\bm\Pi_{k}^{\partial}\bm u_h-\widehat{\bm u}_h)\|_{\partial\mathcal T_h}^2.\label{4.11}
	\end{align}
	\textbf{Step 2.} Subsequently, considering $\left( \mathbb G_2,\boldsymbol{v}_2,\widehat{\boldsymbol{v}}_{2},q_h\right)=\left(-\nu\nabla\bm u_h , \bm 0, \bm 0,0 \right)\in  \bm{\Sigma}_h\times\bm V_h^{\bm 0}\times \widehat{\bm V}_h^{\bm 0}\times Q_h $ and employing Green's formula leads to
	\begin{align*}
		\mathscr{R}&\left(\mathbb{L}_{h}, \boldsymbol{u}_{h}, \widehat{\boldsymbol{u}}_{h}, p_{h} ;-\nu\nabla\bm u_h , \bm 0, \bm 0,0 \right)\\
		&=	(\mathbb L_h,-\nabla\bm u_h)_{\mathcal T_h}-\nu(\bm u_h,\Delta\bm u_h)_{\mathcal T_h}
		+\nu\langle \widehat{\bm u}_h,\nabla\bm u_h\bm n \rangle_{\partial\mathcal T_h}\\
		&=	(\mathbb L_h,-\nabla\bm u_h)_{\mathcal T_h}+\nu(\nabla\bm u_h,\nabla\bm u_h)_{\mathcal T_h}-\nu\langle {\bm u}_h- \widehat{\bm u}_h,\nabla\bm u_h \bm n \rangle_{\partial\mathcal T_h}.
	\end{align*}
	Applying the Cauchy-Schwarz inequality to the first two terms, we establish that
	\begin{align}\label{4.12}
		(\mathbb L_h,-\nabla\bm u_h)_{\mathcal T_h}+\nu(\nabla\bm u_h,\nabla\bm u_h)_{\mathcal T_h}\geq \nu\left\| \nabla\bm u_h\right\|^2_{\mathcal T_h}-\left\| \mathbb L_h\right\|_{\mathcal T_h} \left\| \nabla\bm u_h\right\|_{\mathcal T_h}.
	\end{align}
	The definition of $\bm\Pi^{\partial}_{k}$ in (\ref{L2}) indicates
	\begin{align}
		-\nu\langle {\bm u}_h- \widehat{\bm u}_h,\nabla\bm u_h \bm n \rangle_{\partial\mathcal T_h}&=-\nu\langle \bm\Pi^{\partial}_{k}{\bm u}_h- \widehat{\bm u}_h,\nabla\bm u_h \bm n \rangle_{\partial\mathcal T_h}\nonumber\\
		&\quad+\nu\langle \bm\Pi^{\partial}_{k}{\bm u}_h- {\bm u}_h,\nabla\bm u_h \bm n \rangle_{\partial\mathcal T_h}\nonumber\\
		&=-\nu\langle \bm\Pi^{\partial}_{k}{\bm u}_h- \widehat{\bm u}_h,\nabla\bm u_h \bm n \rangle_{\partial\mathcal T_h}\nonumber\\
		&\geq-\nu C_1\|\eta^{1/2}( \bm\Pi^{\partial}_{k}{\bm u}_h- \widehat{\bm u}_h)\| _{\partial\mathcal T_h}\left\| \nabla\bm u_h\right\|_{\mathcal T_h}.\label{4.13}
	\end{align}
	In conjunction with \eqref{4.12} and \eqref{4.13}, upon employing Young's inequality, we derive
	\begin{align}
		\mathscr{R}&\left(\mathbb{L}_{h},\boldsymbol{u}_{h}, \widehat{\boldsymbol{u}}_{h}, p_{h} ;-\nu\nabla\bm u_h , \bm 0, \bm 0,0 \right)\nonumber\\
		& 	\geq\nu\left\| \nabla\bm u_h\right\|^2_{\mathcal T_h}-\left\| \mathbb L_h\right\|_{\mathcal T_h} \left\| \nabla\bm u_h\right\|_{\mathcal T_h}-\nu C_1\|\eta^{1/2}( \bm\Pi^{\partial}_{k}{\bm u}_h- \widehat{\bm u}_h)\| _{\partial\mathcal T_h}\left\| \nabla\bm u_h\right\|_{\mathcal T_h}\nonumber\\
		&\geq\nu\left\| \nabla\bm u_h\right\|^2_{\mathcal T_h}-(\dfrac{1}{4}\nu\left\| \nabla\bm u_h\right\|^2_{\mathcal T_h}+\nu^{-1}\left\| \mathbb L_h\right\|^2_{\mathcal T_h})\nonumber\\
		&\quad-(\dfrac{1}{4}\nu\left\| \nabla\bm u_h\right\|^2_{\mathcal T_h}+\nu C_1^2\|\eta^{1/2}( \bm\Pi^{\partial}_{k}{\bm u}_h- \widehat{\bm u}_h)\|^2 _{\partial\mathcal T_h})\nonumber\\
		&\geq\dfrac{1}{2}\nu\left\| \nabla\bm u_h\right\|^2_{\mathcal T_h}-(\nu^{-1}\left\| \mathbb L_h\right\|^2_{\mathcal T_h}+\nu C_1^2\|\eta^{1/2}( \bm\Pi^{\partial}_{k}{\bm u}_h- \widehat{\bm u}_h)\|^2 _{\partial\mathcal T_h}).\label{4.14}
	\end{align}
	\textbf{Step 3.} Given $p_h \in L_{0}^{2}(\Omega)$, the continuous inf-sup condition ensures the existence of $\boldsymbol{v} \in H_{0}^{1}(\Omega)$, such that
	\begin{align}
		\nabla \cdot \boldsymbol{v}=p_h,\qquad|\boldsymbol{v}|_{H^1(\Omega)} \lesssim\left\|p_h\right\|_{L^2(\Omega)}.
	\end{align}
	Subsequently, by considering  $\left( \mathbb G_3,\boldsymbol{v}_{ 3},\widehat{\boldsymbol{v}}_{3},q_h\right)=\left(0 , \boldsymbol{\Pi}_k^{RT}\bm v, \bm{\Pi}_{k}^{\partial}\bm v,0 \right)\in  \bm{\Sigma}_h\times\bm V_h^{\bm 0}\times \widehat{\bm V}_h^{\bm 0}\times Q_h  $, we obtain
	\begin{align*}
		&\mathscr{R}\left(\mathbb{L}_{h}, \boldsymbol{u}_{h}, \widehat{\boldsymbol{u}}_{h}, p_{h} ;0,\boldsymbol{\Pi}_k^{RT}\bm v ,\bm{\Pi}_{k}^{\partial}\bm v, 0\right)\\
		&\qquad= -(\nabla\cdot\mathbb L_h,\boldsymbol{\Pi}_k^{RT}\bm v)_{\mathcal T_h}	+\langle\mathbb L_h\bm n,\bm{\Pi}_{k}^{\partial}\bm v \rangle_{\partial\mathcal T_h} -(p_h,\nabla \cdot\boldsymbol{\Pi}_k^{RT}\bm v)_{\mathcal T_h}\\
		&\qquad\quad+\nu\langle \eta(\bm{\Pi}_{k}^{\partial}\bm u_h-\widehat{\bm u}_h ),
		\bm{\Pi}_{k}^{\partial}\boldsymbol{\Pi}_k^{RT} \boldsymbol{v} -\bm{\Pi}_{k}^{\partial}\bm v\rangle_{\partial\mathcal T_h}\\
		&\qquad= (\mathbb L_h,\nabla\boldsymbol{\Pi}_k^{RT}\bm v)_{\mathcal T_h}+\langle\mathbb L_h\bm n,\bm{\Pi}_{k}^{\partial}\bm v-\boldsymbol{\Pi}_k^{RT}\bm v \rangle_{\partial\mathcal T_h}-(p_h,\nabla \cdot\boldsymbol{\Pi}_k^{RT}\bm v)_{\mathcal T_h}\\
		&\qquad\quad+\nu\langle \eta(\bm{\Pi}_{k}^{\partial}\bm u_h-\widehat{\bm u}_h ),
		\bm{\Pi}_{k}^{\partial}\boldsymbol{\Pi}_k^{RT} \boldsymbol{v} -\bm{\Pi}_{k}^{\partial}\bm v\rangle_{\partial\mathcal T_h}.\qquad\text{(Green's formula)}
	\end{align*}
	Utilizing the properties of $\boldsymbol{\Pi}_k^{RT}$ and $\bm{\Pi}_{k}$ as described in \eqref{4.4} and detailed in \Cref{projection properties}, it is evident that
	\begin{align}
		(\mathbb L_h,&\nabla\boldsymbol{\Pi}_k^{RT}\bm v)_{\mathcal T_h}\nonumber\\&\leq\left\|\mathbb L_h\right\|_{\mathcal T_h}\|\nabla\boldsymbol{\Pi}_k^{RT}\bm v\|_{\mathcal T_h}&&\text{(Cauchy-Schwarz inequality)}\nonumber\\
		&\leq \left\|\mathbb L_h\right\|_{\mathcal T_h}( \| \nabla \bm v \|_{\mathcal T_h}+\|\nabla\boldsymbol{\Pi}_k^{RT}\bm v-\nabla \bm v\|_{\mathcal T_h}) &&\text{(triangle inequality)} \nonumber \\
		&=\left\|\mathbb L_h\right\|_{\mathcal T_h}( |\bm v|_{ 1}+\|\nabla(\boldsymbol{\Pi}_k^{RT}\bm v-\bm v)\|_{\mathcal T_h}).\label{4.17} &&\text{(definition of semi-norm)}
	\end{align}
	Turning our attention to the second term within the parentheses of \eqref{4.17}, upon employing the triangle inequality, we obtain
	\begin{align*}
		&\left\|\nabla(\boldsymbol{\Pi}_k^{RT}\bm v-\bm v)\right\|_{\mathcal T_h}\nonumber\\
		&\qquad \leq\left\| \nabla(\bm v-\bm{\Pi}_{k} \bm v)\right\| _{\mathcal T_h}+\left\| \nabla(\boldsymbol{\Pi}_k^{RT}\bm v-\bm{\Pi}_{k}\bm v)\right\| _{\mathcal T_h}\nonumber \\
		&\qquad\leq c_1\left|\bm v\right|_{ 1}+c_2\left\|\mathfrak{h}^{-1} \bm{\Pi}_{k}(\boldsymbol{\Pi}_k^{RT}\bm v-\bm v)\right\| _{\mathcal T_h}\nonumber &&\text{(by \eqref{2.4'} and inverse inequality)}\\
		&\qquad\leq c_1|\bm v|_{ 1}+ {c_2}\left\|\mathfrak{h}^{-1}(\boldsymbol{\Pi}_k^{RT}\bm v-\bm v)\right\|_{\mathcal T_h}\nonumber &&\text{(by property of projection \eqref{2.7})}\\
		&\qquad\leq c_1|\bm v|_{ 1}+ {c_3}|\bm v|_{ 1}\nonumber &&\text{(by property of projection \eqref{4.4})}\\
		&\qquad=C|\bm v|_{ 1}.
	\end{align*}
	Here, $\bm{\Pi}_{k}: [L^2(T)]^d\rightarrow [{\mathcal{P}}_k(T)]^d$ denotes the $L^2$ projection. The constants $c_1$, $c_2$, and $ {c_3}$ are positive and remain independent of $h$, $h_T$, and $\nu$. In conjunction with \eqref{4.17}, this yields
	\begin{align}\label{4.18}
		(\mathbb L_h,\nabla\boldsymbol{\Pi}_k^{RT}\bm v)_{\mathcal T_h}
		\leq(1+C)\left\|\mathbb L_h\right\|_{\mathcal T_h}|\bm v|_{ 1} 
		\leq C_2\left\|\mathbb L_h\right\|_{\mathcal T_h}\left\|p_h\right\|_{\mathcal T_h}.
	\end{align}
	Furthermore, leveraging the commutativity property of $\boldsymbol{\Pi}_k^{RT}$ as outlined in \eqref{4.5}, we obtain
	\begin{align}\label{4.19}
		(p_h,\nabla \cdot\boldsymbol{\Pi}_k^{RT}\bm v)_{\mathcal T_h}=(p_h,\nabla \cdot\bm v)_{\mathcal T_h}=\|p_h\|^2_{\mathcal T_h}.
	\end{align}
	Utilizing the trace inequality and approximations provided in \eqref{2.4'} and \eqref{4.4}, we arrive at
	\begin{align}\label{4.20}
		\langle\mathbb L_h\bm n,\bm{\Pi}_{k}^{\partial}\bm v-\bm{\Pi}_k^{RT}\bm v \rangle_{\partial\mathcal T_h}&=\langle\mathbb L_h\bm n,\bm{\Pi}_{k}^{\partial}\bm v-\bm v \rangle_{\partial\mathcal T_h}+	\langle\mathbb L_h\bm n,\bm v-\bm{\Pi}_k^{RT}\bm v \rangle_{\partial\mathcal T_h}\nonumber\\
		& \leq\|\mathfrak{h}^{1/2}\mathbb L_h\|_{\partial\mathcal T_h}( \|\mathfrak{h}^{-1/2}(\bm{\Pi}_{k}^{\partial}\bm v-\bm v) \|_{\partial\mathcal T_h}\\
		&\quad +\|\mathfrak{h}^{-1/2} (\bm v-\boldsymbol{\Pi}_k^{RT}\bm v)\|_{\partial\mathcal T_h}) \nonumber\\
		&\leq \|\mathbb L_h\|_{\mathcal T_h}(                                                                                                                                                                                                                       C_1|\bm v |_{ 1}+C_2|\bm v |_{ 1})\nonumber\\
		&\leq C_3\left\|\mathbb L_h\right\|_{\mathcal T_h}\left\|p_h\right\|_{\mathcal T_h}.
	\end{align}
	Similar to the previous analysis, we also have
	\begin{align}\label{4.21}
		\nu\langle \eta(\bm{\Pi}_{k}^{\partial}\bm u_h-\widehat{\bm u}_h ), &
		\bm{\Pi}_{k}^{\partial}\boldsymbol{\Pi}_k^{RT} \boldsymbol{v} -\bm{\Pi}_{k}^{\partial}\bm v\rangle_{\partial\mathcal T_h}\nonumber\\
		&\leq\nu\|\eta(\bm{\Pi}_{k}^{\partial}\bm u_h-\widehat{\bm u}_h )\|_{\partial\mathcal T_h}\|	\bm{\Pi}_{k}^{\partial}(\boldsymbol{\Pi}_k^{RT} \boldsymbol{v} -\bm v)\|_{\partial\mathcal T_h}\nonumber\\
		&\leq C\nu\|\eta^{1/2}(\bm{\Pi}_{k}^{\partial}\bm u_h-\widehat{\bm u}_h )\|_{\partial\mathcal T_h}|\bm v|_{ 1}\nonumber\\
		&\leq C_4\nu\|\eta^{1/2}(\bm{\Pi}_{k}^{\partial}\bm u_h-\widehat{\bm u}_h )\|_{\partial\mathcal T_h}\|p_h\|_{\mathcal T_h}.
	\end{align}
	Together with \eqref{4.18}, \eqref{4.19}, \eqref{4.20} and \eqref{4.21}, we obtain
	\begin{align}
		\mathscr{R}&\left(\mathbb{L}_{h}, \boldsymbol{u}_{h}, \widehat{\boldsymbol{u}}_{h}, p_{h} ;0,-\boldsymbol{\Pi}_k^{RT}\bm v ,-\bm{\Pi}_{k}^{\partial}\bm v, 0\right)\nonumber\\
		&\qquad\geq\left\|p_h\right\|^2_{\mathcal T_h}-(C_2+C_3)\left\|\mathbb L_h\right\|_{\mathcal T_h}\left\|p_h\right\|_{\mathcal T_h}-C_4\nu\|\eta^{1/2}(\bm{\Pi}_{k}^{\partial}\bm u_h-\widehat{\bm u}_h )\|_{\partial\mathcal T_h}\|p_h\|_{\mathcal T_h}\nonumber\\
		&\qquad\geq\dfrac{1}{2}\left\|p_h\right\|^2_{\mathcal T_h}-(C_2+C_3)^2\left\|\mathbb L_h\right\|_{\mathcal T_h}^{2}-C_4^2\nu^2\|\eta^{1/2}(\bm{\Pi}_{k}^{\partial}\bm u_h-\widehat{\bm u}_h )\|^2_{\partial\mathcal T_h}\nonumber\\
		&\qquad=\dfrac{1}{2}\left\|p_h\right\|^2_{\mathcal T_h}-C_5^2\left\|\mathbb L_h\right\|_{\mathcal T_h}^{2}-C_4^2\nu^2\|\eta^{1/2}(\bm{\Pi}_{k}^{\partial}\bm u_h-\widehat{\bm u}_h )\|^2_{\partial\mathcal T_h}, \label{4.22}
	\end{align}
	where $C_1,C_2,C_3,C_4,C_5=C_2+C_3$ is positive constants independent of $h,h_T$ and the fluid viscosity coefficient $\nu$.
	
	\noindent\textbf{Step 4.} By considering the implications of \eqref{4.11}, \eqref{4.14}, and \eqref{4.22}, when testing $(\mathbb G_h,\bm v_h,\widehat{\bm v}_h,q_h)$ with $\mathbb G_h=(1+C_1^2+\nu C_4^2+\nu C_5^2)\mathbb G_1+\mathbb G_2+\mathbb G_3\in\bm{\Sigma}_h$, $\boldsymbol{v}_h=(1+C_1^2+\nu C_4^2+\nu C_5^2)\boldsymbol{v}_1+\boldsymbol{v}_2+\boldsymbol{v}_3\in\bm V_h^{\bm 0}$, $\widehat{\boldsymbol{v}}_{h}=(1+C_1^2+\nu C_4^2+\nu C_5^2)\widehat{\boldsymbol{v}}_{1}+\widehat{\boldsymbol{v}}_{2}+\widehat{\boldsymbol{v}}_{3}\in\widehat{\bm V}_h^{\bm 0}$, and $q_h=p_h\in Q_h$, and subsequently summing the results, we obtain
	\begin{align*}
		\mathscr R(\mathbb L_h,\bm u_h,\widehat{\bm u}_h,p_h;
		\mathbb G_h,\bm v_h,\widehat{\bm v}_h,q_h)\geq\dfrac{1}{2}( \interleave(\mathbb L_h,\bm u_h,\widehat{\bm u}_h)\interleave^2+\|p_h\|_{\mathcal T_h}^2).
	\end{align*}
	As in the proof of \Cref{Theorem 2.3}, it remains to control the norm of the constructed test tuple. We have
	\[
	\interleave(\mathbb G_2,\bm v_2,\widehat{\bm v}_2)\interleave
	=\nu^{1/2}\|\nabla\bm u_h\|_{\mathcal T_h}
	\le \interleave(\mathbb L_h,\bm u_h,\widehat{\bm u}_h)\interleave.
	\]
	Furthermore, the stability of the RT projection, the trace inequality, and the continuous inf-sup estimate imply
	\[
	\nu^{1/2}\|\nabla\bm v_3\|_{\mathcal T_h}
	+\nu^{1/2}\|\eta^{1/2}(\bm\Pi_{k}^{\partial}\bm v_3-\widehat{\bm v}_3)\|_{\partial\mathcal T_h}
	\lesssim \|p_h\|_{\mathcal T_h},
	\]
	so that $\interleave(\mathbb G_3,\bm v_3,\widehat{\bm v}_3)\interleave\lesssim \|p_h\|_{\mathcal T_h}$. Consequently,
	\[
	\interleave(\mathbb G_h,\bm v_h,\widehat{\bm v}_h)\interleave+\|q_h\|_{\mathcal T_h}
	\lesssim \interleave(\mathbb L_h,\bm u_h,\widehat{\bm u}_h)\interleave+\|p_h\|_{\mathcal T_h},
	\]
	and division yields \eqref{RTlbb1}.
\end{proof}
\subsection{Error Estimations}\label{section4.3}
Similar to \Cref{section3.3}, this section presents the derivation of the equation in \eqref{RT error equation} and analyzes the a priori error estimates as well as $L^2$ error estimation for velocity. We define the discrete errors based on the following quantities:
\begin{align*}
	\bm\xi_h^{\mathbb L}=\bm{\Pi}_{k}^o\mathbb L-\mathbb L_h,\quad
	\bm\xi_h^{\bm u}=\bm{\Pi}_{k}^{RT}\bm u-\bm u_h,\quad
	\bm\xi_h^{\widehat{\bm u}}=\bm{\Pi}_{k}^{\partial}\bm u-\widehat{\bm u}_h,\quad
	\xi_h^{p}={\Pi}_{k}^{o}p-p_h.
\end{align*}
\begin{lemma}\label{lemma4.4}
	Considering $(\mathbb L,\bm u,p)\in [H^{\epsilon}(\Omega)]^{d\times d}\times[ H^1(\Omega)]^d\times(H^{\epsilon}(\Omega)\cap L_0^2(\Omega))$  {with $0<\epsilon<1/2$} as the solution to \eqref{2.2}, the equation holds for all $(\mathbb G_h,\bm v_h,\widehat{\bm v}_h,q_h)\in \bm{\Sigma}_h\times\bm V_h^{\bm 0}\times \widehat{\bm V}_h^{\bm 0}\times Q_h$:
	\begin{align}\label{RT error equation}
		\mathscr R(\bm{\Pi}_{k}^o\mathbb L,\bm{\Pi}_k^{RT}\bm u,\bm\Pi_{k}^{\partial}{\bm u},\Pi_{k}^op;
		\mathbb G_h,\bm v_h,\widehat{\bm v}_h,q_h)=	(\bm f,\bm v_h)_{\mathcal T_h}	+E(\mathbb L,\bm u;\bm v_h,\widehat{\bm v}_h),
	\end{align}
	where
	\begin{align*}
		E(\mathbb L,\bm u;\bm v_h,\widehat{\bm v}_h)=	\langle(\bm{\Pi}_{k}^o\mathbb L-\mathbb L)\bm n,\widehat{\bm v}_h -\bm v_h\rangle_{\partial\mathcal T_h} 
		+\nu\langle \eta(\bm{\Pi}_k^{RT}\bm u-\bm u),
		\bm{\Pi}_{k}^{\partial}\bm v_h-\widehat{\bm v}_h
		\rangle_{\partial\mathcal T_h}.
	\end{align*}
	
\end{lemma}
\begin{proof}	
	First, we note that
	\begin{align}\label{4.23}
		\bm{\Pi}_k^{RT}\bm u|_T\in[\mathcal{P}_k(T)]^d,\quad\forall T\in\mathcal{T}_h.
	\end{align}
	For any $T\in\mathcal{T}_h$ and $q_h\in\mathcal{P}_k(T)$, utilizing the commutativity property in \eqref{4.5}, we observe that
	\begin{align*}
		(\nabla\cdot\bm{\Pi}_k^{RT}\bm u,q_h)_T=(\nabla\cdot\bm u,q_h)_{ T}=0.
	\end{align*}
	This implies $\nabla\cdot\bm{\Pi}_k^{RT}\bm u=0.$ Thus, \eqref{4.23} is a consequence derived from \Cref{RT_lemma}. According to the definition of $\mathscr R$, we have
	\begin{align*}
		&\mathscr R(\bm{\Pi}_{k}^o\mathbb L,\bm{\Pi}_k^{RT}\bm u,\bm\Pi_{k}^{\partial}{\bm u},\Pi_{k}^op;
		\mathbb G_h,\bm v_h,\widehat{\bm v}_h,q_h)\\
		&\qquad= 	\nu^{-1}( \bm{\Pi}_{k}^o\mathbb L,\mathbb G_h)_{\mathcal T_h}+(\bm{\Pi}_k^{RT}\bm u,\nabla\cdot\mathbb G_h)_{\mathcal T_h}
		-\langle \bm\Pi_{k}^{\partial}{\bm u},\mathbb G_h\bm n \rangle_{\partial\mathcal T_h}\\
		&\qquad\quad	-(\nabla\cdot\bm{\Pi}_{k}^o\mathbb L,\bm v_h)_{\mathcal T_h}	+\langle\bm{\Pi}_{k}^o\mathbb L\bm n,\widehat{\bm v}_h \rangle_{\partial\mathcal T_h} -(\Pi_{k}^op,\nabla \cdot\bm v_h)_{\mathcal T_h}\\
		&\qquad\quad	+\nu\langle\eta\bm{\Pi}_{k}^{\partial}(\bm{\Pi}_k^{RT}\bm u-{\bm u}),
		\bm{\Pi}_{k}^{\partial}\bm v_h-\widehat{\bm v}_h
		\rangle_{\partial\mathcal T_h}
		+ (\nabla\cdot\bm{\Pi}_k^{RT}\bm u,q_h)_{\mathcal T_h}.
	\end{align*}
	For convenience, we define
	\begin{align*}
		&R_1:=\nu^{-1}( \bm{\Pi}_{k}^o\mathbb L,\mathbb G_h)_{\mathcal T_h}+(\bm{\Pi}_k^{RT}\bm u,\nabla\cdot\mathbb G_h)_{\mathcal T_h}-\langle \bm\Pi_{k}^{\partial}{\bm u},\mathbb G_h\bm n \rangle_{\partial\mathcal T_h},\\
		&R_2:=-(\nabla\cdot\bm{\Pi}_{k}^o\mathbb L,\bm v_h)_{\mathcal T_h}	+\langle\bm{\Pi}_{k}^o\mathbb L\bm n,\widehat{\bm v}_h \rangle_{\partial\mathcal T_h} -(\Pi_{k}^op,\nabla \cdot\bm v_h)_{\mathcal T_h},\\
		&R_3:=\nu\langle\eta\bm{\Pi}_{k}^{\partial}(\bm{\Pi}_k^{RT}\bm u-{\bm u}),
		\bm{\Pi}_{k}^{\partial}\bm v_h-\widehat{\bm v}_h
		\rangle_{\partial\mathcal T_h}
		+ (\nabla\cdot\bm{\Pi}_k^{RT}\bm u,q_h)_{\mathcal T_h}.
	\end{align*}
	By the orthogonality of $\bm{\Pi}_{k}^{\partial}$ in \eqref{L2} and property of $\bm{\Pi}_{k}^{RT}$  in \eqref{4.3}, it indicates
	\begin{align}
		R_1=&\nu^{-1}( \bm{\Pi}_{k}^o\mathbb {L},\mathbb G_h)_{\mathcal T_h}+(\bm{\Pi}_k^{RT}\bm u,\nabla\cdot\mathbb G_h)_{\mathcal T_h}-\langle \bm\Pi_{k}^{\partial}{\bm u},\mathbb G_h\bm n \rangle_{\partial\mathcal T_h}\nonumber\\
		=&\nu^{-1}( \mathbb L,\mathbb G_h)_{\mathcal T_h}+(\bm u,\nabla\cdot\mathbb G_h)_{\mathcal T_h}-\langle {\bm u},\mathbb G_h\bm n \rangle_{\partial\mathcal T_h}\nonumber\\
		=&\nu^{-1}( \mathbb L,\mathbb G_h)_{\mathcal T_h}-(\nabla\bm u,\mathbb{G}_h)_{\mathcal T_h}\label{R1}.
	\end{align}
	From \Cref{itegration} and  the orthogonality of $L_2$ projection, we obtain
	\begin{align}
		R_2=&(\bm{\Pi}_{k}^o\mathbb L,\nabla\bm v_h)_{\mathcal T_h}-\left\langle \bm{\Pi}_{k}^o\mathbb L\bm n,\bm v_h\right\rangle_{\partial\mathcal T_h} +\langle\bm{\Pi}_{k}^o\mathbb L\bm n,\widehat{\bm v}_h \rangle_{\partial\mathcal T_h}\nonumber\\
		&-(\Pi_{k}^op-p,\nabla \cdot\bm v_h)_{\mathcal T_h}-(p,\nabla \cdot\bm v_h)_{\mathcal T_h}\nonumber\\
		=& -(\nabla\cdot(\mathbb{L}-p\bm I),\bm v_h)_{\mathcal T_h}+\left\langle \bm{\Pi}_{k}^o\mathbb L\bm n,\widehat{\bm v}_h-\bm v_h\right\rangle_{\partial\mathcal T_h}-(\Pi_{k}^op-p,\nabla \cdot\bm v_h)_{\mathcal T_h}\nonumber \\
		&-(\Pi_{k}^op-p,\nabla \cdot\bm v_h)_{\mathcal T_h}-(\nabla\cdot\mathbb{L},\bm v_h)_{\mathcal T_h}+(\nabla p,\bm v_h)_{\mathcal T_h}\nonumber\\
		=&\langle(\bm{\Pi}_{k}^o\mathbb L-\mathbb L)\bm n,\widehat{\bm v}_h-\bm v_h\rangle_{\partial\mathcal T_h} -(\nabla\cdot\mathbb{L}-p\bm I,\bm v_h)_{\mathcal T_h}.\label{R2}
	\end{align}
	By the definition of $\bm{\Pi}_{k}^o$ and the commutativity property of $\boldsymbol{\Pi}_k^{RT}$ in \eqref{4.5}, it indicates
	\begin{align}
		R_3=\nu\langle\eta(\bm{\Pi}_k^{RT}\bm u-{\bm u}),
		\bm{\Pi}_{k}^{\partial}\bm v_h-\widehat{\bm v}_h
		\rangle_{\partial\mathcal T_h}.\label{R3}
	\end{align}
	Therefore, adding \eqref{R1},\eqref{R2} and \eqref{R3} with the fact that $-\nabla\cdot(\mathbb L - p\bm I)=\bm f$  {in $L^2$ product}, we have finished the proof.
\end{proof}
\endgroup


\begingroup
We prove the following error estimates by using standard interpolation approximation.
\begin{lemma}\label{RT_E}
	For $(\mathbb L,\bm u)\in[{H}^{k}(\Omega)]^{d\times d}\times[ H^{k+1}(\Omega)]^{d}$, it holds the following estimate
	\begin{align}
		|E(\mathbb L,\bm u;\bm v_h,\widehat{\bm v}_h)|\lesssim \nu^{1/2} h^{k}|\bm u|_{k+1}\interleave(\mathbb G_h,\bm v_h,\widehat{\bm v}_h)\interleave,\quad \forall (\bm v_h,\widehat{\bm v}_h)\in  {\bm {V}_h^{\bm 0}}\times \widehat{\bm V}_h^{\bm 0}.
	\end{align}
\end{lemma}
\begin{proof}
	We define
	\begin{align*}		 
		E_1:=\langle(\bm{\Pi}_{k}^o\mathbb L-\mathbb L)\bm n,\widehat{\bm v}_h -\bm v_h\rangle_{\partial\mathcal T_h},\quad E_2:=\nu\langle \eta(\bm{\Pi}_k^{RT}\bm u-\bm u),
		\bm{\Pi}_{k}^{\partial}\bm v_h-\widehat{\bm v}_h
		\rangle_{\partial\mathcal T_h}.
	\end{align*}
	By triangle inequality and \Cref{projection properties}, it holds
	\begin{align}
		|E_1|&\leq\|\bm{\Pi}_{k}^o\mathbb L-\mathbb L\|_{\partial\mathcal T_h}(\|\bm{\Pi}_{k}^{\partial}\bm v_h-\bm v_h\|_{\partial\mathcal T_h} +\|\bm{\Pi}_{k}^{\partial}\bm v_h-\widehat{\bm v}_h\|_{\partial\mathcal T_h})\nonumber\\
		&\lesssim\nu h^{k}|\bm u|_{k+1}(\|\nabla\bm v_h\|_{\mathcal T_h}+\|\eta^{1/2}(\bm{\Pi}_{k}^{\partial}\bm v_h-\widehat{\bm v}_h)\|_{\partial\mathcal T_h})\nonumber\\
		&\leq\nu^{1/2}h^{k}|\bm u|_{k+1}\interleave(\mathbb G_h,\bm v_h,\widehat{\bm v}_h)\interleave.\label{4.28}
	\end{align}
	By the trace inequality and \eqref{4.4}, we have
	\begin{align}
		|E_2|&\leq\nu^{1/2}\|\eta^{1/2}(\bm{\Pi}_{k}^{RT}\bm u-\bm u)\|_{\partial\mathcal T_h}\|\eta^{1/2}(\bm{\Pi}_{k-1}^{\partial}\bm v_h-\widehat{\bm v}_h)\|_{\partial\mathcal T_h}\nonumber\\
		&\lesssim \nu^{1/2} h^{k}|\bm u|_{k+1} \interleave(\mathbb G_h,\bm v_h,\widehat{\bm v}_h)\interleave.\label{4.29}
	\end{align}
	Then the conclusion have been proved by adding \eqref{4.28} and \eqref{4.29} together.
\end{proof}

\begin{theorem} Let $ (\mathbb L,\bm u,p)\in[{H}^{k}(\Omega)]^{d\times d}\times[ H^{k+1}(\Omega)]^{d}\times (H^{\epsilon}(\Omega)\cap L_0^2(\Omega))$ with $0<\epsilon<1/2$ and $(\mathbb G_h,\bm v_h,\widehat{\bm v}_h,q_h)\in \bm{\Sigma}_h\times\bm V_h^{\bm 0}\times \widehat{\bm V}_h^{\bm 0}\times Q_h$ be the solution to \Cref{2.2} and \Cref{second scheme}, respectively. Then it holds the error estimate
	\begin{align}\label{RT Priori Error}
		\begin{split}
		&\interleave(\bm{\Pi}_{k}^o\mathbb L-\mathbb L_h,\bm{\Pi}_{k}^{RT}\bm u-\bm u_h,\bm{\Pi}_{k}^{\partial}\bm u-\widehat{\bm u}_h)\interleave\\
		&\qquad+\|\Pi_{k-1}^op-p_h\|_{\mathcal{T}_h}\le C\nu^{1/2} h^{k}|\bm u|_{k+1}.
		\end{split}
	\end{align}
\end{theorem}
\begin{proof}
We use \Cref{LBB2}, it holds
	\begin{equation}\label{4.26}
		\begin{aligned}
			\interleave(\bm{\Pi}_{k}^o\mathbb L-\mathbb L_h,&\bm{\Pi}_{k}^{RT}\bm u-\bm u_h,\bm{\Pi}_{k}^{\partial}\bm u-\widehat{\bm u}_h)\interleave+\|\Pi_{k-1}^op-p_h\|_{\mathcal{T}_h}\\
			&\lesssim \sup_{ {{\bm 0\neq (\mathbb G_h,\bm v_h,\widehat{\bm v}_h,q_h)\in \bm{\Sigma}_h\times\bm V_h^{\bm 0}\times \widehat{\bm V}_h^{\bm 0}\times Q_h}}}\frac{	\mathscr B(\bm\xi_h^{\mathbb L},\bm\xi_h^{\bm u},\bm\xi_h^{\widehat{\bm u}},\xi_h^{p};
				\mathbb G_h,\bm v_h,\widehat{\bm v}_h,q_h)}{	\interleave(\mathbb G_h,\bm v_h,\widehat{\bm v}_h)\interleave+\|q_h\|_{\mathcal{T}_h}}\\
			&=\sup_{ {{\bm 0\neq (\mathbb G_h,\bm v_h,\widehat{\bm v}_h,q_h)\in \bm{\Sigma}_h\times\bm V_h^{\bm 0}\times \widehat{\bm V}_h^{\bm 0}\times Q_h}}}\frac{E(\mathbb L,\bm u;\bm v_h,\widehat{\bm v}_h)}{	\interleave(\mathbb G_h,\bm v_h,\widehat{\bm v}_h)\interleave+\|q_h\|_{\mathcal{T}_h}}\\
			&\lesssim\nu^{1/2} h^{k}|\bm u|_{k+1},
		\end{aligned}
	\end{equation}
which leads to the intended conclusion.
\end{proof}
Same as \Cref{estimete for p}, by triangle inequality,  we obtain
\begin{corollary}
	Let $(\mathbb L,\bm u,p)\in[{H}^{k}(\Omega)]^{d\times d}\times[ H^{k+1}(\Omega)]^{d}\times  (H^{k}(\Omega)\cap L_0^2(\Omega))$ and $(\mathbb G_h,\bm v_h,\widehat{\bm v}_h,q_h)\in \bm{\Sigma}_h\times\bm V_h^{\bm 0}\times \widehat{\bm V}_h^{\bm 0}\times Q_h$ be the solution to \Cref{2.2} and \Cref{second scheme}, respectively. Then it holds the error estimate
	\begin{align*}
	\nu^{-1/2}\|\mathbb L-\mathbb L_h\|_{\mathcal T_h}
	+\nu^{1/2}\|\nabla(\bm u-\bm u_h)\|_{\mathcal T_h}
	&\leq Ch^k\nu^{1/2}|\bm u|_{k+1},\\
	\|p_h-p\|_{\mathcal T_h}&\leq Ch^k(\nu^{1/2}|\bm u|_{k+1}+ |p|_{k}).
\end{align*}
\end{corollary}

Similarly to \eqref{3.24}, we consider the following dual problem of \Cref{2.2}: find a solution $(\bm\xi, \theta)$ satisfying
\begin{subequations}
	\begin{align}\label{4.34}
		\nu^{-1}\bm\phi -\nabla \bm \xi&=0&\text{in}\quad\Omega,\\
		-\nabla\cdot\bm\phi+\nabla \theta&=\boldsymbol{\Pi}_k^{RT} \boldsymbol{u}-\boldsymbol{u}_{h}&\text{in}\quad\Omega,\\
		\nabla \cdot \bm\xi&=0 &\text{in}\quad\Omega,\\
		\bm\xi&=\mathbf{0}&\text{on}\quad\partial\Omega.
	\end{align}
\end{subequations}
Here $\bm\phi=\nu\nabla\bm \xi$, and $\bm u$ and $\bm u_h$ denote the exact and discrete velocities associated with \Cref{Ori_problem2} and \Cref{second scheme}, respectively. We assume the following regularity hypothesis: there exists $C>0$ such that
\begin{align}\label{regularity2}
	\nu|\boldsymbol{\xi}|_{2}+|\theta|_{1} \leq C\left\|\boldsymbol{\Pi}_k^{RT} \boldsymbol{u}-\boldsymbol{u}_{h}\right\|_{0}.
\end{align}
We notice that \eqref{regularity2} holds when $\Omega\subset\mathbb{R}^2$ is convex\cite{MR0404849} or when $\Omega\subset\mathbb{R}^3$ is a convex polyhedron \cite{MR977489}.
\begin{theorem}\label{theorem 4.7}
	Let $(\mathbb L,\bm u,p)\in[H^{k}(\Omega)]^{d\times d}\times[ H^{k+1}(\Omega)]^{d}\times  (H^{\epsilon}(\Omega)\cap L_0^2(\Omega))$  {with $0<\epsilon<1/2$} be the solution to \Cref{2.2} and $(\mathbb G_h,\bm v_h,\widehat{\bm v}_h,q_h)\in \bm{\Sigma}_h\times\bm V_h^{\bm 0}\times \widehat{\bm V}_h^{\bm 0}\times Q_h$ be the solution to \Cref{2.4}. It holds the error estimate
	\begin{align}
		\|\bm u-\bm u_h\|_0\le C h^{k+1}|\bm u|_{k+1}.
	\end{align}
\end{theorem}

\begin{proof}
	Similar to \Cref{lemma4.4}, for all $(\mathbb G_h,\bm v_h,\widehat{\bm v}_h,q_h)\in \bm{\Sigma}_h\times\bm V_h^{\bm g}\times \widehat{\bm V}_h^{\bm g}\times Q_h,$ it holds
	\begin{align}\label{4.31}
		\mathscr R(\bm{\Pi}_{k}^o\varPhi,&\bm{\Pi}_k^{RT}\bm \xi,\bm\Pi_{k}^{\partial}{\bm \xi},\Pi_{k}^o\theta;
		\mathbb G_h,\bm v_h,\widehat{\bm v}_h,q_h)	 \nonumber\\=&(\boldsymbol{\Pi}_k^{RT} \boldsymbol{u}-\boldsymbol{u}_{h},\bm v_h)	+E(\bm\phi,\bm \xi;\bm v_h,\widehat{\bm v}_h).
	\end{align}
	Taking $\bm v_h=\bm{\Pi}_k^{RT}\bm u-\bm u_h$  and noticing the fact $\mathscr R(\mathbb L,-\bm u,p;\mathbb G,\bm v,q)=\mathscr R(\mathbb G,-\bm v,q;\mathbb L,\bm u,p)$, we have 
	\begin{align*}
		\left\| \boldsymbol{\Pi}_k^{RT} \boldsymbol{u}-\boldsymbol{u}_{h}\right\| _0^2
		=&\mathscr R(\bm{\Pi}_{k}^o\varPhi,\bm{\Pi}_k^{RT}\bm \xi,\bm\Pi_{k}^{\partial}{\bm \xi},\Pi_{k}^o\theta;
		\bm\xi_h^{\mathbb L},\bm\xi_h^{\bm u},\bm\xi_h^{\widehat{\bm u}},-\xi_h^{p})\\
		&-E(\bm\phi,\bm \xi;\bm{\Pi}_k^{RT}\bm u-\bm u_h,\bm{\Pi}_{k}^{\partial}\bm u-\widehat{\bm u}_h)\\
		=&\mathscr R(\bm\xi_h^{\mathbb L},-\bm\xi_h^{\bm u},\bm\xi_h^{\widehat{\bm u}},-\xi_h^{p};\bm{\Pi}_{k}^o\varPhi,-\bm{\Pi}_k^{RT}\bm \xi,\bm\Pi_{k}^{\partial}{\bm \xi},\Pi_{k}^o\theta)\\
		&-E(\bm\phi,\bm \xi;\bm{\Pi}_k^{RT}\bm u-\bm u_h,\bm{\Pi}_{k}^{\partial}\bm u-\widehat{\bm u}_h).
	\end{align*}
	This together with \Cref{lemma4.4}, yields
	\begin{align*}
		\left\| \boldsymbol{\Pi}_k^{RT} \boldsymbol{u}-\boldsymbol{u}_{h}\right\| _0^2=&E(\mathbb L,\bm u;\bm{\Pi}_k^{RT}\bm \xi,\bm\Pi_{k}^{\partial}{\bm \xi})-E(\bm\phi,\bm \xi;\bm{\Pi}_k^{RT}\bm u-\bm u_h,\bm{\Pi}_{k}^{\partial}\bm u-\widehat{\bm u}_h).
	\end{align*}
	Since $\bm\xi\in {[ H_0^1(\Omega)]^d},\bm u\in[{H}^2(\Omega)]^d$, we have 
	\begin{align*}
		\langle(&\bm{\Pi}_{k}^o\mathbb L-\mathbb L)\bm n,\bm\Pi_{k}^{\partial}{\bm \xi}\rangle_{\partial\mathcal T_h}\\&=\nu\langle(\bm{\Pi}_{k}^o\nabla\bm u-\nabla\bm u)\bm n , \bm\Pi_{k}^{\partial}{\bm \xi}\rangle_{\partial\mathcal T_h}=\nu\langle\bm{\Pi}_{k}^o\nabla\bm u\cdot\bm n, \bm\Pi_{k}^{\partial}{\bm \xi}\rangle_{\partial\mathcal T_h}\\&=\nu\langle\bm{\Pi}_{k}^o\nabla\bm u\cdot\bm n, {\bm \xi}\rangle_{\partial\mathcal T_h}=\nu\langle\bm{\Pi}_{k}^o\nabla\bm u, {\bm \xi}\rangle_{\partial\mathcal T_h}=\nu\langle(\bm{\Pi}_{k}^o\nabla\bm u-\nabla\bm u)\bm n , {\bm \xi}\rangle_{\partial\mathcal T_h} {.}
	\end{align*}
	By the approximation properties of projections, we have
	\begin{align}\label{4.32}
		E(\mathbb L,\bm u;\bm{\Pi}_k^{RT}\bm \xi,\bm\Pi_{k}^{\partial}{\bm \xi})
		&=	\langle(\bm{\Pi}_{k}^o\mathbb L-\mathbb L)\bm n,\bm\Pi_{k}^{\partial}{\bm \xi} -\bm{\Pi}_k^{RT}\bm \xi\rangle_{\partial\mathcal T_h}\nonumber \\
		&\qquad+\nu\langle \eta(\bm{\Pi}_k^{RT}\bm u-\bm u),
		\bm{\Pi}_{k}^{\partial}(\bm{\Pi}_k^{RT}\bm \xi-{\bm \xi})
		\rangle_{\partial\mathcal T_h}\nonumber\\
		&\leq\nu\left\|\bm{\Pi}_{k}^o\nabla\bm u-\nabla\bm u\right\|_{0,\partial\mathcal T_h}\left\|\bm{\Pi}_k^{RT}\bm \xi-\bm \xi \right\|_{0,\partial\mathcal T_h}\nonumber\\
		&\qquad+\nu \left\|\eta(\bm{\Pi}_k^{RT}\bm u-\bm u) \right\|_{0,\partial\mathcal T_h}\left\|\bm{\Pi}_k^{RT}\bm \xi-\bm \xi \right\|_{0,\partial\mathcal T_h}\nonumber\\
		&\lesssim \nu h^{k+1}|\bm u|_{k+1}|\bm \xi|_2.
	\end{align}
	Similarly,  from \eqref{4.26} we have
	\begin{align*}
		&E(\bm\phi,\bm \xi;\bm{\Pi}_k^{RT}\bm u-\bm u_h,\bm{\Pi}_{k}^{\partial}\bm u-\widehat{\bm u}_h)\\
		&\qquad=\langle(\bm{\Pi}_{k}^o\bm\phi-\bm\phi)\bm n,\bm{\Pi}_{k-1}^{\partial}\bm u-\widehat{\bm u}_h-(\bm{\Pi}_k^{RT}\bm u-\bm u_h)\rangle_{\partial\mathcal T_h} \\
		&\qquad\quad+\nu\langle \eta(\bm{\Pi}_k^{RT}\bm \xi-\bm \xi),
		\bm{\Pi}_{k-1}^{\partial}(\bm{\Pi}_k^{RT}\bm u-{\bm u_h})-(\bm{\Pi}_{k}^{\partial}\bm u-\widehat{\bm u}_h)
		\rangle_{\partial\mathcal T_h}\\
		&\qquad\leq\nu\|\bm{\Pi}_{k}^o\nabla\bm \xi-\nabla\bm\xi \|_{\partial\mathcal T_h}( \|\bm{\Pi}_{k}^{\partial}\bm u-\widehat{\bm u}_h \|_{\partial\mathcal T_h}+\left\|\bm{\Pi}_k^{RT}\bm u-\bm u_h \right\|_{\partial\mathcal T_h} )\\
		&\qquad\quad+ \nu\|\eta(\bm{\Pi}_k^{RT}\bm \xi-\bm \xi) \|_{\partial\mathcal T_h}( \|\bm{\Pi}_{k}^{\partial}\bm u-\widehat{\bm u}_h \|_{\partial\mathcal T_h}+\|\bm{\Pi}_k^{RT}\bm u-\bm u_h \|_{\partial\mathcal T_h} )\\
		&\qquad\lesssim\nu  |\bm\xi|_2 h^{k+1}|\bm u|_{k+1}.
	\end{align*}
	Then, together with \eqref{4.32} and regularity hypothesis (\ref{regularity2}), it holds
	\begin{align}
		\left\|\bm{\Pi}_k^{RT}\bm u-\bm u_h \right\|_0 \lesssim h^{k+1}|\bm u|_{k+1}.
	\end{align}
	By triangle inequality and the approximation property of the $\bm\Pi_k^{RT}$, we can get the $ L^2$ error estimation for velocity.
\end{proof}
\section{\texorpdfstring{$H(\mathrm{div})$}{H(div)}-conforming HDG methods with continuous traces}\label{sec5}
In this section we discuss the corresponding continuous-trace variants. The arguments are parallel to those of \Cref{sec3,sec4}; therefore we emphasize the definitions and the resulting estimates.
\subsection{BDM elements HDG}
We now consider the following finite element spaces. For $ k\ge 1$, we define
\begin{align*}
	\bm{\Sigma}_h&=\{\mathbb L_h\in \bm L^2(\Omega):\mathbb L_h|_T\in [\mathcal P_{k-1}(T)]^{d\times d},\forall T\in\mathcal T_h\},\\
	\bm {V^g}_h&=\{\bm v_h\in  {\bm H({\rm div};\Omega)}: \bm v_h|_T\in [\mathcal{P}_k(T)]^d,\forall T\in\mathcal T_h,\bm v_h\cdot \bm n |_{\partial \Omega}=\Pi_k^{\partial} (\bm g\cdot\bm n)\},\\
	\widehat{\bm V}_h^{\bm g}&=\{\widehat{\bm v}_h\in C^0(\mathcal E_h):\widehat{\bm v}_h|_E\in [\mathcal P_{k}(E)]^d,\forall E\in\mathcal E_h,\ \widehat{\bm v}_h|_{\Gamma}=\mathcal I_k^\Gamma \bm g\},\\
	Q_h&=\{q_h\in L_0^2(\Omega):q_h|_T\in \mathcal{P}_{k-1}(T),\forall T\in\mathcal T_h\},
\end{align*}
where $\mathcal I_k^\Gamma$ denotes the continuous Lagrange interpolation operator on the boundary mesh $\mathcal E_h^\partial$.

Find $(\mathbb L_h,\bm u_h,\widehat{\bm u}_h,q_h)\in \bm{\Sigma}_h\times\bm {V^g}_h\times \widehat{\bm V}_h^{\bm g}\times Q_h$ such that 
\begin{subequations}
	\begin{align}
		\nu^{-1}(\mathbb L_h,\mathbb G_h)_{\mathcal T_h}+(\bm u_h,\nabla\cdot\mathbb G_h)_{\mathcal T_h}
		-\langle \widehat{\bm u}_h ,\mathbb G_h\bm n \rangle_{\partial\mathcal T_h}&= 0,\\
		-(\nabla\cdot\mathbb L_h,\bm v_h)_{\mathcal T_h}
		+\langle\mathbb L_h\bm n,\widehat{\bm v}_h \rangle_{\partial\mathcal T_h} -(p_h,\nabla \cdot\bm v_h)_{\mathcal T_h} \nonumber
		&\\
		+\nu\langle \eta(\bm u_h-\widehat{\bm u}_h ),
		\bm v_h-\widehat{\bm v}_h
		\rangle_{\partial\mathcal T_h}
		&=(\bm f,\bm v_h)_{\mathcal T_h},\\
		(\nabla\cdot\bm u_h,q_h)_{\mathcal T_h}&=0 {,}
	\end{align}
	{
		holds for all $(\mathbb G_h,\bm v_h,\widehat{\bm v}_h,q_h)\in \bm{\Sigma}_h\times\bm V_h^{\bm 0}\times \widehat{\bm V}_h^{\bm 0}\times Q_h$.
	}
\end{subequations}

\subsection{RT elements HDG}
We now consider the following finite element spaces. For $ k\ge 1$, we define
\begin{align*}
	\bm{\Sigma}_h&=\{\mathbb L_h\in \bm L^2(\Omega):\mathbb L_h|_T\in [\mathcal P_{k}(T)]^{d\times d},\forall T\in\mathcal T_h\},\\
	\bm {V^g}_h&=\{\bm v_h\in  {\bm H({\rm div};\Omega)}: \bm v_h|_T\in {\bm{ RT}}_k(T),\forall T\in\mathcal T_h,\bm v_h\cdot \bm n |_{\partial \Omega}=\Pi_k^{\partial} (\bm g\cdot\bm n)\},\\
	\widehat{\bm V}_h^{\bm g}&=\{\widehat{\bm v}_h\in C^0(\mathcal E_h):\widehat{\bm v}_h|_E\in [\mathcal P_{k}(E)]^d,\forall E\in\mathcal E_h,\ \widehat{\bm v}_h|_{\Gamma}=\mathcal I_k^\Gamma \bm g\},\\
	Q_h&=\{q_h\in L_0^2(\Omega):q_h|_T\in \mathcal{P}_{k}(T),\forall T\in\mathcal T_h\}.
\end{align*}
Our methods reads:
Find $(\mathbb L_h,\bm u_h,\widehat{\bm u}_h,q_h)\in \bm{\Sigma}_h\times\bm {V^g}_h\times \widehat{\bm V}_h^{\bm g}\times Q_h$ such that 
\begin{subequations}
	\begin{align}
		\nu^{-1}(\mathbb L_h,\mathbb G_h)_{\mathcal T_h}+(\bm u_h,\nabla\cdot\mathbb G_h)_{\mathcal T_h}
		-\langle \widehat{\bm u}_h ,\mathbb G_h\bm n \rangle_{\partial\mathcal T_h}&= 0,\\
		-(\nabla\cdot\mathbb L_h,\bm v_h)_{\mathcal T_h}
		+\langle\mathbb L_h\bm n,\widehat{\bm v}_h \rangle_{\partial\mathcal T_h} -(p_h,\nabla \cdot\bm v_h)_{\mathcal T_h} \nonumber
		&\\
		+\nu\langle \eta(\bm\Pi_k^{\partial}\bm u_h-\widehat{\bm u}_h ),
		\bm\Pi_k^{\partial}\bm v_h-\widehat{\bm v}_h
		\rangle_{\partial\mathcal T_h}
		&=(\bm f,\bm v_h)_{\mathcal T_h},\\
		(\nabla\cdot\bm u_h,q_h)_{\mathcal T_h}&=0
	\end{align}
	{
		holds for all $(\mathbb G_h,\bm v_h,\widehat{\bm v}_h,q_h)\in \bm{\Sigma}_h\times\bm V_h^{\bm 0}\times \widehat{\bm V}_h^{\bm 0}\times Q_h$.
	}
\end{subequations}

\subsection{Error Estimation}
The continuous-trace variants satisfy estimates analogous to those proved in \Cref{sec3,sec4}. The only substantive change is the replacement of the discontinuous skeleton projector by the continuous boundary interpolation operator $\mathcal I_k^\Gamma$. We record a representative result.

\begin{theorem}
Assume that $(\mathbb L,\bm u,p)\in[H^{k}(\Omega)]^{d\times d}\times[H^{k+1}(\Omega)]^{d}\times(H^{k}(\Omega)\cap L_0^2(\Omega))$, and let $(\mathbb L_h,\bm u_h,\widehat{\bm u}_h,p_h)$ denote the solution of either continuous-trace scheme defined in this section. Then the same convergence orders as in \Cref{sec3,sec4} hold for the corresponding BDM and RT variants. In particular, for the BDM continuous-trace method,
\begin{align}
	\nu^{-1/2}\|\mathbb L-\mathbb L_h\|_{\mathcal T_h}
	+\nu^{1/2}\|\nabla(\bm u-\bm u_h)\|_{\mathcal T_h}
	&\leq Ch^k\nu^{1/2}|\bm u|_{k+1},\\
	\|p-p_h\|_{\mathcal T_h}&\leq Ch^k(\nu^{1/2}|\bm u|_{k+1}+ |p|_{k}),\\
	\|\bm u-\bm u_h\|_{\mathcal T_h}&\le C h^{k+1}|\bm u|_{k+1}.
\end{align}
\end{theorem}

\begin{proof}
The proof follows the arguments of \Cref{sec3,sec4} after replacing the discontinuous trace space by its continuous counterpart and using the boundary interpolation operator $\mathcal I_k^\Gamma$ on $\Gamma$. We omit the routine modifications.
\end{proof}

\section{Spectral property of the Schur complement}\label{sec6}
In this section we study the Schur complement associated with the BDM discontinuous-trace formulation from \Cref{first Scheme}. We state the result in terms of spectral equivalence with respect to the pressure mass matrix, which is the formulation relevant for preconditioning.

For each element $K\in\mathcal T_h$, given $\bm w\in [L^2(K)]^d$ and $\bm\lambda\in [L^2(\partial K)]^d$, we define the local flux $\mathbb L_h^{\bm w,\bm\lambda}\in[\mathcal P_k(K)]^{d\times d}$ by
\begin{align*}
	\nu^{-1}(\mathbb L_h^{\bm w,\bm\lambda},\mathbb G_h)_{K}
	=-(\bm w,\nabla\cdot\mathbb G_h)_{K}
	+\langle \bm\lambda ,\mathbb G_h\bm n \rangle_{\partial K},
	\qquad \forall \mathbb G_h\in[\mathcal P_k(K)]^{d\times d}.
\end{align*}
After static condensation, the BDM scheme can be written as: find $(\bm u_h,\widehat{\bm u}_h,p_h)\in \bm V_h^{\bm g}\times \widehat{\bm V}_h^{\bm g}\times Q_h$ such that
\begin{subequations}\label{HDG2}
\begin{align}
	\nu(\mathbb L_h^{\bm u_h,\widehat{\bm u}_h},\mathbb L_h^{\bm v_h,\widehat{\bm v}_h})_{\mathcal T_h}
	+\nu\langle \eta(\bm{\Pi}_{k-1}^{\partial}\bm u_h-\widehat{\bm u}_h ),
	\bm{\Pi}_{k-1}^{\partial}\bm v_h-\widehat{\bm v}_h
	\rangle_{\partial\mathcal T_h}
	-(p_h,\nabla \cdot\bm v_h)_{\mathcal T_h}
	&=(\bm f,\bm v_h)_{\mathcal T_h},\\
	(\nabla\cdot\bm u_h,q_h)_{\mathcal T_h}&=0
\end{align}
for all $(\bm v_h,\widehat{\bm v}_h,q_h)\in \bm V_h^{\bm 0}\times \widehat{\bm V}_h^{\bm 0}\times Q_h$.
\end{subequations}
The corresponding linear system has the block form
\begin{align}\label{matrix}
	\begin{bmatrix}
		{\bm A} &B^T\\
		B &O
	\end{bmatrix}
	\begin{bmatrix}
		{\bf u}\\
		{\bf p}
	\end{bmatrix}
	=
	\begin{bmatrix}
		{\bf f}\\
		{\bf 0}
	\end{bmatrix}.
\end{align}
Here ${\bf u}$ collects the coefficients of $(\bm u_h,\widehat{\bm u}_h)$ and ${\bf p}$ collects the coefficients of $p_h$. Moreover,
\begin{align*}
	\langle{\bm A}{\bf v},{\bf v}\rangle
	=\nu\|\mathbb L_h^{\bm v_h,\widehat{\bm v}_h}\|_{\mathcal T_h}^2
	+\nu\|\eta^{1/2}(\bm{\Pi}_{k-1}^{\partial}\bm v_h-\widehat{\bm v}_h )\|^2_{\partial\mathcal T_h},
\end{align*}
and the divergence coupling is defined by
\begin{align*}
	\langle B{\bf v},{\bf p}\rangle
	=\langle B^T{\bf p},{\bf v}\rangle
	=-(p_h,\nabla\cdot\bm v_h)_{\mathcal T_h}.
\end{align*}

Applying the same argument as in Step three of \Cref{Theorem 2.3}, we obtain
\begin{align}
	\|p_h\|_{\mathcal T_h}\lesssim \sup_{\bm 0\neq \bm v_h\in \bm V_h^{\bm 0}}
	\frac{(p_h,\nabla\cdot\bm v_h)_{\mathcal T_h}}
	{(\|\mathbb L_h^{\bm v_h,\widehat{\bm v}_h}\|_{\mathcal T_h}^2+\|\eta^{1/2}(\bm{\Pi}_{k-1}^{\partial}\bm v_h-\widehat{\bm v}_h )\|^2_{\partial\mathcal T_h})^{1/2}}
	\lesssim \|p_h\|_{\mathcal T_h}. \label{schur-inf-sup}
\end{align}
Therefore,
\begin{align*}
	\|p_h\|_{\mathcal T_h}
	&\lesssim \sup_{{\bf v}\neq {\bf 0}}\frac{\langle B^T{\bf p}, {\bf v}\rangle}{\langle  {\bm A}{\bf v},{\bf v}\rangle^{1/2}}
	= \sup_{{\bf w}\neq {\bf 0}}\frac{\langle  {\bm A}^{-1/2}B^T{\bf p}, {\bf w}\rangle}{\langle{\bf w},{\bf w}\rangle^{1/2}} \\
	&= \|{\bm A}^{-1/2}B^T{\bf p}\|
	= \langle B{\bm A}^{-1}B^T{\bf p},{\bf p} \rangle^{1/2}
	\lesssim \|p_h\|_{\mathcal T_h}.
\end{align*}
This proves the following result.

\begin{theorem}
Let $S:=B{\bm A}^{-1}B^T$ be the pressure Schur complement. There exist constants $c,C>0$, independent of $h$ and $\nu$, such that
\begin{align*}
	c\|p_h\|_{\mathcal T_h}^2
	\le \langle S{\bf p},{\bf p}\rangle
	\le C\|p_h\|_{\mathcal T_h}^2
\end{align*}
for all $p_h\in Q_h$. Equivalently, the Schur complement is spectrally equivalent to the pressure mass operator, and the mass-preconditioned Schur complement has condition number bounded independently of $h$ and $\nu$.
\end{theorem}

\section{Application: tangential boundary control}\label{sec7}

In this section we consider the tangential Dirichlet boundary control problem in two space dimensions with $\nu=1$. We assume that $\{\mathcal T_h\}_{h>0}$ is a family of quasi-uniform triangulations, where $\mathcal E_h^o$ denotes the set of interior edges and $\mathcal E_h^{\partial}$ the set of boundary edges. To keep the presentation focused, we formulate the control discretization only for the BDM method with discontinuous traces; the other variants can be treated analogously. For earlier HDG discretizations of Dirichlet boundary control problems, see \cite{MR3831243,MR3992054,MR4057428,MR4169689,MR4381532,MR4527345}.

The tangential boundary control problem is formulated as follows: find $\bm u$ such that
\begin{align}
	\min_{\bm u\in \bm U}J(\bm u) = \frac 1 2\|\bm y_{\bm u}-\bm y_d\|_{\bm L^2(\Omega)}^2+\frac{\gamma}{2}\|\bm u\|_{\bm U}^2,
\end{align}
where $\bm y_d$ is the desired state and $\bm y_{\bm u}:=\bm y$ is the solution of 
\begin{subequations}\label{stokes-t}
	\begin{align}
		-\Delta\bm y+\nabla p&=\bm f&\text{in }\Omega,\\
		\nabla\cdot\bm y&=0&\text{in }\Omega,\\
		\bm y&=\bm u&\text{on }\partial\Omega,
	\end{align}
\end{subequations}
and $\gamma$ is a positive constant, $\bm U$ is defined as
\begin{align}
	\bm U:=\{\bm u=u\bm\tau:u\in L^2(\Gamma)\},
\end{align}
with the definition $\|\bm u\|_{\bm U}=\|u\|_{L^2(\Gamma)}$, $\bm\tau$ is the unit tangential vector to the boundary $\Gamma$.

The corresponding first-order optimality system is: find $(\bm y, p,\bm z,w,\bm u)$ such that
\begin{subequations}\label{op}
	\begin{align}
		-\Delta\bm y+\nabla p&=\bm f\text{ in }\Omega,\quad
		\nabla\cdot\bm y=0\text{ in }\Omega,\quad
		\bm y=\bm u ~\text{ on }\partial\Omega, \\
		-\Delta\bm z+\nabla w&=\bm y-\bm y_d\text{ in }\Omega,\quad
		\nabla\cdot\bm z=0\text{ in }\Omega,\quad
		\frac{\partial\bm z}{\partial\bm n}\cdot\bm\tau=\gamma\bm u\text{ on }\partial\Omega,
	\end{align}
\end{subequations}
We introduce $\mathbb L=\nabla\bm y$ and $\mathbb T=\nabla\bm z$ to obtain 
\begin{subequations}\label{op2}
	\begin{align}
		\mathbb L - \nabla\bm y&=\bm 0\text{ in }\Omega,\quad
		-\nabla\cdot\mathbb L+\nabla p=\bm f\text{ in }\Omega,\quad
		\nabla\cdot\bm y=0\text{ in }\Omega,\quad
		\bm y=\bm u~\text{on }\partial\Omega, \\
		\mathbb T - \nabla\bm z&=\bm 0\text{ in }\Omega,\quad
		-\nabla\cdot\mathbb T+\nabla w=\bm y-\bm y_d\text{ in }\Omega,\\
		\nabla\cdot\bm z&=0\text{ in }\Omega,\quad
		\frac{\partial\bm z}{\partial\bm n}\cdot\bm\tau=\gamma\bm u ~\text{ on }\partial\Omega.
	\end{align}
\end{subequations}

We introduce the finite element spaces as:
\begin{align*}
	\bm{\Sigma}_h&=\{\mathbb L_h\in \bm L^2(\Omega):\mathbb L_h|_T\in [\mathcal P_{k-1}(T)]^{2\times 2},\forall T\in\mathcal T_h\},\\
	\bm {V}_h^{\bm 0}&=\{\bm v_h\in  {\bm H({\rm div};\Omega)}: \bm v_h|_T\in [\mathcal P_k(T)]^2,\forall T\in\mathcal T_h,\bm v_h\cdot\bm n|_{\Gamma}=0\},\\
	\widehat{\bm V}_h^{\bm 0}&=\{\widehat{\bm v}_h\in\bm L^2(\mathcal E_h):\widehat{\bm v}_h|_E\in [\mathcal P_{k-1}(E)]^2,\forall E\in\mathcal E_h,	\widehat{\bm v}_h|_{\Gamma}=\bm 0
	\},\\
	U_h&=\{v_h\in L^2(\mathcal E_h^{\partial}):v_h|_E\in \mathcal P_{k-1}(E),\forall E\in\mathcal E_h^{\partial}
	\},\\
	Q_h&=\{q_h\in L_0^2(\Omega):q_h|_T\in \mathcal{P}_{k-1}(T),\forall T\in\mathcal T_h\}.
\end{align*}
The HDG method reads: find $(\mathbb L_h,\bm y_h, \widehat{\bm y}_h^o, p_h, \mathbb T_h,\bm z_h,\widehat{\bm z}_h,w_h,u_h)\in \bm{\Sigma}_h\times \bm V_h^{\bm 0}\times\widehat{\bm V}_h^{\bm 0}\times Q_h\times \bm{\Sigma}_h\times \bm V_h^{\bm 0}\times\widehat{\bm V}_h^{\bm 0}\times Q_h\times U_h$  such that 
\begin{subequations}
	\begin{align}\label{op-dis}
		\mathscr B(\mathbb L_h,\bm y_h,\widehat{\bm y}_h^o,p_h;
		\mathbb G_1,\bm v_1,\widehat{\bm v}_1,q_1)
		&=	(\bm f,\bm v_1)_{\mathcal T_h}   +\langle u_h\bm\tau,\mathbb G_1\bm n \rangle_{\Gamma}
		+\langle \eta u_h\bm\tau,\bm v_1\rangle_{\Gamma},
		\\
		\mathscr B(\mathbb T_h,\bm z_h,\widehat{\bm z}_h,w_h;
		\mathbb G_2,\bm v_2,\widehat{\bm v}_2,q_2)&=	(\bm y_h-\bm y_d,\bm v_2)_{\mathcal T_h},\\
		\langle\gamma u_h\bm\tau-\mathbb T_h\bm n+\eta\bm{\Pi}_k^{\partial}\bm z_h,\mu\bm\tau \rangle_{\Gamma}&=0
	\end{align}
\end{subequations}
holds for all $(\mathbb G_1,\bm v_1, \widehat{\bm v}_1, q_1, \mathbb G_2,\bm v_2, \widehat{\bm v}_2, q_2,\mu)\in \bm{\Sigma}_h\times \bm V_h^{\bm 0}\times\widehat{\bm V}_h^{\bm 0}\times Q_h\times \bm{\Sigma}_h\times \bm V_h^{\bm 0}\times\widehat{\bm V}_h^{\bm 0}\times Q_h\times U_h$.

To derive the control error estimates, we introduce the auxiliary discrete co-problems: find $(\mathbb L_h(u),\\ \bm y_h(u), \widehat{\bm y}_h^o(u), p_h(u), \mathbb T_h(u),\bm z_h(u),\widehat{\bm z}_h(u),w_h(u))\in \bm{\Sigma}_h\times \bm V_h^{\bm 0}\times\widehat{\bm V}_h^{\bm 0}\times Q_h\times \bm{\Sigma}_h\times \bm V_h^{\bm 0}\times\widehat{\bm V}_h^{\bm 0}\times Q_h$  such that 
\begin{align}
	&\mathscr B(\mathbb L_h(u),\bm y_h(u),\widehat{\bm y}_h^o(u),p_h(u);
	\mathbb G_1,\bm v_1,\widehat{\bm v}_1,q_1) \notag\\
	&\qquad=	(\bm f,\bm v_1)_{\mathcal T_h}   +\langle \Pi_{k-1}^{\partial} u\bm\tau,\mathbb G_1\bm n \rangle_{\Gamma}
	+\langle \eta \Pi_{k-1}^{\partial}u\bm\tau,\bm v_1\rangle_{\Gamma},
	\\
	&\mathscr B(\mathbb T_h(u),\bm z_h(u),\widehat{\bm z}_h(u),w_h;
	\mathbb G_2,\bm v_2,\widehat{\bm v}_2,q_2)=	(\bm y_h(u)-\bm y_d,\bm v_2)_{\mathcal T_h}.
\end{align}
holds for all   $(\mathbb G_1,\bm v_1, \widehat{\bm v}_1, q_1, \mathbb G_2,\bm v_2, \widehat{\bm v}_2, q_2)\in \bm{\Sigma}_h\times \bm V_h^{\bm 0}\times\widehat{\bm V}_h^{\bm 0}\times Q_h\times \bm{\Sigma}_h\times \bm V_h^{\bm 0}\times\widehat{\bm V}_h^{\bm 0}\times Q_h$.

We use \Cref{BDM error equation lemma} to get the following results.
\begin{lemma}
	Let $(\mathbb L,\bm y,p, \mathbb T, \bm z, w, u)\in   [H^{\epsilon}(\Omega)]^{2\times 2}\times[ H^1(\Omega)]^2\times (L_0^2(\Omega)\cap H^{\epsilon}(\Omega))\times [H^{\epsilon}(\Omega)]^{2\times 2}\times[ H^1(\Omega)]^2\times (L_0^2(\Omega)\cap H^{\epsilon}(\Omega))\times L^2(\Gamma)$  {with $0<\epsilon<1/2$} be the solution to \Cref{op2}, then we have
	\begin{align*}
		&\mathscr B(\bm{\Pi}_{k-1}^o\mathbb L,\bm{\Pi}_k^{BDM}\bm y,\widetilde{\bm\Pi}_{k-1}^{\partial}\bm y,\Pi_{k-1}^op;
		\mathbb G_1,\bm v_1,\widehat{\bm v}_1,q_1) \nonumber
		\\
		&\qquad=	(\bm f,\bm v_1)_{\mathcal T_h}  +E(\mathbb L,\bm y;\bm v_1,\widehat{\bm v}_1)
		+\langle \Pi_{k-1}^{\partial} u\bm\tau,\mathbb G_1\bm n \rangle_{\Gamma}
		+\langle \eta \Pi_{k-1}^{\partial} u\bm\tau,\bm v_1\rangle_{\Gamma},
		\\
		& \mathscr B(\bm{\Pi}_{k-1}^o\mathbb T,\bm{\Pi}_k^{BDM}\bm z,\bm\Pi_{k-1}^{\partial}\bm z,\Pi_{k-1}^o w;
		\mathbb G_2,\bm v_2,\widehat{\bm v}_2,q_2) \nonumber\\
		&\qquad=	(\bm\Pi_k^o\bm y-\bm y_d,\bm v_2)_{\mathcal T_h} +E(\mathbb T,\bm z;\bm v_2,\widehat{\bm v}_2)
	\end{align*}
	holds for all   $(\mathbb G_1,\bm v_1, \widehat{\bm v}_1, q_1, \mathbb G_2,\bm v_2, \widehat{\bm v}_2, q_2)\in \bm{\Sigma}_h\times \bm V_h^{\bm 0}\times\widehat{\bm V}_h^{\bm 0}\times Q_h\times \bm{\Sigma}_h\times \bm V_h^{\bm 0}\times\widehat{\bm V}_h^{\bm 0}\times Q_h$.
\end{lemma}
To simplify the notation, we define
\begin{align*}
	\zeta_{\mathbb{L}}=\mathbb{L}_h(u)-\mathbb{L}_h, \quad \zeta_{\boldsymbol{y}}=\boldsymbol{y}_h(u)-\boldsymbol{y}_h, \quad \zeta_p=p_h(u)-p_h,  \\
	\zeta_{\mathbb{T}}=\mathbb{T}_h(u)-\mathbb{T}_h, \quad \zeta_z=\boldsymbol{z}_h(u)-\boldsymbol{z}_h, \quad \zeta_w=w_h(u)-w_h,\\
	\zeta_{\widehat{\boldsymbol{y}}}=\widehat{\boldsymbol{y}}_h^o(u)-\widehat{\boldsymbol{y}}_h^o \text { on } \mathcal{E}_h^o \text { and } \zeta_{\widehat{\boldsymbol{y}}}=\Pi_{k-1}^{\partial} u \boldsymbol{\tau}-u_h \boldsymbol{\tau} \text { on } \mathcal{E}_h^{\partial}, \\
	\zeta_{\widehat{z}}=\widehat{\boldsymbol{z}}_h(u)-\widehat{\boldsymbol{z}}_h \text { on } \mathcal{E}_h^o \text { and } \zeta_{\widehat{z}}=0 \text { on } \mathcal{E}_h^{\partial} .
\end{align*}

With the same proof in \cite{MR4527345} we can get the following result.
\begin{lemma}
	We have
	\begin{align*}
		\gamma\left\|u-u_h\right\|_{\mathcal{E}_h^{\partial}}^2+\left\|\zeta_{\boldsymbol{y}}\right\|_{\mathcal{T}_h}^2= & \langle\gamma u \boldsymbol{\tau}-\mathbb{T}_h(u) \boldsymbol{n}+\eta \boldsymbol{\Pi}_{k-1}^{\partial} \boldsymbol{z}_h(u),\left(u-u_h\right) \boldsymbol{\tau}\rangle_{\mathcal{E}_h^{\partial}} \\
		& -\langle\gamma u_h \boldsymbol{\tau}-\mathbb{T}_h \boldsymbol{n}+\eta \boldsymbol{\Pi}_{k-1}^{\partial} \boldsymbol{z}_h,\left(u-u_h\right) \boldsymbol{\tau}\rangle_{\mathcal{E}_h^{\partial}} .
	\end{align*}
\end{lemma}
{
	
	In the following parts of this section, we assume the solution of \eqref{op}-\eqref{op2} satisfies
	$$
	\boldsymbol{y} \in \boldsymbol{H}^{r_{\bm y}}(\Omega), \quad \bm z \in \boldsymbol{H}^{r_{\bm z}}(\Omega),
	\text{ with }
	r_{\bm y}>\frac 3 2, \quad r_{\bm z}>2.
	$$
	We set 
	$$
	s_{\bm y}=\min \left\{r_{\bm y}, k+1\right\},\quad s_{\bm z}=\min \left\{r_{\bm z}, k+1\right\}.
	$$

}
\endgroup

By \Cref{th-35,theorem 2.8}  we have the error estimates as follows.
\begin{lemma} We have the error estimate
	\begin{align*}
		\interleave(\bm{\Pi}_{k-1}^o\mathbb L-\mathbb L_h(u),\bm{\Pi}_{k}^{BDM}\bm y-\bm y_h(u),\widetilde{\bm{\Pi}}_{k-1}^{\partial}\bm y-{\bm y}^o_h(u))\interleave&\lesssim h^{s_{\bm y-1}}|\bm y|_{s_{\bm y}},\\
		\|\bm{\Pi}_{k-1}^{\partial}\bm y-{\bm y}_h(u)\|_{\mathcal T_h}&\lesssim h^{s_{\bm y}}|\bm y|_{s_{\bm y}} {.}
	\end{align*}
\end{lemma}
We use the similar proof in \Cref{th-35,theorem 2.8} to get
\begin{lemma} We have the error estimate
	\begin{align*}
		\interleave(\bm{\Pi}_{k-1}^o\mathbb T-\mathbb T_h(u),\bm{\Pi}_{k}^{BDM}\bm z-\bm z_h(u),\bm{\Pi}_{k-1}^{\partial}\bm z-{\bm z}_h(u))\interleave&\lesssim h^{s_{\bm y}}|\bm y|_{s_{\bm y}}+h^{s_{\bm z}-1}|\bm z|_{s_{\bm z}},\\
		\|\bm{\Pi}_{k-1}^{\partial}\bm z-{\bm z}_h(u)\|_{\mathcal T_h}&\lesssim h^{s_{\bm y}}|\bm y|_{s_{\bm y}}+h^{s_{\bm z}-1}|\bm z|_{s_{\bm z}}.
	\end{align*}
\end{lemma}

\begin{lemma}
	Let $(\boldsymbol{y}, u)$ and $\left(\boldsymbol{y}_h, u_h\right)$ be the solutions of \eqref{op2} and \eqref{op-dis}, respectively. We have
	$$
	\begin{aligned}
		\left\|u-u_h\right\|_{\mathcal{E}_h^{\partial}} + \left\|\boldsymbol{y}-\boldsymbol{y}_h\right\|_{\mathcal{T}_h} \lesssim h^{s_{\bm y}-\frac{1}{2}}|\boldsymbol{y}|_{s_{\boldsymbol{y}}}+h^{s_{\bm z}-\frac{3}{2}}|\bm z|_{s_{\bm z}} .
	\end{aligned}
	$$
\end{lemma}

\begingroup
\begin{proof}
	Since $\gamma u \boldsymbol{\tau}-\mathbb{T} \boldsymbol{n}=0$ on $\mathcal{E}_h^{\partial}$ and $\gamma u_h \boldsymbol{\tau}-\mathbb{T}_h \boldsymbol{n}+\eta\boldsymbol{\Pi}_{k-1}^{\partial} \boldsymbol{z}_h=0$ on $\mathcal{E}_h^{\partial}$ we have
	$$
	\begin{aligned}
		\gamma\left\|u-u_h\right\|_{\mathcal{E}_h^{\partial}}^2+\left\|\zeta_{\boldsymbol{y}}\right\|_{\mathcal{T}_h}^2 & =\langle\gamma u \boldsymbol{\tau}-\mathbb{T}_h(u) \boldsymbol{n}+\eta \boldsymbol{\Pi}_{k-1}^{\partial} \boldsymbol{z}_h(u),\left(u-u_h\right) \boldsymbol{\tau}\rangle_{\mathcal{E}_h^{\partial}} \\
		& =\langle\left(\mathbb{T}-\mathbb{T}_h(u)\right) \boldsymbol{n}+\eta \boldsymbol{\Pi}_{k-1}^{\partial} \boldsymbol{z}_h(u),\left(u-u_h\right) \boldsymbol{\tau}\rangle_{\mathcal{E}_h^{\partial}} .
	\end{aligned}
	$$
	Next, since $\widehat{\boldsymbol{z}}_h(u)=\boldsymbol{z}=\mathbf{0}$ on $\mathcal{E}_h^{\partial}$ we have
	$$
	\begin{aligned}
		&\|\boldsymbol{\Pi}_{k-1}^{\partial} \bm z_h(u) \|_{\mathcal{E}_h^{\partial}}  = \|\boldsymbol{\Pi}_{k-1}^{\partial}( \bm z_h(u)- {\bm \Pi}_{k}^{BDM}\bm  z+ {\bm \Pi}_{k}^{BDM} \bm z-\bm  z+ \bm z)-\widehat{\bm z}_h(u) \|_{\mathcal{E}_h^{\partial}} \\
		& \qquad\leq \|\boldsymbol{\Pi}_{k-1}^{\partial} (\bm z_h(u)-\bm{\Pi}_{k}^{BDM}\bm z)-(\widehat{\bm z}_h(u)-\bm\Pi_{k-1}^{\partial}\bm z) \|_{\partial \mathcal{T}_h}+ \|{\bm \Pi}_{k}^{BDM} \boldsymbol{z}-\boldsymbol{z} \|_{\mathcal{E}_h^{\partial}} .
	\end{aligned}
	$$
	This together with the above lemmas gives
	$$
	\begin{aligned}
		\|u-u_h \|_{\mathcal{E}_h^{\partial}}+ \|\zeta_{\boldsymbol{y}} \|_{\mathcal{T}_h} \lesssim & \|\eta (\mathbb T-\mathbb T_h(u))\|_{\partial\mathcal T_h} + \|\eta({\bm \Pi}_{k}^{BDM} \boldsymbol{z}-\boldsymbol{z} )\|_{\mathcal{E}_h^{\partial}}\\
		& +\|\eta(\boldsymbol{\Pi}_{k-1}^{\partial} (\bm z_h(u)-\bm{\Pi}_{k-1}^{BDM}\bm z)-(\widehat{\bm z}_h(u)-\bm\Pi_{k-1}^{\partial}\bm z) )\|_{\partial \mathcal{T}_h} .
	\end{aligned}
	$$
	By the above lemmas and properties of the $L^2$ projection, we have
	$$
	\left\|u-u_h\right\|_{\mathcal{E}_h^{\partial}}+\left\|\zeta_{\boldsymbol{y}}\right\|_{\mathcal{T}_h} \lesssim h^{s_{\bm y}-\frac{1}{2}}|\boldsymbol{y}|_{s_{\boldsymbol{y}}}+h^{s_{\bm z}-\frac{3}{2}}|\bm z|_{s_{\bm z}} .
	$$
	Then, by the triangle inequality we obtain
	$$
	\left\|\boldsymbol{y}-\boldsymbol{y}_h\right\|_{\mathcal{T}_h} \lesssim h^{s_{\bm y}-\frac{1}{2}}|\boldsymbol{y}|_{s_{\bm y}}+h^{s_{\bm z}-\frac{3}{2}}|\bm z|_{s_{\bm z}}.
	$$
\end{proof}

The next result follows immediately.
\begin{theorem} Let $\mathbb{L}$ and $\mathbb{L}_h$ be the solutions of \eqref{op2} and \eqref{op-dis}, respectively. Then it holds $$\left\|\mathbb{L}-\mathbb{L}_h\right\|_{\mathcal{T}_h} \lesssim h^{s_{\bm y}-1}|\boldsymbol{y}|_{s_{\bm y}}+h^{s_{\bm z}-2}|\bm z|_{s_{\bm z}}.$$
	
\end{theorem}

	\section{Numerical Experiments}\label{sec8}
In this section we present two-dimensional experiments on triangular meshes and three-dimensional experiments on tetrahedral meshes to illustrate the theoretical results for the BDM- and RT-based HDG schemes. All examples are implemented in C++ using the Eigen library \cite{eigenweb} and Hypre \cite{10.1007/3-540-47789-6_66}.
We denote our methods employing discontinuous traces with BDM and RT elements as BDM and RT, respectively. For our methods utilizing continuous traces with BDM and RT elements, we denote them as BDM2 and RT2, respectively. To solve the linear system within our methods, we employ static condensation to eliminate the unknown $\mathbb L_h$, and we utilize AMG and CG methods to solve the Schur complement system of \eqref{matrix}.
The numerical results support the following conclusions:
\begin{enumerate}
	\item The convergence rates exhibited by our methods align with the theoretical outcomes presented in \Cref{section3.3} and \Cref{section4.3}, remaining independent of the fluid viscosity coefficient.
	\item The quantity $\|\nabla\cdot\bm u_h\|_{\mathcal{T}_h}$ stays at or near machine precision, confirming the exactly divergence-free property of the proposed methods and their pressure robustness.
\end{enumerate}  
\subsection{Two-Dimensional Numerical Results}
The body forces $\bm f$ are selected such that the analytical solution to \Cref{Ori_problem2}, considering the homogeneous Dirichlet boundary condition, is expressed as
\begin{align*}
	u_1&=-x^2(x-1)^2 y(y-1)(2 y-1), \\ 
	u_2&=x(x-1)(2 x-1) y^2(y-1)^2, \\ 
	p&=x^6-y^6.
\end{align*}
\endgroup

In the forthcoming numerical examples, we consider $\Omega=[0,1]\times[0,1]$, with a constant viscosity of $\nu=1.00E+00$ and $1.00E-03$, where $\nu=Re^{-1}$. The computational mesh consists of a regular triangulation with $2\times n\times n$ triangles, denoted as an $n \times n$ mesh. The results, depicted in \Cref{table1}-\Cref{table7}, include 'Iter' representing the iterations to solve the Schur complement system, while 'DOF' refers to the degrees of freedom within the entire linear system prior to static condensation.

\begin{table}
	\renewcommand{\arraystretch}{0.9}
	\centering
	\resizebox{1.0\textwidth}{!}{
		\begin{tabular}{c|c|c|c|c|c|c|c|c|c|c}
			\Xhline{1pt}
			\multirow{2}{*}{$k$}&\multirow{2}{*}{Mesh}&\multirow{2}{*}{Iter}&\multirow{2}{*}{DOF}&\multicolumn{2}{c|}{$\frac{\|\mathbb{L}-\mathbb{L}_h\|_{\mathcal{T}_h}}{\|\mathbb{L}\|_{\mathcal{T}_h}}$}&\multicolumn{2}{c|}{$\frac{\|\bm u-\bm u_h\|_{\mathcal{T}_h}}{\|\bm u\|_{\mathcal{T}_h}}$}&\multicolumn{2}{c|}{$\frac{\|p-p_h\|_{\mathcal{T}_h}}{\|p\|_{\mathcal{T}_h}}$}&\multirow{2}{*}{$\|\nabla\cdot \bm u_h\|_{\mathcal{T}_h}$}\\
			\cline{5-10}
			& & &  &Error&Rate &Error&Rate&Error&Rate&\\
			\hline
			&$2\times2$&5&105&8.7876E-01&-&1.6346E+00&-&7.4775E-01&-&1.5914E-14\\
			&$4\times4$&16&385&4.9997E-01&0.81&4.1603E-01&1.97&4.4817E-01&0.74&3.5020E-15\\
			&$8\times8$&36&1473&2.6443E-01&0.92&1.1110E-01&1.90&2.3639E-01&0.92&2.9951E-11\\
			1&$16\times16$&48&5761&1.3431E-01&0.98&2.8978E-02&1.94&1.1983E-01&0.98& 1.2925E-11\\
			&$32\times32$&52&22785&6.7437E-02&0.99&7.4045E-03&1.97&6.0121E-02&1.00&1.0886E-10\\
			&$64\times64$&57&90625&1.6892E-02&1.00&1.8709E-03&1.98&3.0085E-02&1.00&3.1427E-11\\
			&$128\times128$&64&361473&1.6892E-02&1.00&4.7018E-04&1.99&1.5045E-02&1.00&1.2387E-11\\
			\Xhline{1pt}
			&$2\times2$&16&257&4.3054E-01&-&3.4121E-01&-&2.1303E-01&-&3.1673E-13\\
			&$4\times4$&50&969&1.2459E-01&1.79&4.6550E-02&2.87&6.6175E-02&1.69&8.0576E-11\\
			&$8\times8$&68&3761&3.3334E-02&1.90&5.9407E-03&2.97&1.7483E-02&1.92&1.3345E-10 \\
			2&$16\times16$&78&14817&8.5262E-03&1.97&7.3986E-04&3.01&4.4313E-03&1.98&4.0777E-11\\
			&$32\times32$&84&58817&2.1490E-03&1.99&9.2249E-05&3.00&1.1116E-03&2.00&1.5691E-11\\
			&$64\times64$&82&234369&5.3897E-04&2.00&1.1521E-05 &3.00&2.7814E-04&2.00&1.0472E-10 \\
			&$128\times128$&87&935681&1.3492E-04&2.00&1.4399E-06&3.00&6.9551E-05&2.00&1.3483E-10\\
			\Xhline{1pt}
	\end{tabular}}
	\caption{Results for $\nu=1.0$ with $BDM$ element on triangular meshes}
	\label{table1}
\end{table}
\begin{table}
	\renewcommand{\arraystretch}{0.9}
	\centering
	\resizebox{1.0\textwidth}{!}{
		\scalebox{0.8}{
			\begin{tabular}{c|c|c|c|c|c|c|c|c|c|c}
				\Xhline{1pt}
				\multirow{2}{*}{$k$}&\multirow{2}{*}{Mesh}&\multirow{2}{*}{Iter}&\multirow{2}{*}{DOF}&\multicolumn{2}{c|}{$\frac{\|\mathbb{L}-\mathbb{L}_h\|_{\mathcal{T}_h}}{\|\mathbb{L}\|_{\mathcal{T}_h}}$}&\multicolumn{2}{c|}{$\frac{\|\bm u-\bm u_h\|_{\mathcal{T}_h}}{\|\bm u\|_{\mathcal{T}_h}}$}&\multicolumn{2}{c|}{$\frac{\|p-p_h\|_{\mathcal{T}_h}}{\|p\|_{\mathcal{T}_h}}$}&\multirow{2}{*}{$\|\nabla\cdot \bm u_h\|_{\mathcal{T}_h}$}\\
				\cline{5-10}
				& & &  &Error&Rate &Error&Rate&Error&Rate&\\
				\hline
				&$2\times2$&5&105&8.7876E-01&-&1.6346E+00&-&7.4725E-01&-&8.4867E-14\\
				&$4\times4$&15&385&4.9997E-01&0.81&4.1603E-01&1.97&4.4803E-01&0.74&2.5906E-13\\
				&$8\times8$&40&1473&2.6443E-01&0.92&1.1110E-01&1.90&2.3631E-01&0.92&2.3752E-09\\
				1&$16\times16$&50&5761&1.3431E-01&0.98&2.8978E-02&1.94&1.1980E-01&0.98&1.2337E-09\\
				&$32\times32$&52&22785&6.7437E-02&0.99 &7.4045E-03&1.97&6.0111E-02&0.99&2.2800E-09\\
				&$64\times64$&54&90625&3.3765E-02&1.00&1.8709E-03&1.98&3.0081E-02&1.00&1.1615E-09\\
				&$128\times128$&55&361473&1.6892E-02&1.00&4.7018E-04&1.99&1.5044E-02&1.00&7.5991E-10\\
				\Xhline{1pt}
				&$2\times2$&14&257&4.3054E-01&-&3.4121E-01&-&2.1226E-01&-&3.0745E-11\\
				&$4\times4$&65&969&1.2459E-01&1.79&4.6550E-02&2.87&6.6096E-02&1.68&7.6675E-09\\
				&$8\times8$&71&3761&3.3334E-02&1.90&5.9407E-03&2.97&1.7471E-02&1.92&1.1510E-08\\
				2&$16\times16$&72&14817&8.5262E-03&1.97&7.3986E-04&3.01&4.4290E-03&1.98&6.5855E-09\\
				&$32\times32$&74&58817&2.1490E-03&1.99&9.2249E-05&3.00&1.1111E-03 &1.99&5.8807E-09\\
				&$64\times64$&79&234369&5.3897E-04&2.00&1.1521E-05&3.00&2.7802E-04&2.00&2.5037E-09\\
				&$128\times128$&81&935681&1.3492E-04&2.00&1.4398E-06&3.00&6.9521E-05&2.00& 1.4712E-09\\
				\Xhline{1pt}
	\end{tabular}}}
	\caption{Results for $\nu=1.0e-3$ with $BDM$ element on triangular meshes}
	\label{table2}
\end{table}	
\begin{table}
	\renewcommand{\arraystretch}{0.9}
	\centering
	\resizebox{\textwidth}{!}{
		\begin{tabular}{c|c|c|c|c|c|c|c|c|c|c}
			\Xhline{1pt}
			\multirow{2}{*}{$k$}&\multirow{2}{*}{Mesh}&\multirow{2}{*}{Iter}&\multirow{2}{*}{DOF}&\multicolumn{2}{c|}{$\frac{\|\mathbb{L}-\mathbb{L}_h\|_{\mathcal{T}_h}}{\|\mathbb{L}\|_{\mathcal{T}_h}}$}&\multicolumn{2}{c|}{$\frac{\|\bm u-\bm u_h\|_{\mathcal{T}_h}}{\|\bm u\|_{\mathcal{T}_h}}$}&\multicolumn{2}{c|}{$\frac{\|p-p_h\|_{\mathcal{T}_h}}{\|p\|_{\mathcal{T}_h}}$}&\multirow{2}{*}{$\|\nabla\cdot \bm u_h\|_{\mathcal{T}_h}$}\\
			\cline{5-10}
			& & &  &Error&Rate &Error&Rate&Error&Rate&\\
			\hline
			&$2\times2$&15&233&4.6891E-01&-&4.5948E-01&-&2.1492E-01&-&0.0000E+00\\
			&$4\times4$&41&881&1.7854E-01&1.39&1.7348E-01&1.41&6.7251E-02&1.68&8.6488E-16\\
			&$8\times8$&55&3425&7.2508E-02&1.30&5.0628E-02&1.78&1.8715E-02&1.85&3.4117E-10\\
			1&$16\times16$&59&13505&3.3003E-02&1.14&1.3290E-02&1.93&5.5924E-03&1.74&2.5742E-11\\
			&$32\times32$&61&53633&1.6005E-02&1.04&3.3680E-03&1.98&2.0421E-03&1.45&1.0445E-11\\
			&$64\times64$&63&213761&7.9353E-03&1.01&8.4501E-04&1.99&9.0145E-04&1.18&3.1954E-11\\
			&$128\times128$&78&853505&3.9589E-03&1.00&2.1144E-04&2.00&4.3445E-04&1.05&2.0361E-11\\
			\Xhline{1pt}
			&$2\times2$&34&433&1.6022E-01&-&2.0030E-01&-&5.0748E-02&-&0.0000E+00\\
			&$4\times4$&73&1657&3.5563E-02&2.17&2.8303E-02&2.82&7.5372E-03&2.75&1.2594E-14\\
			&$8\times8$&102&6481&7.4159E-03&2.26&3.6936E-03&2.94&1.1090E-03&2.76&1.2885E-10\\
			2&$16\times16$&117&25633&1.6382E-03&2.18&4.6259E-04&3.00&1.9417E-04&2.51&5.0604E-11\\
			&$32\times32$&128&101953&3.8556E-04&2.09&5.7562E-05&3.01&4.0965E-05&2.24&3.4890E-11\\
			&$64\times64$&133&406657&9.3815E-05&2.04&7.1681E-06&3.01&9.5770E-06&2.10&1.7449E-11\\
			&$128\times128$&126&1624321&2.3165E-05&2.02&8.9395E-07&3.00&2.3324E-06&2.04&7.2754E-11\\
			\Xhline{1pt}
	\end{tabular}}
	\caption{Results for $\nu=1.0$ with $RT$ element on triangular meshes}
	\label{table3}
\end{table}	

\begin{table}
	\renewcommand{\arraystretch}{0.9}
	\centering
	\resizebox{\textwidth}{!}{
		\begin{tabular}{c|c|c|c|c|c|c|c|c|c|c}
			\Xhline{1pt}
			\multirow{2}{*}{$k$}&\multirow{2}{*}{Mesh}&\multirow{2}{*}{Iter}&\multirow{2}{*}{DOF}&\multicolumn{2}{c|}{$\frac{\|\mathbb{L}-\mathbb{L}_h\|_{\mathcal{T}_h}}{\|\mathbb{L}\|_{\mathcal{T}_h}}$}&\multicolumn{2}{c|}{$\frac{\|\bm u-\bm u_h\|_{\mathcal{T}_h}}{\|\bm u\|_{\mathcal{T}_h}}$}&\multicolumn{2}{c|}{$\frac{\|p-p_h\|_{\mathcal{T}_h}}{\|p\|_{\mathcal{T}_h}}$}&\multirow{2}{*}{$\|\nabla\cdot \bm u_h\|_{\mathcal{T}_h}$}\\
			\cline{5-10}
			& & &  &Error&Rate &Error&Rate&Error&Rate&\\
			\hline
			&$2\times2$&13&233&4.6891E-01&-&4.5948E-01&-&2.1226E-01&-&0.0000E+00\\
			&$4\times4$&50&881&1.7854E-01&1.39&1.7348E-01&1.41&6.6096E-02&1.68&1.4285E-12\\
			&$8\times8$&61&3425&7.2508E-02&1.30&5.0628E-02&1.78&1.7471E-02&1.92&2.3668E-09\\
			1&$16\times16$&56&13505&3.3003E-02&1.14&1.3290E-02&1.93&4.4290E-03&1.98&2.7758E-09\\
			&$32\times32$&54&53633&1.6005E-02&1.04&3.3680E-03&1.98&1.1111E-03&1.99&2.7501E-09\\
			&$64\times64$&58&213761&7.9353E-03&1.01&8.4501E-04&1.99&2.7802E-04&2.00&2.4598E-09\\
			&$128\times128$&58&853505&3.9589E-03&1.00&2.1144E-04&2.00&6.9522E-05&2.00&1.3415E-09\\
			\Xhline{1pt}
			&$2\times2$&30&433&1.6022E-01&-&2.0030E-01&-&5.0191E-02&-&0.0000E+00\\
			&$4\times4$&81&1657&3.5563E-02&2.17&2.8303E-02&2.82&7.1632E-03&2.81&2.8494E-11\\
			&$8\times8$&102&6481&7.4159E-03&2.26&3.6936E-03&2.94&9.2462E-04&2.95&9.2019E-09\\
			2&$16\times16$&101&25633&1.6382E-03&2.18&4.6259E-04&3.00&1.1650E-04&2.99& 1.2033E-08\\
			&$32\times32$&111&101953&3.8556E-04&2.09&5.7562E-05&3.01&1.4591E-05 &3.00&1.1449E-08\\
			&$64\times64$&112&406657&9.3815E-05&2.04&7.1681E-06&3.01&1.8248E-06&3.00&4.9414E-09\\
			&$128\times128$&122&1624321&2.3165E-05&2.02&8.9398E-07&3.00&2.2814E-07&3.00&6.0632E-09\\
			\Xhline{1pt}
	\end{tabular}}
	\caption{Results for $\nu=1.0e-3$ with $RT$ element on triangular meshes}
	\label{table4}
\end{table}	

\begin{table}
	\renewcommand{\arraystretch}{0.9}
	\centering
	\resizebox{\textwidth}{!}{
		\begin{tabular}{c|c|c|c|c|c|c|c|c|c|c}
			\Xhline{1pt}
			\multirow{2}{*}{$k$}&\multirow{2}{*}{Mesh}&\multirow{2}{*}{Iter}&\multirow{2}{*}{DOF}&\multicolumn{2}{c|}{$\frac{\|\mathbb{L}-\mathbb{L}_h\|_{\mathcal{T}_h}}{\|\mathbb{L}\|_{\mathcal{T}_h}}$}&\multicolumn{2}{c|}{$\frac{\|\bm u-\bm u_h\|_{\mathcal{T}_h}}{\|\bm u\|_{\mathcal{T}_h}}$}&\multicolumn{2}{c|}{$\frac{\|p-p_h\|_{\mathcal{T}_h}}{\|p\|_{\mathcal{T}_h}}$}&\multirow{2}{*}{$\|\nabla\cdot \bm u_h\|_{\mathcal{T}_h}$}\\
			\cline{5-10}
			& & &  &Error&Rate &Error&Rate&Error&Rate&\\
			\hline
			&$2\times2$&5&91&8.5225E-01&-&1.0301E+00&-&7.4725E-01&-&0.0000E+00\\
			&$4\times4$&15&323&5.3123E-01&0.68&2.6035E-01&1.98&4.4803E-01& 0.74&2.5124E-14\\
			&$8\times8$&24&1219&3.0403E-01&0.81&6.5810E-02&1.98&2.3631E-01 &0.92&6.4850E-14 \\
			1&$16\times16$&30&4739&1.6321E-01&0.90&1.6926E-02&1.96& 1.1980E-01&0.98&9.0067E-10\\
			&$32\times32$&30&18691&8.4393E-02&0.95&4.2789E-03&1.98&6.0111E-02&0.99&6.6416E-10\\
			&$64\times64$&31&74243&4.2886E-02&0.98&1.0733E-03&2.00&3.0081E-02&1.00&1.0473E-09 \\
			&$128\times128$&33&295939&2.1615E-02&0.99&2.6860E-04&2.00&1.5044E-02 &1.00&8.1687E-10\\
			\Xhline{1pt}
			&$2\times2$&14&243&4.3798E-01&-&2.8725E-01&-&2.1226E-01&-&0.0000E+00\\
			&$4\times4$&48&907&1.2805E-01&1.77&3.9618E-02&2.86&6.6096E-02&1.68&6.2653E-11\\
			&$8\times8$&68&3507&1.2805E-01&1.84&5.2795E-03&2.91&1.7471E-02&1.92&5.4972E-09\\
			2&$16\times16$&68&13795&9.7642E-03&1.87&6.7579E-04&2.97&4.4290E-03&1.98&9.0096E-09\\
			&$32\times32$&69&54723&2.6173E-03&1.90&8.3955E-05&3.01&1.1111E-03&1.99&7.5920E-09\\
			&$64\times64$&76&217987&6.8013E-04&1.94&1.0379E-05&3.02&2.7802E-04&2.00&4.0029E-09\\
			&$128\times128$&76&870147&1.7326E-04&1.97&1.2881E-06&3.01&6.9521E-05&2.00&2.2944E-09\\
			\Xhline{1pt}
	\end{tabular}}
	\caption{Results for $\nu=1.0e-3$ with $BDM2$ element on triangular meshes}
	\label{table5}
\end{table}	

\begin{table}
	\renewcommand{\arraystretch}{0.9}
	\centering
	\resizebox{\textwidth}{!}{
		\begin{tabular}{c|c|c|c|c|c|c|c|c|c|c}
			\Xhline{1pt}
			\multirow{2}{*}{$k$}&\multirow{2}{*}{Mesh}&\multirow{2}{*}{Iter}&\multirow{2}{*}{DOF}&\multicolumn{2}{c|}{$\frac{\|\mathbb{L}-\mathbb{L}_h\|_{\mathcal{T}_h}}{\|\mathbb{L}\|_{\mathcal{T}_h}}$}&\multicolumn{2}{c|}{$\frac{\|\bm u-\bm u_h\|_{\mathcal{T}_h}}{\|\bm u\|_{\mathcal{T}_h}}$}&\multicolumn{2}{c|}{$\frac{\|p-p_h\|_{\mathcal{T}_h}}{\|p\|_{\mathcal{T}_h}}$}&\multirow{2}{*}{$\|\nabla\cdot \bm u_h\|_{\mathcal{T}_h}$}\\
			\cline{5-10}
			& & &  &Error&Rate &Error&Rate&Error&Rate&\\
			\hline
			&$2\times2$&5&91&8.5225E-01&-&1.0301E+00&-&7.4764E-01&-&0.0000E+00\\
			&$4\times4$&16&323&5.3123E-01&0.68&2.6035E-01&1.98&4.4820E-01&0.74&1.0442E-15 \\
			&$8\times8$&25&1219&3.0403E-01&0.81 & 6.5810E-02 &1.98&2.3643E-01&0.92& 9.8384E-17 \\
			1&$16\times16$&31&4739&1.6321E-01&0.90&1.6926E-02&1.96&1.1987E-01&0.98&1.1844E-11\\
			&$32\times32$&35&18691&8.4393E-02 &0.95& 4.2789E-03&1.98&6.0152E-02&0.99&3.6107E-12\\
			&$64\times64$&38&74243&4.2886E-02&0.98 &1.0733E-03&2.00&3.0103E-02&1.00& 1.8807E-12\\
			&$128\times128$&41&295939&2.1615E-02&0.99&2.6860E-04&2.00&1.5055E-02&1.00&4.4443E-12\\
			\Xhline{1pt}
			&$2\times2$&15&243&4.3798E-01&-&2.8725E-01&-&2.1300E-01&-&0.0000E+00\\
			&$4\times4$&48&907&1.2805E-01&1.77&3.9618E-02&2.86 &6.6166E-02 &1.69&2.1572E-12\\
			&$8\times8$&62&3507&3.5670E-02&1.84&5.2795E-03&2.91&1.7480E-02&1.92&2.3932E-11\\
			2&$16\times16$&71&13795& 9.7642E-03&1.87&6.7579E-04&2.97&4.4303E-03&1.98&1.3434E-10\\
			&$32\times32$&73&54723&2.6173E-03&1.90&8.3955E-05&3.01&1.1113E-03&2.00&1.0707E-10\\
			&$64\times64$&75&217987&6.8013E-04&1.94&1.0379E-05&3.02 &2.7809E-04&2.00&1.3155E-10\\
			&$128\times128$&84&870147&1.7326E-04&1.97&1.2881E-06&3.01&6.9540E-05&2.00&2.6619E-10\\
			\Xhline{1pt}
	\end{tabular}}
	\caption{Results for $\nu=1.0$ with $BDM2$ element on triangular meshes}
	\label{table6}
\end{table}	

\begin{table}
	\renewcommand{\arraystretch}{0.9}
	\centering
	\resizebox{\textwidth}{!}{
		\begin{tabular}{c|c|c|c|c|c|c|c|c|c|c}
			\Xhline{1pt}
			\multirow{2}{*}{$k$}&\multirow{2}{*}{Mesh}&\multirow{2}{*}{Iter}&\multirow{2}{*}{DOF}&\multicolumn{2}{c|}{$\frac{\|\mathbb{L}-\mathbb{L}_h\|_{\mathcal{T}_h}}{\|\mathbb{L}\|_{\mathcal{T}_h}}$}&\multicolumn{2}{c|}{$\frac{\|\bm u-\bm u_h\|_{\mathcal{T}_h}}{\|\bm u\|_{\mathcal{T}_h}}$}&\multicolumn{2}{c|}{$\frac{\|p-p_h\|_{\mathcal{T}_h}}{\|p\|_{\mathcal{T}_h}}$}&\multirow{2}{*}{$\|\nabla\cdot \bm u_h\|_{\mathcal{T}_h}$}\\
			\cline{5-10}
			& & &  &Error&Rate &Error&Rate&Error&Rate&\\
			\hline
			&$2\times2$&13&187&6.5553E-01&-&5.8710E-01&-&2.1226E-01&-&0.0000E+00\\
			&$4\times4$&53&707&4.2349E-01&0.63&2.9587E-01&0.99&6.6096E-02&1.68&4.4009E-12\\
			&$8\times8$&67&2755&2.5869E-01 &0.71&1.0540E-01&1.49&1.7471E-02&1.92&4.0022E-09\\
			1&$16\times16$&69&10883&1.4278E-01 &0.86&2.9718E-02 &1.83&4.4290E-03 &1.98&5.8858E-09\\
			&$32\times32$&70&43267&7.4795E-02 &0.93&7.7837E-03&1.93&1.1111E-03 &1.99&6.2820E-09\\
			&$64\times64$&74&172547&3.8258E-02 &0.97&1.9872E-03&1.97&2.7802E-04&2.00 &3.4920E-09 \\
			&$128\times128$&73&689155&1.9346E-02&0.98&5.0187E-04&1.99&6.9523E-05&2.00 &1.8969E-09\\
			\Xhline{1pt}
			&$2\times2$&34&387&2.4838E-01&-&2.1572E-01&-&5.0933E-02 &-&0.0000E+00\\
			&$4\times4$&79&1483&9.9200E-02&1.32&3.2091E-02&2.75&8.1087E-03&2.65& 3.5993E-14\\
			&$8\times8$&116&5811&3.3413E-02&1.57&4.3395E-03&2.89&1.8803E-03&2.11&9.3827E-11\\
			2&$16\times16$&143&23011&9.7807E-03&1.77&5.5595E-04&2.96&5.9079E-04&1.67&4.9966E-11\\
			&$32\times32$&151&91587&2.6295E-03&1.90&7.0295E-05&2.98&1.6831E-04& 1.81&2.4422E-11\\
			&$64\times64$&167&365443&6.7864E-04&1.95&8.8565E-06&2.99& 4.4468E-05& 1.92&1.8224E-10\\
			&$128\times128$&175&1459971&1.7210E-04&1.98&1.1126E-06&2.99&1.1381E-05&1.97&8.1758E-11\\
			\Xhline{1pt}
	\end{tabular}}
	\caption{Results for $\nu=1.0e-3$ with $RT2$ element on triangular meshes}
	\label{table7}
\end{table}

\subsection{Three-Dimensional Numerical Results}
The body forces $\bm f$ are selected to ensure that the analytical solution to \Cref{Ori_problem2}, considering the homogeneous Dirichlet boundary condition, is given by
\begin{align*}
	u_1&=-200(x-x^2)^2(2y^3-3y^2+y)(2z^3-3z^2+z), \\ 
	u_2&=-100(y-y^2)^2(2x^3-3x^2+x)(2z^3-3z^2+z), \\
	u_3&=-100(z-z^2)^2(2y^3-3y^2+y)(2x^3-3x^2+x),\\ 
	p&=100(x^6+y^6-2z^6).
\end{align*}

In the subsequent numerical examples, we consider $\Omega=[0,1]\times[0,1]\times[0,1]$ with a constant viscosity of $\nu=1.00E+00$ and $1.00E-03$, where $\nu=Re^{-1}$. The computational mesh comprises a regular triangulation with $ 6\times n\times n\times n$ tetrahedra, denoted as an $n \times n \times n$ mesh. The results, presented in \Cref{table8}-\Cref{table11}, include  {`}Iter', which also represents the iterations required to solve the Schur complement system, while  {`}DOF' refers to the degrees of freedom within the entire linear system before static condensation.

\begingroup
\begin{table}
	\renewcommand{\arraystretch}{0.9}
	\centering
	\resizebox{1.0\textwidth}{!}{
		\begin{tabular}{c|c|c|c|c|c|c|c|c|c|c}
			\Xhline{1pt}
			\multirow{2}{*}{$k$}&\multirow{2}{*}{Mesh}&\multirow{2}{*}{Iter}&\multirow{2}{*}{DOF}&\multicolumn{2}{c|}{$\frac{\|\mathbb{L}-\mathbb{L}_h\|_{\mathcal{T}_h}}{\|\mathbb{L}\|_{\mathcal{T}_h}}$}&\multicolumn{2}{c|}{$\frac{\|\bm u-\bm u_h\|_{\mathcal{T}_h}}{\|\bm u\|_{\mathcal{T}_h}}$}&\multicolumn{2}{c|}{$\frac{\|p-p_h\|_{\mathcal{T}_h}}{\|p\|_{\mathcal{T}_h}}$}&\multirow{2}{*}{$\|\nabla\cdot \bm u_h\|_{\mathcal{T}_h}$}\\
			\cline{5-10}
			& & &  &Error&Rate &Error&Rate&Error&Rate&\\
			\hline
			&$2\times2\times2$&47&1200&8.2757E-01&-&1.1413E+00&-&5.8170E-01&-&1.7207E-09\\
			1&$4\times4\times4$&100&9024&5.3705E-01&0.62&3.9107E-01&1.55 &3.5086E-01&0.73&2.4429E-09\\
			&$8\times8\times8$&135&69888&2.8917E-01&0.89&1.0477E-01&1.90&1.8552E-01&0.92&1.3340E-09\\
			&$16\times16\times16$&150&549888&1.4753E-01&0.97&2.6813E-02 &1.97&9.4119E-02&0.98&6.0735E-10\\
			&$32\times32\times32$&171&4362240&7.4212E-02&0.99&6.7694E-03&1.99&4.7232E-03&0.99&2.4689E-10\\
			\Xhline{1pt}
			&$2\times2\times2$&162&4008&4.6570E-01&-&4.5183E-01&-&1.8075E-01&-&4.5717E-09\\
			2&$4\times4\times4$&174&30624&1.6556E-01&1.49&6.4831E-02&2.80&5.5424E-02&1.71&3.5457E-09\\
			&$8\times8\times8$&174&239232&4.5781E-02&1.85&8.3567E-03 &2.96&1.4601E-02&1.92 &2.6601E-09 \\
			&$16\times16\times16$&152&1890816&1.1787E-02&1.96&1.0488E-03&2.99&3.6985E-03&1.98&1.5847E-09\\
			\Xhline{1pt}
	\end{tabular}}
	\caption{Results for $\nu=1.0e-3$ with $BDM$ element on tetrahedral meshes}
	\label{table8}
\end{table}	

\begin{table}
	\renewcommand{\arraystretch}{0.9}
	\centering
	\resizebox{1.0\textwidth}{!}{
		\begin{tabular}{c|c|c|c|c|c|c|c|c|c|c}
			\Xhline{1pt}
			\multirow{2}{*}{$k$}&\multirow{2}{*}{Mesh}&\multirow{2}{*}{Iter}&\multirow{2}{*}{DOF}&\multicolumn{2}{c|}{$\frac{\|\mathbb{L}-\mathbb{L}_h\|_{\mathcal{T}_h}}{\|\mathbb{L}\|_{\mathcal{T}_h}}$}&\multicolumn{2}{c|}{$\frac{\|\bm u-\bm u_h\|_{\mathcal{T}_h}}{\|\bm u\|_{\mathcal{T}_h}}$}&\multicolumn{2}{c|}{$\frac{\|p-p_h\|_{\mathcal{T}_h}}{\|p\|_{\mathcal{T}_h}}$}&\multirow{2}{*}{$\|\nabla\cdot \bm u_h\|_{\mathcal{T}_h}$}\\
			\cline{5-10}
			& & &  &Error&Rate &Error&Rate&Error&Rate&\\
			\hline
			&$2\times2\times2$&44&1200&8.2757E-01&-&1.1413E+00 &-&5.8171E-01&-&1.8683E-11\\
			1&$4\times4\times4$&77&9024&5.3705E-01&0.62&3.9107E-01&1.55&3.5086E-01&0.73&2.8919E-11\\
			&$8\times8\times8$&95&69888&2.8917E-01&0.89&1.0477E-01&1.90&1.8552E-01&0.92&1.3721E-11\\
			&$16\times16\times16$&105&549888&1.4753E-01&0.97&2.6813E-02&1.97& 9.4119E-02&0.98&6.1961E-12\\
			&$32\times32\times32$&112&4362240&7.42122E-02&0.99&6.76942E-03&1.99& 4.72322E-02&0.99&2.1318E-12\\
			\Xhline{1pt}
			&$2\times2\times2$&98&4008&4.6570E-01&-&4.5183E-01&-&1.8075E-01&-&1.3483E-10 \\
			2&$4\times4\times4$&130&30624&1.6556E-01&1.49& 6.4831E-02&2.80&5.5425E-02&1.71&2.4215E-11\\
			&$8\times8\times8$&135&239232&4.5781E-02&1.85&8.3567E-03&2.96&1.4601E-02&1.92&1.1929E-11\\
			&$16\times16\times16$&139&1890816&1.1787E-02 &1.96&1.0488E-03&2.99&3.6986E-03&1.98& 4.3270E-12\\
			\Xhline{1pt}
	\end{tabular}}
	\caption{Results for $\nu=1.0$ with $BDM$ element on tetrahedral meshes}
	\label{table9}
\end{table}

\begin{table}
	\renewcommand{\arraystretch}{0.9}
	\centering
	\resizebox{1.0\textwidth}{!}{
		\begin{tabular}{c|c|c|c|c|c|c|c|c|c|c}
			\Xhline{1pt}
			\multirow{2}{*}{$k$}&\multirow{2}{*}{Mesh}&\multirow{2}{*}{Iter}&\multirow{2}{*}{DOF}&\multicolumn{2}{c|}{$\frac{\|\mathbb{L}-\mathbb{L}_h\|_{\mathcal{T}_h}}{\|\mathbb{L}\|_{\mathcal{T}_h}}$}&\multicolumn{2}{c|}{$\frac{\|\bm u-\bm u_h\|_{\mathcal{T}_h}}{\|\bm u\|_{\mathcal{T}_h}}$}&\multicolumn{2}{c|}{$\frac{\|p-p_h\|_{\mathcal{T}_h}}{\|p\|_{\mathcal{T}_h}}$}&\multirow{2}{*}{$\|\nabla\cdot \bm u_h\|_{\mathcal{T}_h}$}\\
			\cline{5-10}
			& & &  &Error&Rate &Error&Rate&Error&Rate&\\
			\hline
			&$2\times2\times2$&38&921&9.8498E-01&-&5.7605E-01&-&5.8170E-01&-&4.3601E-09\\
			1&$4\times4\times4$&65&6807& 8.1093E-01& 0.28&4.9310E-01&0.22&3.5086E-01&0.73&3.3376E-09\\
			&$8\times8\times8$&74&52491&5.0081E-01&0.70&2.3639E-01&1.06&1.8552E-01&0.92&1.4031E-09\\
			&$16\times16\times16$&76&412563&2.6730E-01&0.91&7.4923E-02&1.66&9.4119E-02&0.98 &8.6188E-10\\
			&$32\times32\times32$&76&3271971&1.3587E-01&0.98&2.0072E-02 &1.90&4.7232E-02 &0.99 &3.7277E-10\\
			\Xhline{1pt}
			&$2\times2\times2$&134&3303&6.3074E-01&-&3.4050E-01&-&1.8075E-01&-& 6.7592E-09\\
			2&$4\times4\times4$&139&25035&2.6757E-01&1.24&8.9550E-02&1.93&5.5424E-02&1.71&5.6031E-09\\
			&$8\times8\times8$&134&195219&7.8583E-02&1.77&1.1873E-02&2.91 &1.4601E-02&1.92&3.4693E-09\\
			&$16\times16\times16$&136&1542435&2.1535E-02&1.87 &1.3767E-03&3.11&3.6985E-03&1.98&1.6102E-09\\
			\Xhline{1pt}
	\end{tabular}}
	\caption{Results for $\nu=1.0e-3$ with $BDM2$ element on tetrahedral meshes}
	\label{table10}
\end{table}

\begin{table}
	\renewcommand{\arraystretch}{0.9}
	\centering
	\resizebox{1.0\textwidth}{!}{
		\begin{tabular}{c|c|c|c|c|c|c|c|c|c|c}
			\Xhline{1pt}
			\multirow{2}{*}{$k$}&\multirow{2}{*}{Mesh}&\multirow{2}{*}{Iter}&\multirow{2}{*}{DOF}&\multicolumn{2}{c|}{$\frac{\|\mathbb{L}-\mathbb{L}_h\|_{\mathcal{T}_h}}{\|\mathbb{L}\|_{\mathcal{T}_h}}$}&\multicolumn{2}{c|}{$\frac{\|\bm u-\bm u_h\|_{\mathcal{T}_h}}{\|\bm u\|_{\mathcal{T}_h}}$}&\multicolumn{2}{c|}{$\frac{\|p-p_h\|_{\mathcal{T}_h}}{\|p\|_{\mathcal{T}_h}}$}&\multirow{2}{*}{$\|\nabla\cdot \bm u_h\|_{\mathcal{T}_h}$}\\
			\cline{5-10}
			& & &  &Error&Rate &Error&Rate&Error&Rate&\\
			\hline
			&$2\times2\times2$&38&921&9.8498E-01&-&5.7605E-01&-&5.8171E-01&-&4.3372E-12\\
			1&$4\times4\times4$&61&6807&8.1093E-01&0.28&4.9310E-01&0.22&3.5087E-01 &0.73 &8.5161E-12\\
			&$8\times8\times8$&70&52491&5.0081E-01&0.70&2.3639E-01&1.06&1.8553E-01&0.92&3.2946E-12\\
			&$16\times16\times16$&73&412563 &2.6730E-01&0.91&7.4923E-02&1.66&9.4124E-02& 0.98&1.5268E-12 \\
			&$32\times32\times32$&74&3271971&1.35868E-01&0.98&2.00720E-02&1.90&4.72346E-02&0.99&1.5268E-12 \\
			\Xhline{1pt}
			&$2\times2\times2$&98&3303& 6.3074E-01&-&3.4050E-01& -&1.8076E-01&-&6.2362E-11\\
			2&$4\times4\times4$&125&25035&2.6757E-01&1.24&8.9550E-02&1.93&5.5430E-02&1.71&1.5837E-11\\
			&$8\times8\times8$&130&195219&7.8583E-02&1.77&1.1873E-02 &2.91&1.4603E-02&1.92& 5.5316E-12\\
			&$16\times16\times16$&133&1542435& 2.1535E-02&1.87&1.3767E-03& 3.11& 3.6994E-03&1.98&2.6960E-12\\
			\Xhline{1pt}
	\end{tabular}}
	\caption{Results for $\nu=1.0$ with $BDM2$ element on tetrahedral meshes}
	\label{table11}
\end{table}	

\subsection{Tangential boundary control problem}
The domain is the unit square $\Omega=(0,1)^2$ and the data is chosen as
$$
\begin{aligned}
	& y_1=-2 \pi^2 \sin ^2\left(\pi x_1\right) \cos \left(\pi x_2\right)-2 \pi^2 \sin \left(\pi x_1\right) \sin \left(2 \pi x_2\right) {,} \\
	& y_2=2 \pi^2 \cos \left(\pi x_1\right) \sin ^2\left(\pi x_2\right)+2 \pi^2 \sin \left(\pi x_2\right) \sin \left(2 \pi x_1\right)  {,}\\
	& z_1=\pi \sin ^2\left(\pi x_1\right) \sin \left(2 \pi x_2\right), \quad z_2=-\pi \sin ^2\left(\pi x_2\right) \sin \left(2 \pi x_1\right), \\
	& p=10^nx \text{ if }x< 0.5, \quad p=-10^n(x-0.5) \text{ if }x\ge 0.5,\\
	& w=10^ny \text{ if }y< 0.5, \quad w=-10^n(y-0.5) \text{ if }y\ge 0.5, \quad \gamma=1 .
\end{aligned}
$$
Here $n$ is a parameter. We take $n=3$ and $7$. The results are presented as below, where 
\begin{align*}
	\operatorname{div}\left(\boldsymbol{y}_h\right)=\max _{K \in \mathcal{T}_h} \frac{1}{|K|} \int_K\left|\nabla \cdot \boldsymbol{y}_h\right| \mathrm{d} \boldsymbol{x} {.}
\end{align*}
These results confirm the pressure robustness of the proposed method.
\endgroup

\begin{table}
	\caption{Pressure-robustness:  Errors and observed convergence orders for the control $u$, state $\bm y$,  and its flux $\mathbb L$.}
	\label{Table1}
	\centering
	\begin{tabular}{c|c|c|c|c|c|c|c|c|c}
		\Xhline{1pt}
		
		\multirow{2}{*}{$k$} &
		\multirow{2}{*}{$n$} &
		\multirow{2}{*}{$\frac{\sqrt{2}}{h}$} &
		\multirow{2}{*}{$\textup{div}(\bm y_h)$}&
		
		\multicolumn{2}{c|}{$\|\bm y-\bm y_h\|_{L^2(\mathcal T_h)}$} &

		\multicolumn{2}{c|}{$\|\mathbb L-\mathbb L_h\|_{L^2(\mathcal T_h)}$}&

		\multicolumn{2}{c}{$\|u-u_h\|_{L^2(\mathcal E_h^{\partial})}$}
		\\
		
		\cline{5-10}
		& 	& 	& &Error &Rate  &Error &Rate &Error &Rate    \\
		
		\Xhline{1pt}
		
		&		&	4	&	8.8818E-16	&	8.7605E+00	&		&	5.3280E+01	&		&	6.6617E+00	&		\\
		&		&	8	&	6.6613E-16	&	2.1975E+00	&	2.00 	&	2.7916E+01	&	0.93 	&	3.3964E+00	&	0.97 	\\
		1	&	3	&	16	&	4.4409E-16	&	5.4087E-01	&	2.02 	&	1.4101E+01	&	0.99 	&	1.6509E+00	&	1.04 	\\
		&		&	32	&	3.0531E-16	&	1.3407E-01	&	2.01 	&	7.0468E+00	&	1.00 	&	8.0570E-01	&	1.03 	\\
		&		&	64	&	1.8041E-16	&	3.3412E-02	&	2.00 	&	3.5201E+00	&	1.00 	&	3.9775E-01	&	1.02 	\\
		\Xhline{1pt}
		&		&	4	&	1.7764E-15	&	8.7605E+00	&		&	5.3280E+01	&		&	6.6617E+00	&		\\
		&		&	8	&	8.8818E-16	&	2.1975E+00	&	2.00 	&	2.7916E+01	&	0.93 	&	3.4154E+00	&	0.96 	\\
		1	&	7	&	16	&	4.4409E-16	&	5.4087E-01	&	2.02 	&	1.4101E+01	&	0.99 	&	1.6604E+00	&	1.04 	\\
		&		&	32	&	3.0531E-16	&	1.3407E-01	&	2.01 	&	7.0468E+00	&	1.00 	&	7.9634E-01	&	1.06 	\\
		&		&	64	&	1.8735E-16	&	3.3412E-02	&	2.00 	&	3.5201E+00	&	1.00 	&	3.9208E-01	&	1.02 	\\
		
		\Xhline{1pt}
		&		&	4	&	1.6871E-15	&	1.1810E+00	&		&	1.4040E+01	&		&	1.5526E+00	&		\\
		&		&	8	&	1.7415E-15	&	1.5197E-01	&	2.96 	&	3.7803E+00	&	1.89 	&	4.3691E-01	&	1.83 	\\
		2	&	3	&	16	&	1.0398E-15	&	1.9367E-02	&	2.97 	&	1.0278E+00	&	1.88 	&	1.0731E-01	&	2.03 	\\
		&		&	32	&	9.0076E-16	&	2.4535E-03	&	2.98 	&	2.9129E-01	&	1.82 	&	2.7875E-02	&	1.94 	\\
		&		&	64	&	4.6434E-16	&	3.1170E-04	&	2.98 	&	8.7294E-02	&	1.74 	&	6.9773E-03	&	2.00 	\\
		
		\Xhline{1pt}
		&		&	4	&	1.8272E-15	&	1.1810E+00	&		&	1.4040E+01	&		&	1.6163E+00	&		\\
		&		&	8	&	1.7192E-15	&	1.5197E-01	&	2.96 	&	3.7803E+00	&	1.89 	&	4.2119E-01	&	1.94 	\\
		2	&	7	&	16	&	1.0735E-15	&	1.9367E-02	&	2.97 	&	1.0278E+00	&	1.88 	&	1.1099E-01	&	1.92 	\\
		&		&	32	&	9.0188E-16	&	2.4535E-03	&	2.98 	&	2.9129E-01	&	1.82 	&	2.7413E-02	&	2.02 	\\
		&		&	64	&	4.7971E-16	&	3.1170E-04	&	2.98 	&	8.7295E-02	&	1.74 	&	6.9481E-03	&	1.98 	\\
		
		\Xhline{1pt}
	\end{tabular}
\end{table}

\begin{table}
	\caption{Pressure-robustness: Errors and observed convergence orders for the dual state $\bm z$,  and its flux $\mathbb T$.}
	\label{Table2}
	\centering
	\begin{tabular}{c|c|c|c|c|c|c|c}
		\Xhline{1pt}
		
		\multirow{2}{*}{$k$} &
		\multirow{2}{*}{$n$} &
		\multirow{2}{*}{$\frac{\sqrt{2}}{h}$} &
		\multirow{2}{*}{$\textup{div}(\bm z_h)$}&
		
		\multicolumn{2}{c|}{$\|\bm z-\bm z_h\|_{L^2(\mathcal T_h)}$} &

		\multicolumn{2}{c}{$\|\mathbb T-\mathbb T_h\|_{L^2(\mathcal T_h)}$}

		\\
		
		\cline{5-8}
		& & 	& &Error &Rate  &Error &Rate    \\

		\Xhline{1pt}
		&		&	4	&	1.1102E-16	&	8.5118E-01	&		&	6.2681E+00	&		\\
		&		&	8	&	5.5510E-17	&	2.4732E-01	&	1.78 	&	3.3656E+00	&	0.90 	\\
		1	&	3	&	16	&	3.6210E-17	&	6.5538E-02	&	1.92 	&	1.7145E+00	&	0.97 	\\
		&		&	32	&	2.4210E-17	&	1.6797E-02	&	1.96 	&	8.6015E-01	&	1.00 	\\
		&		&	64	&	1.5610E-17	&	4.2445E-03	&	1.98 	&	4.3026E-01	&	1.00 	\\

		\Xhline{1pt}
		&		&	4	&	1.1102E-16	&	8.5118E-01	&		&	6.2681E+00	&		\\
		&		&	8	&	5.5510E-17	&	2.4732E-01	&	1.78 	&	3.3656E+00	&	0.90 	\\
		1	&	7	&	16	&	3.4690E-17	&	6.5538E-02	&	1.92 	&	1.7145E+00	&	0.97 	\\
		&		&	32	&	2.2880E-17	&	1.6797E-02	&	1.96 	&	8.6015E-01	&	1.00 	\\
		&		&	64	&	1.3010E-17	&	4.2445E-03	&	1.98 	&	4.3026E-01	&	1.00 	\\

		\Xhline{1pt}
		
		&		&	4	&	1.8332E-16	&	1.6150E-01	&		&	1.9279E+00	&		\\
		&		&	8	&	1.3189E-16	&	2.1210E-02	&	2.93 	&	5.0620E-01	&	1.93 	\\
		2	&	3	&	16	&	8.9800E-17	&	2.7042E-03	&	2.97 	&	1.2776E-01	&	1.99 	\\
		&		&	32	&	6.5330E-17	&	3.4078E-04	&	2.99 	&	3.2010E-02	&	2.00 	\\
		&		&	64	&	3.9000E-17	&	4.2742E-05	&	3.00 	&	8.0066E-03	&	2.00 	\\
		
		\Xhline{1pt}
		
		&		&	4	&	1.4317E-16	&	1.6150E-01	&		&	1.9279E+00	&		\\
		&		&	8	&	1.3096E-16	&	2.1210E-02	&	2.93 	&	5.0620E-01	&	1.93 	\\
		2	&	7	&	16	&	8.9600E-17	&	2.7042E-03	&	2.97 	&	1.2776E-01	&	1.99 	\\
		&		&	32	&	6.4560E-17	&	3.4078E-04	&	2.99 	&	3.2010E-02	&	2.00 	\\
		&		&	64	&	3.8520E-17	&	4.2742E-05	&	3.00 	&	8.0066E-03	&	2.00 	\\

		\Xhline{1pt}
	\end{tabular}
\end{table}

\begingroup
\section{Conclusion}
We developed a family of $H(\mathrm{div})$-conforming HDG methods for the steady Stokes equations and showed that the discrete velocities are exactly divergence-free. For the BDM variants we obtained the strongest theory, including optimal energy-norm and $L^2$-velocity estimates, while for the RT variants we proved optimal velocity convergence together with weaker pressure estimates. The analysis removes the $H^1$-regularity assumption on the pressure that appeared in our earlier pressure-robust control work. We also established a spectral-equivalence result for the Schur complement and applied the BDM discontinuous-trace method to tangential boundary control. Future work includes a sharper RT-pressure analysis, a more systematic treatment of the continuous-trace variants, and extensions to Navier--Stokes flows and more realistic three-dimensional benchmarks.
\endgroup

\section*{Acknowledgements}
G.\ Chen is supported by National natural science Foundation of China (NSFC) under grant number 11801063 and number 12171341.
The research of Y. Zhang is partially supported by  the US National Science Foundation (NSF) under grant number DMS-2111315.

\section*{Conflict of Interest} The authors declare that they have no conflict of interest.

\section*{Data Availability} The datasets generated during and/or analyzed during the current study are available from the corresponding author on reasonable request.

\section*{Code Availability} The codes during the current study are available from the corresponding author on reasonable request.

\end{document}